\documentclass[12pt,draftcls,onecolumn]{IEEEtran}
\usepackage{amsmath,amssymb,amsfonts}
\usepackage{algorithmic}
\usepackage{graphicx}
\usepackage{algorithm,algorithmic}
\usepackage{hyperref}
\hypersetup{hidelinks=true}
\usepackage{textcomp}
\usepackage{dsfont}
\usepackage[maxfloats=100]{morefloats}[2015/07/22]
\usepackage{multirow}
\usepackage{wrapfig}
\usepackage{caption}
\usepackage{booktabs}       
\usepackage{nicefrac}       
\usepackage{microtype}      
\usepackage{enumerate}
\usepackage{mathrsfs}
\usepackage{lipsum}
\usepackage{mathtools}
\usepackage{float}
\usepackage{bbm}
\usepackage{varioref}
\usepackage{hyperref}
\usepackage{orcidlink}

\usepackage{rotating}
\usepackage[normalem]{ulem}
\usepackage{xcolor}
\def\argmin{\mathop{\rm argmin}}
\DeclareUnicodeCharacter{00A0}{~}
\newtheorem{assumption}{Assumption}
\newtheorem{theorem}{Theorem}
\newtheorem{lemma}{Lemma}
\newtheorem{proposition}{Proposition}
\newtheorem{definition}{Definition}

\newtheorem{remark}{Remark}
\newtheorem{proof}{Proof}
\newcommand{\EXP}[1]{\mathsf{E}\!\left[#1\right]}
\definecolor{ao}{rgb}{0.0, 0.5, 0.0}

\newcommand{\bx}{\mathbf{x}}
\newcommand{\by}{\mathbf{y}}
\newcommand{\bfun}{\mathbf{f}}

\newcommand{\bunit}{\mathbf{1}}
\newcommand{\bW}{\mathbf{W}}

\newcommand{\bG}{\mathbf{G}}

\newcommand{\bxi}{\boldsymbol{\xi}}
\newcommand{\txsum}{\textstyle\sum}
\newcommand{\sF}{\mathscr{F}}

\newcommand{\fy}[1]{\begin{color}{black}#1\end{color}}
\newcommand{\yq}[1]{\begin{color}{black}#1\end{color}}
\newcommand{\bz}[1]{\begin{color}{black}#1\end{color}}
\newcommand{\yqp}[1]{\begin{color}{black}#1\end{color}} 
\newcommand{\yqo}[1]{\begin{color}{black}#1\end{color}}
\newcommand{\yqt}[1]{\begin{color}{black}#1\end{color}}
\usepackage[mathscr]{euscript}

\begin{document}
\sloppy
\title{Iteratively Regularized Gradient Tracking Methods for Optimal Equilibrium Seeking}
\author{Yuyang Qiu\,\orcidlink{0009-0002-6470-9708}, Farzad Yousefian\,\orcidlink{0000-0003-2628-741X}, and Brian Zhang\,\orcidlink{0000-0002-6601-7361} 
\thanks{
This work is supported in part by the National Science Foundation under CAREER Grant ECCS-1944500,
in part by the Office of Naval Research under Grant N00014-22-1-2757, and in part by the Department of Energy under Grant DE-SC0023303.}
\thanks{The authors are affiliated with Rutgers University, Piscataway, NJ 08854, USA. They are contactable at yuyang.qiu@rutgers.edu, farzad.yousefian@rutgers.edu, and blz19@scarletmail.rutgers.edu, respectively. A substantially preliminary version of this work appeared in the \yqt{proceedings} of the 2021 American Control Conference~\cite{yousefian2021bilevel}.}}

\maketitle

\begin{abstract}
In noncooperative Nash games, equilibria are often inefficient. This is exemplified by the Prisoner's Dilemma and was first provably shown in the 1980s. Since then, understanding the quality of Nash equilibrium (NE) received considerable attention, leading to the emergence of inefficiency measures characterized by the best or the worst equilibrium. Traditionally, computing an optimal NE in monotone regimes is done through two-loop schemes which lack scalability and provable performance guarantees. The goal in this work lies in the development of among the first single-timescale distributed gradient tracking optimization methods for optimal NE seeking over networks. Our main contributions are as follows. By employing a regularization-based relaxation approach within two existing distributed gradient tracking methods, namely Push-Pull and DSGT, we devise and analyze two single-timescale iteratively regularized gradient tracking algorithms. The first method addresses computing the optimal NE over directed networks, while the second method addresses a stochastic variant of this problem over undirected networks. For both methods, we establish the convergence to the optimal NE and derive new convergence rate statements for the consensus error of the generated iterates. We provide preliminary numerical results on a Nash-Cournot game.
\end{abstract}

\section{Introduction}
\label{sec:introduction}
Noncooperative game theory~\cite{nash1951non} offers a rigorous mathematical framework for the modeling and analysis of multi-agent decision-making systems. It has been applied in a myriad of applications including transportation systems~\cite{wardrop1952road}, communication~\cite{comNE2016}, telecommunication networks~\cite{altman98}, smart grids~\cite{smartg_21}, social networks~\cite{GHADERI20143209}, economics~\cite{dockner2000differential}, and robotics~\cite{chung2011search}, among others. These applications share the following key characteristic in common; a collection of independent and self-interested agents (i.e., entities such as \fy{sensors}, people, robots, for example) compete with each other where each agent seeks to optimize an individual objective function. The result of this noncooperative rationality is mathematically captured by the concept of Nash equilibrium (NE). In view of the agents' selfish behavior, it is natural to expect the underlying system to reach a non-optimal global performance. In fact, equilibria are known to be inefficient. This is exemplified by the Prisoner's Dilemma and was first provably shown by Dubey~\cite{Dupey86}. Since then, understanding the quality of NE received considerable attention in addressing questions such as: {\it Is there a unique equilibrium? If not, how many equilibria exist? What is the most desirable equilibrium of a game?} Addressing these questions is critically important from the standpoint of the game's protocol designer in areas such as routing~\cite{roughgarden2002bad} and load balancing~\cite{correa2004selfish}. Indeed, in many cases a multitude of equilibria exist~\cite{papadimitriou2001algorithms}, and as such, finding a single equilibrium could be far from satisfactory. Research efforts in tackling these questions led to the emergence of inefficiency measures \fy{characterized by the best or the worst equilibrium with respect to a global welfare function}. 
In  the game theory literature, the problem of \fy{computing an optimal} equilibrium has been regarded as a computationally challenging task~\cite{complexityNE08,hazan11}. This is primarily because the \fy{optimal NE seeking problem} is cast as an optimization problem with equilibrium constraints. Naturally in optimization, the presence of constraints is tackled by leveraging a handful of avenues including the celebrated Lagrangian duality theory, projected schemes, and penalty techniques (cf.~\cite{Bertsekas2016}). However, when the constraint set is characterized as the set of equilibria, which is indeed the case in \fy{optimal NE seeking}, these standard approaches may not be applicable. Our focus in this paper lies in addressing the following key research question: {\it Can we devise single-timescale distributed first-order methods (over both undirected and directed networks) for \fy{optimal NE seeking}, in settings possibly afflicted with the presence of uncertainty, with provable global guarantees?} Here, the term {\it single-timescale} refers to a method that does not involve solving an inner-level optimization/variational problem and the term {\it global} refers to computing a global optimal solution to the NE selection optimization problem. 
\fy{To this end, in this work, we aim to address a distributed optimization problem with a distributed variational inequality (VI) constraint. This is of the form 
\begin{align}\label{eqn:bilevel_problem}
	&\min_{x \in \mathbb{R}^n}\  \textstyle\sum_{i=1}^m f_i(x) \quad \hbox{s.t.} \quad  x \in \text{SOL} (X, \textstyle\sum_{i=1}^m F_i),
\end{align}
where there are $m$ agents, each agent is associated with a local function $f_i$ and a mapping $F_i:\mathbb{R}^n\to\mathbb{R}^n$. Agents can communicate over a directed or an undirected network. Here, for a given nonempty convex set $X \yqt{\subseteq} \mathbb{R}^n$ and mapping $F$, $\mbox{SOL}(X,F) \triangleq \{x \in X \mid F(x)^{\top}(y-x)\geq 0, \forall y \in X\}$, denotes the solution set of $\mbox{VI}(X,F)$. The VI problem~\cite{facchinei02finite} is an immensely powerful mathematical framework for capturing several important problem classes. In particular, under standard differentiability and convexity assumptions~\cite[Prop. 1.4.2]{facchinei02finite}, the set of all NEs among a group of players is equal to the solution set of a Cartesian VI (see section~\ref{sec:num} for a detailed example of how a Nash game can be reformulated as a distributed VI problem). We note that solving problem~\eqref{eqn:bilevel_problem} is challenging, because of the following reasons: (i) The feasible solution set in \eqref{eqn:bilevel_problem} is itself the solution set of a VI problem. Therefore, it is often unavailable. In fact, there does not exist an algorithm that can compute all Nash equilibria~\cite{feinstein2023characterizing}; (ii) The set $\mbox{SOL}(X,F)$ is characterized by infinitely many, and possibly nonconvex, inequality constraints. Therefore, traditional Lagrangian duality theory may not be applicable in this setting. The classical notion of duality for VIs finds its origin in the work by Mosco~\cite{mosco1972dual} and its extensions were studied in~\cite{LDM_VI_2000}. Such duality frameworks are focused on the resolution of VIs in the presence of standard inequalities within the set $X$, which is substantially different than the setting considered in this work.} 
\subsection{Related \fy{work}}
The traditional approach for computing an optimal NE lies in {\it sequential regularization (SR)}, described as follows. At \fy{time step} $k$, given a strictly positive regularization parameter $\lambda_k$, convex set $X$, strongly convex function $f$, and monotone mapping $\fy{F}$, the regularized problem $\text{VI}(X,\fy{F}+\lambda_k\nabla f)$ is solved. Let $x^*_{\lambda_k}$ denote the unique solution of this problem. When $\lambda_k \to 0$, it can be shown that the solution trajectory of $\{x^*_{\lambda_k}\}$, known as {\it Tikhonov trajectory}, converges to the unique optimal NE. 
Although it is theoretically appealing, a key shortcoming of the SR scheme is that it is a two-loop scheme where at each iteration, an increasingly more difficult VI problem needs to be solved. This computational burden is exacerbated further by stochasticity in the objective function and the equilibrium constraints where it may be impossible to compute an exact solution to a regularized VI problem. Further, the computational complexity of the SR scheme appears to be unknown. \fy{To address these shortcomings, inexact SR schemes~\cite{facchinei2014vi} and iteratively regularized (IR) gradient methods~\cite{koshal12regularized,kannan2012distributed,yousefian2017smoothing,lei2023distributed,yousefian2021bilevel,staudigl2024tikhonov} have been studied more recently. In prior work by the authors~\cite{kaushik2021method,samadi2023improved}, both asymptotic and non-asymptotic guarantees of IR schemes for addressing centralized VI-constrained optimization problems are studied. Notably, in~\cite{samadi2023improved} it is shown that IR schemes may admit the fastest known convergence speed in addressing VI-constrained optimization problems and their subclasses~\cite{samadi2024achieving}. In~\cite{jalilzadeh2024stochastic}, an iteratively penalized scheme is devised to address stochastic settings. In~\cite{kaushik2023incremental}, an IR incremental gradient method is devised for addressing problem \eqref{eqn:bilevel_problem}, where it is assumed that agents communicate over a \yqt{directed \fy{cycle} graph}. 
 
 Another avenue for addressing the optimal NE selection problem lies in leveraging the fixed-point
selection theory that has been employed in centralized~\cite{yamada2005hybrid} and distributed settings~\cite{benenati2022optimal,benenati2023optimal}. Of these, the work in~\cite{benenati2023optimal} proposes a distributed method with asymptotic convergence guarantees for tracking the optimal NE in deterministic settings, over undirected and time-varying networks.} 

\subsection{Research gap and contributions}
\fy{Despite recent progress in addressing optimal NE seeking problems, it appears that  problem~\eqref{eqn:bilevel_problem} over networks has not been addressed. More precisely, we are unaware of IR gradient tracking (GT) schemes for optimal NE seeking in that the welfare loss objective is distributed among the agents, as in~\eqref{eqn:bilevel_problem}. 
In recent years, distributed GT methods have been designed and analyzed for addressing the canonical problem of the form $\min_{x \in X} \ \sum_{i=1}^m f_i(x)$ in convex~\cite{pu2020push,song2023optimal}, nonconvex~\cite{xin2021improved}, and stochastic settings~\cite{pu2021distributed}, among others. Motivated by the extensive recent work on GT methods for addressing standard distributed optimization problems, we devise and analyze two new iteratively regularized distributed GT methods equipped with provable guarantees. Our main contributions are as follows.

\noindent {\bf (i)} In the first part of the paper, we consider problem~\eqref{eqn:bilevel_problem} over directed networks and assume that the global mapping is monotone and the global objective is strongly convex. We consider the Push-Pull method in~\cite{pu2021distributed} that addresses unconstrained distributed optimization problems over directed networks. By leveraging the IR framework in the Push-Pull method, we devise an iteratively regularized method called IR-Push-Pull for addressing~\eqref{eqn:bilevel_problem}, where at each iteration, the information of iteratively regularized local mappings is pushed to the neighbors, while the information about the local copies of vector $x$ is pulled from the neighbors. In Theorem~\ref{thm:rate_ana}, we establish the convergence to the unique optimal NE and derive new convergence rate statements for the consensus errors of the generated iterates. 

\noindent {\bf (ii)} In the second part of the paper, we consider a stochastic variant of problem~\eqref{eqn:bilevel_problem} over undirected networks given as 
	\begin{equation}\label{eqn:bilevel_problem_stoch}
	\begin{split}
		&\min_{x \in \mathbb{R}^n}\  \textstyle\sum\nolimits_{i=1}^m \EXP{f_i(x,\xi_i)}\\
		&\hbox{s.t.} \quad  x \in \text{SOL}(\mathbb{R}^n,  \textstyle\sum\nolimits_{i=1}^m \EXP{F_i(x,\xi_i)}),
	\end{split}
	\end{equation}
where $\xi_i \in \mathbb{R}^d$ is a local random variable associated with agent $i$. We consider the DSGT method in~\cite{pu2021distributed} that addresses unconstrained distributed stochastic optimization problems over undirected networks. Again, by leveraging the IR framework, we devise a single-timescale method called IR-DSGT for addressing~\eqref{eqn:bilevel_problem_stoch}. In Theorem~\ref{thm:rate}, we establish the convergence to the unique optimal NE in a mean-squared sense and derive new non-asymptotic convergence rate statement and error bounds for mean-squared of consensus errors.}

\yqt{\noindent {\bf Outline of the paper.} We organize the remaining sections as follows\fy{.} In \fy{section~\ref{sec:alg}}, we present \fy{and analyze the IR-Push-Pull method} for addressing problem \eqref{eqn:bilevel_problem} over directed networks and provide convergence guarantees for computing the optimal NE \fy{in Theorem~\ref{thm:rate_ana}}. In \fy{section~\ref{sec:stoch}}, we present \fy{and analyze the IR-DSGT method for} addressing the stochastic problem \eqref{eqn:bilevel_problem_stoch}, over undirected networks and provide the convergence guarantees in \fy{Theorem~\ref{thm:rate}}. In section~\ref{sec:num}, we validate our theoretical findings through numerical experiments on \fy{a} Nash-Cournot game. \fy{Concluding remarks are presented in section~\ref{sec:conc}.}}

 \noindent {\bf Notation.} For an integer $m$, the set $\{1,\ldots,m\}$ is denoted as $[m]$. A vector $x$ is assumed to be a column vector (unless otherwise noted) 
and $x^{\top}$ denotes its transpose. We use $\|x\|_2$ to denote the Euclidean vector norm of $x$. A continuously differentiable function 
$f: \mathbb{R}^n\rightarrow \mathbb{R}$ is said to be $\mu_f$--strongly convex if and only if its gradient mapping is $\mu_f$--strongly monotone, 
i.e.,  $\left(\nabla f(x) -\nabla f(y)\right)^{\top}(x-y)\geq \mu_f\|x-y\|_2^2$ for any $x,y \in \mathbb{R}^n$. 
Also, it is said to be $L_f$--smooth if its gradient mapping is Lipschitz continuous with parameter $L_f>0$, i.e., 
for any $x, y \in\mathbb{R}^n$, we have $\|\nabla f(x)-\nabla f(y)\|_2\leq L_f\|x-y\|_2$. \yq{A mapping $F$ is $L_F$--Lipschitz continuous 
if for any $x,y\in \mathbb{R}^n$, we have $\|F(x)- \fy{F}(y)\|_2\leq L_F\|x-y\|_2$; and $F$ is merely monotone if 
$\left(F(x) -F(y)\right)^{\top}(x-y)\geq 0$ for any $x,y \in \mathbb{R}^n$.} We use the following definitions:
\begin{align}
&\mathbf{x} \triangleq [x_1,\ \ldots,\ x_m]^{\top} , \quad \mathbf{y} \triangleq [y_1,\ \ldots,\ y_m]^{\top}\in \mathbb{R}^{m\times n}\label{eqn:notation_eqn1}\\
&{f}(x)  \triangleq  \textstyle\sum\nolimits_{i=1}^m f_i(x), \quad \mathbf{f}(\mathbf{x})  \triangleq  \textstyle\sum\nolimits_{i=1}^m f_i(x_i), \label{eqn:notation_eqn2}\\
&\nabla \mathbf{f}(\mathbf{x}) \triangleq [\nabla f_1(x_1),\ \ldots,\ \nabla f_m(x_m)]^{\top}  \in \mathbb{R}^{m\times n},\label{eqn:notation_eqn3}\\
&\yq{F(x)  \triangleq  \textstyle\sum\nolimits_{i=1}^m F_i(x),}\label{eqn:notation_eqn4}\\
&\bz{\mathbf{F}(\mathbf{x}) \triangleq [F_1(x_1),\ \ldots,\ F_m(x_m)]^{\top}  \in \mathbb{R}^{m\times n}}.\label{eqn:notation_eqn5}
\end{align}
Here, $x_i$ denotes the local copy of the decision vector for agent $i$, and $\mathbf{x}$ includes the local copies of all agents. Vector $y_i$ denotes the auxiliary variable for agent $i$ to track the average of regularized gradient mappings.


\section{Deterministic setting over directed networks}\label{sec:alg}
\fy{In this section, we consider problem~\eqref{eqn:bilevel_problem} over directed networks, under the following main assumption. 
\begin{assumption}\label{assum:problem}
	\noindent (a) Function $f:\mathbb{R}^n \to \mathbb{R}$ is $\mu_f$--strongly convex. (b) Functions $f_i:\mathbb{R}^n \to \mathbb{R}$ are $L_f$--smooth. (c) Mapping $F:\mathbb{R}^n \to \mathbb{R}^n$ is real-valued and (merely) monotone. (d) Mappings $F_i:\mathbb{R}^n \to \mathbb{R}^n$ are $L_F$--Lipschitz continuous.
\end{assumption}
Throughout this section, we utilize the following notation and preliminaries.}
Given a set of nodes $\mathcal{N}$, a directed graph (digraph) is denoted by $\mathcal{G} =\left(\mathcal{N},\mathcal{E}\right)$ where 
$\mathcal{E} \subseteq \mathcal{N}\times \mathcal{N}$ is the set of ordered pairs of vertices. For any edge $(i,j)\in\mathcal{E}$, $i$ 
and $j$ are called parent node and child node, respectively. Graph $\mathcal{G}$ is called {\it strongly connected} if there is a path 
between the pair of any two different vertices. The digraph induced by a given nonnegative matrix $\mathbf{B}\in \mathbb{R}^{m\times m}$ 
is denoted by $\mathcal{G}_{\mathbf{B}} \triangleq \left(\mathcal{N}_{\mathbf{B}},\mathcal{E}_{\mathbf{B}}\right)$, where $\mathcal{N}_{\mathbf{B}} 
\triangleq [m]$ and $(j,i) \in \mathcal{E}_{\mathbf{B}}$ if and only if $B_{ij}>0$. We let $\mathcal{N}_{\mathbf{B}}^{\text{in}}(i)$ and 
$\mathcal{N}_{\mathbf{B}}^{\text{out}}(i)$ denote the set of parents (in-neighbors) and the set of children (out-neighbors) of vertex $i$, 
respectively. Also, $\mathcal{R}_\mathbf{B}$ denotes the set of roots of all possible spanning trees in $\mathcal{G}_\mathbf{B}$. Throughout \fy{this section}, we use the following definition 
of a matrix norm: Given an arbitrary vector norm $\|\cdot\|$, the induced norm of a matrix $W \in \mathbb{R}^{m \times n}$ is defined as 
$\|\mathbf{W}\|\triangleq \left\|\left[\left\|\mathbf{W}_{\bullet 1}\right\|,\ldots,\left\|\mathbf{W}_{\bullet n}\right\|\right]\right\|_2$.
\begin{remark}\label{rem:norms}
Under the above definition of matrix norm, it can be \fy{shown that}  $\|\mathbf{A}\mathbf{x}\| \leq \|\mathbf{A}\|\|\mathbf{x}\|$ for any 
$\mathbf{A} \in \mathbb{R}^{m \times m}$ and $\mathbf{x} \in \mathbb{R}^{m \times p}$. Also, for any $a \in \mathbb{R}^{m \times 1}$ and 
$x \in \mathbb{R}^{1 \times n}$, we have $\|ax\|=\|a\|\|x\|_2$ \fy{(cf.~\cite[Lemma 5]{pu2020push})}.
\end{remark}
\fy{\begin{remark}
Throughout this work, we assume that the set $X$ in problem~\eqref{eqn:bilevel_problem} is equal to $\mathbb{R}^n$, as our focus in this work primarily lies in addressing the equilibrium constraint. Notably, the methods and results in this work can be extended to address the more general setting when $X \subset \mathbb{R}^n$. In particular, in our prior work~\cite{qiu2023zeroth}, we employ Moreau smoothing to address the constraint set and show that the solution to the approximate smoothed problem is an approximate solution to the original problem. In the numerical experiments section, we employ this technique to address the presence of the constraint set $X$. 
\end{remark}}
\fy{In resolving} problem~\eqref{eqn:bilevel_problem}, in view of the presence of the \fy{VI} constraints, Lagrangian duality \fy{is generally not applicable. In fact, even in the special case where the VI constraint captures the solution set of an optimization problem, standard constraint qualifications (e.g., Slater condition) fail to hold.} Overcoming this challenge calls for new relaxation rules that can tackle the \fy{VI constraint}. To this end, motivated by the recent success of so-called \textit{iteratively regularized (IR)} algorithms~\cite{kaushik2021method,samadi2023improved,kaushik2023incremental,koshal12regularized,kannan2012distributed}, we develop Algorithm~\ref{algorithm:IR-push-pull}. \fy{Central to the IR framework is the principle that} the regularization parameter $\lambda_k$ is updated after every step within the algorithm. \fy{Algorithm~\ref{algorithm:IR-push-pull} is an iteratively regularized variant of the Push-Pull method in~\cite{pu2020push}. The novelty in the design of Algorithm~\ref{algorithm:IR-push-pull} lies in how we address the presence of the VI constraint through the IR approach.} Here, each agent holds a local copy of the global 
variable $x$, denoted by $x_{i,k}$, and an auxiliary variable $y_{i,k}$ is used to track the average of a regularized gradient. At each iteration, each agent $i$ 
uses the $i$th row of two matrices $\mathbf{R} =[R_{ij}]\in\mathbb{R}^{m\times m}$ and $\mathbf{C} =[C_{ij}]\in\mathbb{R}^{m\times m}$ to update vectors $x_{i,k}$ 
and $y_{i,k}$, respectively. Below, we state the main assumptions on the these two \textit{weight mixing} matrices. 
\begin{assumption}\label{assum:RC}
\noindent (a) The matrix $\mathbf{R}$ is nonnegative, with a strictly positive diagonal, and is row-stochastic, i.e., $\mathbf{R} \mathbf{1} =\mathbf{1}$. 
\noindent (b) The matrix $\mathbf{C}$ is nonnegative, with a strictly positive diagonal, and is column-stochastic, i.e., $ \mathbf{1} ^{\top}\mathbf{C}=\mathbf{1}^{\top}$. 
\noindent (c) The induced digraphs $\mathcal{G}_{\mathbf{R}}$ and  $\mathcal{G}_{\mathbf{C}^{\top}}$ satisfy $\mathcal{R}_{\mathbf{R}}\cap \mathcal{R}_{\mathbf{C}^{\top}}\neq \emptyset$. 
\end{assumption}

\begin{algorithm}
	\caption{IR-Push-Pull}\label{algorithm:IR-push-pull}
	  \begin{algorithmic}[1]
	  \STATE\textbf{Input:} For all $i \in [m]$, agent $i$ sets \fy{stepsize} $\gamma_{i,0} \geq 0$, pulling weights $R_{ij} \geq 0$ for all $j \in \mathcal{N}_{\mathbf{R}}^{\text{in}}(i)$, pushing weights $C_{ij} \geq 0$ for all $j \in \mathcal{N}_{\mathbf{C}}^{\text{out}}(i)$,  an arbitrary initial point $x_{i,0} \in \mathbb{R}^n$ and $y_{i,0}:=\yq{F_i(x_{i,0})}+\lambda_0\nabla f_i(x_{i,0})$; 
	  \FOR {$k=0,1,\ldots$}
			   \STATE For all $i \in [m]$, agent $i$ receives (pulls) the vector $x_{j,k}-\gamma_{j,k}y_{j,k}$ from each agent $j \in \mathcal{N}_{\mathbf{R}}^{\text{in}}(i)$, sends (pushes) $C_{\ell i}y_{i,k}$ to each agent $\ell \in \mathcal{N}_{\mathbf{C}}^{\text{out}}(i)$, and does the following updates:
		\STATE $x_{i,k+1} := \sum\nolimits_{j=1}^m R_{ij}\left(x_{j,k}-\gamma_{j,k}y_{j,k}\right)$;
		\STATE $y_{i,k+1} := \sum\nolimits_{j=1}^m C_{ij}y_{j,k}+ \yq{F_i(x_{i,k+1})}+\lambda_{k+1}\nabla f_i(x_{i,k+1})-\yq{F_i(x_{i,k})}-\lambda_{k}\nabla f_i(x_{i,k})$;
	 \ENDFOR
	 \end{algorithmic}
  \end{algorithm}
   Assumption \ref{assum:RC} does not require the strong condition of a doubly stochastic matrix for communication in a directed network. In turn, utilizing a push-pull protocol and in a similar fashion to~\cite{pu2020push}, it only entails a row stochastic $\mathbf{R}$ and a column stochastic matrix $\mathbf{C}$. An example is as follows where agent $i$ chooses scalars $r_i,c_i>0$ and sets $R_{i,j} := 1/\left(\left|\mathcal{N}_{\mathbf{R}}^{\text{in}}(i)\right|+r_i\right)$ for $j \in \mathcal{N}_{\mathbf{R}}^{\text{in}}(i)$, $R_{i,i} :=r_i/\left(\left|\mathcal{N}_{\mathbf{R}}^{\text{in}}(i)\right|+r_i\right)$, $C_{\ell, i} :=1/\left(\left|\mathcal{N}_{\mathbf{C}}^{\text{out}}(i)\right|+c_i\right)$ for $\ell \in \mathcal{N}_{\mathbf{C}}^{\text{out}}(i)$, $C_{i, i} :=c_i/\left(\left|\mathcal{N}_{\mathbf{C}}^{\text{out}}(i)\right|+c_i\right)$, and $0$ otherwise.
Note that Assumption~\ref{assum:RC}(c) is weaker than imposing strong connectivity on $\mathcal{G}_{\mathbf{R}}$ and $\mathcal{G}_{\mathbf{C}}$. The update 
rules in Algorithm \ref{algorithm:IR-push-pull} can be compactly represented as the following:
\begin{align}
\mathbf{x}_{k+1} &:= \mathbf{R}\left(\mathbf{x}_k-\boldsymbol{\gamma}_k\mathbf{y}_k\right),\label{alg:IRPP_compact1}\\
\mathbf{y}_{k+1} &:= \mathbf{C}\mathbf{y}_k+\yq{\mathbf{F}(\mathbf{x}_{k+1})}+ \lambda_{k+1}\nabla \mathbf{f}(\mathbf{x}_{k+1}) 
\bz{\hspace{0.1cm}-}\yq{\mathbf{F}(\mathbf{x}_k)}  -\lambda_k\nabla \mathbf{f}(\mathbf{x}_k)\label{alg:IRPP_compact2},
\end{align}
where $\boldsymbol{\gamma}_k\geq 0$ is defined as $\boldsymbol{\gamma}_k \triangleq \text{diag}\left(\gamma_{1,k},\ldots,\gamma_{m,k}\right)$.

\subsection{Preliminaries of convergence analysis of Algorithm \ref{algorithm:IR-push-pull}}

Under Assumption \ref{assum:RC}, there exists a unique nonnegative left eigenvector $u \in \mathbb{R}^m$ such that $u^{\top}\mathbf{R} = u^{\top}$ and $u^{\top}\mathbf{1} =m$. 
Similarly, there exists a unique nonnegative right eigenvector $v \in \mathbb{R}^m$ such that $\mathbf{C}v = v$ and $\mathbf{1}^{\top}v =m$ (cf. Lemma 1 in 
\cite{pu2020push}). \fy{We utilize the following in the analysis of Algorithm \ref{algorithm:IR-push-pull}.} 
\begin{definition}\label{eqn:defs}
For $k\geq 0$ and the regularization parameter $\lambda_k>0$, let 
$x^* \triangleq \argmin_{x \in \fy{{\scriptsize \mbox{SOL}}(\mathbb{R}^n,F)}}\{ f(x)\} \in \mathbb{R}^{1\times n}$, \yq{$x^*_{\lambda_k} \triangleq \mbox{SOL}(\mathbb{R}^n,F+ \lambda_k \nabla f)$}. 
We define the mapping $\mathbf{G}_k(\mathbf{x})\triangleq \yq{\mathbf{F}\left(\mathbf{x}\right)}+\lambda_k \nabla \mathbf{f}\left(\mathbf{x}\right)\in \mathbb{R}^{m\times n}$, 
and functions $G_k(\mathbf{x})\triangleq \tfrac{1}{m}\mathbf{1}^{\top}\mathbf{G}_k(\mathbf{x})\in \mathbb{R}^{1\times n}$, $\mathscr{G}_k(x) \triangleq G_k\left(\mathbf{1}x^{\top}\right)\ \in \mathbb{R}^{1\times n}$, $\bar g_k \triangleq \mathscr{G}_k(\bar x_k)\ \in \mathbb{R}^{1\times n}$. 
We let $L_k \triangleq \yqo{L_F}+\lambda_kL_f$, and define vectors $\bar x_k \triangleq \tfrac{1}{m}u^{\top}\mathbf{x}_k\ \in \mathbb{R}^{1\times n}$, and 
$\bar y_k \triangleq \tfrac{1}{m}\mathbf{1}^{\top}\mathbf{y}_k\ \in \mathbb{R}^{1\times n}$. Lastly, $\Lambda_k \triangleq \left|1- \frac{\lambda_{k+1}}{\lambda_{k}}\right|$.
\end{definition}
Here, $x^*$ denotes the optimal solution of problem \eqref{eqn:bilevel_problem} and $x^*_\lambda$ is defined as the \fy{unique} solution to a regularized problem. 
Note that the \yq{strong monotonicity of $F(x)+\lambda_k\nabla f(x)$} implies that $x^*_{\lambda_k}$ exists and is a unique vector. Also, under Assumption~\ref{assum:problem}, the set \yq{$\text{SOL}\left(\mathbb{R}^n, \sum_{i=1}^m F_i\right)$} is closed and convex. As such, from the strong convexity of $f$ and 
\fy{invoking~\cite[Prop. 1.1.2]{Bertsekas2016}} again, we conclude that $x^*$ also exists and is a unique vector. The sequence $\{x^*_{\lambda_k}\}$ is the 
so-called \textit{Tikhonov trajectory} and plays a key role in the convergence analysis (cf. \cite[Ch. 12]{facchinei02finite}). The mapping $\mathbf{G}_k(\mathbf{x})$ 
denotes the regularized gradient matrix. The vector $\bar x_k$ holds a weighted average of \fy{local copies of agents'} iterates. Next, we consider a family 
of update rules for the sequences of the \fy{stepsize} and the regularization parameter under which the convergence and rate analysis can be performed. 
\begin{assumption}[Update rules]\label{assum:update_rules}
Let \yq{$\hat \gamma_k:=\tfrac{\hat \gamma}{(k+\Gamma)^a}$ and $\lambda_k:=\tfrac{\lambda}{(k+\Gamma)^b}$} for all $k\geq 0$ where $\fy{\hat\gamma},\fy{\lambda},\Gamma, a$ 
and $b$ are strictly positive scalars \fy{and $\hat \gamma_k \triangleq \max_{j \in [m]}\, \gamma_{j,k}$}. Let $a>b>0$, $a+b<1$, \fy{and $2a+3b<2$}. Assume that $\Gamma \geq 1$ and 
\fy{$\Gamma\geq  \hat{\Gamma}_1 \triangleq \sqrt[{1-a-b}]{\tfrac{4}{\fy{\hat\gamma}\fy{\lambda}\mu_f\tau}}$, for some $\tau>0$}. Also, let $\alpha_k \geq \theta \hat \gamma_k$ for 
$k\geq 0$ for some $\theta>0$, where $\alpha_k  \triangleq  \tfrac{1}{m}u^{\top}\boldsymbol{\gamma}_kv$. 
\end{assumption}
\fy{The scalars $a$ and $b$ prescribe the tuning rules for the stepsize and iterative regularization. The specified assumptions on these scalars play a key role in establishing the convergence of the method.} The constant $\theta$ in Assumption \ref{assum:update_rules} measures the size of the range within which the agents in 
$\mathcal{R}_\mathbf{R}\cap \mathcal{R}_{\mathbf{C}^{\fy{\top}}}$ select their stepsizes. The condition $\alpha_k \geq \theta \hat \gamma_k$
is satisfied in many cases including the case where all the agents choose strictly positive stepsizes (see~\cite[Remark 4]{pu2020push}). 
In the following lemma, we list some of the main properties of the update rules in Assumption~\ref{assum:update_rules} that will be used in the analysis.
\begin{lemma}[Properties of the update rules]\label{lem:update_rules_props}
\fy{Let Assumption \ref{assum:update_rules} hold. Then, the following results hold.}

\noindent \yq{(i)} $\{\lambda_k\}_{k=0}^\infty$ is a decreasing strictly positive sequence satisfying $\lambda_k \to 0$; $\{\hat \gamma_k\}_{k=0}^\infty$ is a decreasing strictly positive sequence such that $\hat \gamma_k \to 0$ and $\frac{\hat \gamma_k}{\lambda_k} \to 0 $. 

\noindent \yq{(ii)} $\frac{\Lambda_k}{\lambda_k} \to 0$,  $\Lambda_{k+1}\leq \Lambda_{k}$ for all $k\geq 0$, $\Lambda_{k-1} \leq \frac{1}{k+\yq{\Gamma}}$ for $k\geq 1$, where $\Lambda_k$ is given by Def. \ref{eqn:defs}.

\noindent \yq{(iii)} $\frac{(k+\yq{\Gamma})\hat\gamma_k\lambda_k}{\yq{(k+\Gamma -1)}\hat\gamma_{k-1}\lambda_{k-1}}\leq 1+\yq{0.5\mu_f\hat \gamma_k\lambda_k\tau}$ for all \yq{$k\geq 1$ and $\tau>0$}.
\end{lemma}

\begin{proof}
\noindent \yq{(i)} Recall that \yq{$\hat \gamma_k=\frac{\fy{\hat\gamma}}{(k+\Gamma)^a}$} and \yq{$\lambda_k=\frac{ \fy{\lambda}}{(k+\Gamma)^b}$} where $0<b<a<1$, \yq{$\Gamma \geq 1$} and $a+b<1$. Consequently, $\{\hat \gamma_k\}_{k=0}^\infty$ and $\{\lambda_k\}_{k=0}^\infty$ are strictly positive decreasing sequences and $\hat \gamma_k \to 0$, $\lambda_k \to 0$, and $\frac{\hat \gamma_k}{\lambda_k} \to 0 $. 

\noindent \yq{(ii)} \yq{First}, we show that $\Lambda_{k-1} \leq \frac{1}{k+\yq{\Gamma}}$ for $k\geq 1$. From Def. \ref{eqn:defs} and that $\lambda_{k}\leq \lambda_{k-1}$, for any $k\geq 1$ we have
\begin{equation}\label{eqn:Lambda_open}
\begin{split}
	\Lambda_{k-1}&=1-\tfrac{\lambda_k}{\lambda_{k-1}}=1-\tfrac{\fy{\lambda}(k+\yq{\Gamma})^{-b}}{\fy{\lambda}\yq{(k-1+\Gamma)}^{-b}}
	\bz{=1-(\tfrac{\yq{k+\Gamma-1}}{k+\yq{\Gamma}})^b=1-(1-\tfrac{1}{k+\yq{\Gamma}})^b.}
\end{split}
\end{equation}
From $0<b<a$ and $a+b<1$, we have $b<0.5$. This implies that \bz{$(1-\frac{1}{k+\yq{\Gamma}})^b\geq (1-\frac{1}{k+\yq{\Gamma}})^{0.5}$}. Combining this relation with \eqref{eqn:Lambda_open}, we have
\begin{align*}
\Lambda_{k-1} &\leq 1-(1-\tfrac{1}{k+\yq{\Gamma}})^{0.5}=\tfrac{1-(1-\tfrac{1}{k+\yq{\Gamma}})}{1+\sqrt{1-\tfrac{1}{k+\yq{\Gamma}}}}
=\left({1+\sqrt{1-\tfrac{1}{k+\yq{\Gamma}}}}\right)^{-1}\tfrac{1}{k+\yq{\Gamma}}\leq \tfrac{1}{k+\yq{\Gamma}},
\end{align*}
where the last inequality is implied from $k\geq 1$. Next, we show $\Lambda_{k+1}\leq \Lambda_k$ for all $k\geq 0$. From \eqref{eqn:Lambda_open}, we have
$
\Lambda_{k+1} =1-(1-\tfrac{1}{\yq{k+2+\Gamma}})^b\leq1-(1-\tfrac{1}{\yq{k+1+\Gamma}})^b = \Lambda_{k}.
$

\noindent \yq{(iii)} 
From \fy{the update rules} of $\hat \gamma_k$ and $\lambda_k$, we have
\bz{\begin{align*}
&(\tfrac{(k+\yq{\Gamma})\hat\gamma_k\lambda_k}{\yq{(k+\Gamma-1)}\hat\gamma_{k-1}\lambda_{k-1}}-1)\tfrac{1}{\hat \gamma_k \lambda_k \yq{\mu_f\tau}}
=(\tfrac{(k+\yq{\Gamma})}{\yq{(k+\Gamma-1)}}(\tfrac{\yq{k+\Gamma-1}}{k+\yq{\Gamma}})^{a+b}-1)\tfrac{(k+\yq{\Gamma})^{a+b}}{\fy{\hat\gamma} \fy{\lambda} \yq{\mu_f\tau}}\\
&\yq{=} ((1+\tfrac{1}{\yq{k+\Gamma-1}})^{1-a-b}-1)\tfrac{(k+\yq{\Gamma})^{a+b}}{\fy{\hat\gamma} \fy{\lambda} \yq{\mu_f\tau}}\yq{\leq \tfrac{(k+\Gamma)^{a+b}}{(k+\Gamma -1)\fy{\hat\gamma} \fy{\lambda} \mu_f\tau}}
\yq{\leq \tfrac{1+\tfrac{1}{\Gamma}}{(k+\Gamma)^{1-a-b}\fy{\hat\gamma} \fy{\lambda} \mu_f\tau}}\yq{\leq \tfrac{2}{\Gamma^{1-a-b}\fy{\hat\gamma} \fy{\lambda} \mu_f\tau}}\yq{\leq \tfrac{1}{2},}
\end{align*}}
\yq{where the last two inequalities are implied by $\Gamma \geq 1$ and $\Gamma^{1-a-b}\geq  \tfrac{4}{\fy{\hat\gamma}\fy{\lambda}\mu_f\tau}$. Then, the \fy{result} in (iii) \fy{follows}.}
\end{proof}

Next, we present some key properties of the regularized sequence $\{x^*_{\lambda_k}\}$ that will be used in the rate analysis. 
\begin{lemma}[Properties of Tikhonov trajectory]\label{lemma:IR-props}
Let Assumptions \ref{assum:problem} and \ref{assum:update_rules} hold and $x^*_{\lambda_k}$ be given by Def. \ref{eqn:defs}. Then, we have: (i) The sequence $\{x^*_{\lambda_k}\}$ converges to the unique solution of problem \eqref{eqn:bilevel_problem}, i.e., $x^*$. (ii) There exists a scalar $M>0$ such that for any $k \geq 1$, we have $
\|x^*_{\lambda_k}-x^*_{\lambda_{k-1}}\|_2 \leq \frac{M}{\mu_f}\Lambda_{k-1}$.
\end{lemma}
\begin{proof}
\fy{The proof can be done in a similar vein to that of~\cite[Lemma 4.5]{kaushik2021method}.}
\end{proof}
In the following, we state \fy{some}  properties of the regularized maps to be used in \fy{obtaining} error bounds in the next section. 
\begin{lemma}\label{lemma:grad_track_props}
Consider Algorithm \ref{algorithm:IR-push-pull}. Let Assumptions \ref{assum:problem} and \ref{assum:RC} hold. For any $k\geq 0$, mappings $G_k$, $\mathscr{G}_k$, and $\bar g_k$ given by Def. \ref{eqn:defs} satisfy the following relations: (i) We have that $\bar y_k = G_k(\mathbf{x}_k)$. (ii) We have $\mathscr{G}_k\left(x^*_{\lambda_k}\right) = 0$. (iii)  The mapping $\mathscr{G}_k(x)$ is  $(\lambda_k\mu_f)$-strongly monotone and Lipschitz continuous with parameter $L_k$. (iv) We have $\|\bar y_k - \bar g_k \|_2\leq \frac{L_k}{\sqrt{m}}\left\|\mathbf{x}_k-\mathbf{1}\bar x_k\right\|_2$ and $\|\bar g_k \|_2\leq L_k\|\bar x_k-x^*_{\lambda_k}\|_2$.
\end{lemma}

\begin{proof}
\noindent \yq{(i)} Multiplying both sides of \eqref{alg:IRPP_compact2} by $\tfrac{1}{m}\mathbf{1}^{\top}$ and from the definitions of $\mathbf{G}_k$ and $G_k$ in Def. \ref{eqn:defs}, we obtain
$
\bar y_{k+1}=\tfrac{1}{m}\mathbf{1}^{\top}\mathbf{y}_{k} + \tfrac{1}{m}\mathbf{1}^{\top}\mathbf{G}_{k+1}(\mathbf{x}_{k+1}) - \tfrac{1}{m}\mathbf{1}^{\top}\mathbf{G}_{k}(\mathbf{x}_{k})
=\bar y_{k} +G_{k+1}(\mathbf{x}_{k+1})-G_{k}(\mathbf{x}_{k}),
$
where we used $\mathbf{1}^{\top}\mathbf{C} = \mathbf{1}^{\top}$.
From Algorithm \ref{algorithm:IR-push-pull}, we have $\mathbf{y}_0 : = \yq{\mathbf{F}(\mathbf{x}_0)} + \lambda_0 \nabla \mathbf{f}(\mathbf{x}_0)=\mathbf{G}_0(\mathbf{x}_0)$, implying that $\bar y_{0}=G_{0}(\mathbf{x}_{0})$. From the two preceding relations, we obtain that $\bar y_k = G_k\left(\mathbf{x}_k\right)$.

\noindent \yq{(ii, iii)}
From Def. \ref{eqn:defs}, we have that $G_k(\mathbf{x}) = \tfrac{1}{m}\sum_{i=1}^m\left(\yq{ F_i\left(x_{i,k}\right)}+\lambda_k\nabla f_i\left(x_{i,k}\right)\right)$.
Thus, from the definition of $\mathscr{G}_k$ we obtain that $\mathscr{G}_k(x) = G_k\left(\mathbf{1}x^{\top}\right) = \tfrac{1}{m}\sum_{i=1}^m\left(\yq{F_i\left(x\right)}+\lambda_k\nabla f_i\left(x\right)\right) = \tfrac{1}{m}\left(\yq{F(x)} +\lambda_k \nabla f(x)\right)$.
Thus, from the definition of $x^*_{\lambda_k}$ in Def. \ref{eqn:defs}, we obtain $\mathscr{G}_k\left(x^*_{\lambda_k}\right) = 0$.
Also, from Assumption \ref{assum:problem}, we conclude that  $\mathscr{G}_k(x)$ is a $(\lambda_k\mu_f)$-strongly monotone mapping and Lipschitz continuous with parameter
$L_k \triangleq \yq{L_F}+\lambda_kL_f$ for $k\geq 0$.

\noindent \yq{(iv)} For any $\mathbf{u},\mathbf{v} \in \mathbb{R}^{m\times n}$, with $u_i,v_i \in \mathbb{R}^n$ denoting the $i^{\text{th}}$ row of $\mathbf{u},\mathbf{v}$, respectively, we have
\begin{align*}
&\left\|G_k(\mathbf{u}) - G_k(\mathbf{v})\right\|_2
=\left\|\tfrac{1}{m}\mathbf{1}^{\top}\left(\yq{\mathbf{F}(\mathbf{u})}+\lambda_k\nabla \mathbf{f}(\mathbf{u})\right) - \tfrac{1}{m}\mathbf{1}^{\top}\left(\yq{\mathbf{F}\left(\mathbf{v}\right)}+\lambda_k \nabla \mathbf{f}\left(\mathbf{v}\right)\right)\right\|_2\\
&\leq \tfrac{1}{m}\|\textstyle\sum\nolimits_{i=1}^{m}\yq{F_i(u_{i})}-\textstyle\sum\nolimits_{i=1}^m \yq{F_i(v_i)}\|_2 
+\tfrac{\lambda_k}{m}\|\textstyle\sum\nolimits_{i=1}^m\nabla f_i(u_i)-\textstyle\sum\nolimits_{i=1}^m \nabla f_i\left(v_i\right)\|_2\\
&\leq \tfrac{1}{m}\textstyle\sum_{i=1}^m\left(\left\|\yq{F_i(u_i)}- \yq{F_i\left(v_i\right)}\right\|_2+\lambda_k\left\|\nabla f_i(u_i)- \nabla f_i\left(v_i\right)\right\|_2\right)\\
&\leq \tfrac{1}{m}\textstyle\sum\nolimits_{i=1}^m\left(\yq{L_F}\|u_i-v_i\|_2+\lambda_kL_{f}\|u_i-v_i\|_2\right)
\leq \tfrac{L_k}{m}\textstyle\sum\nolimits_{i=1}^m\|u_i-v_i\|_2\leq \tfrac{L_k}{\sqrt{m}}\|\mathbf{u}-\mathbf{v}\|_2.
\end{align*}
Consequently, we obtain $\|\bar y_k - \bar g_k\|_2 = \left\|G_k(\mathbf{x}_k)-G_k(\mathbf{1}\bar x_k)\right\|_2\leq \frac{L_k}{\sqrt{m}}\|\mathbf{x}_k-\mathbf{1}\bar x_k\|_2$.
Also, using the Lipschitzian property of $\mathscr{G}_k$ in part (ii) and $\mathscr{G}_k(x^*_{\lambda_k}) = 0$, we obtain
$\|\bar g_k\|_2=\|\mathscr{G}_k(\bar x_k)\|_2 = \|\mathscr{G}_k(\bar x_k)-\mathscr{G}_k\left(x^*_{\lambda_k}\right)\|_2
\leq L_k\|\bar x_k-x^*_{\lambda_k}\|_2.$
\end{proof}
\yq{Next, we \fy{derive a} bound \fy{on} the \fy{squared} distance between \fy{the} average \fy{iterate} and the Tikhonov trajectory.
\begin{lemma}\label{lem:gradient_recursion} Let $\alpha_k   \leq\frac{\lambda_k\mu_f}{L_0^2}$ for $k\geq 0$. We have
$
\|\bar x_{k} -\alpha_k  \bar g_k -x^*_{\lambda_k}\|_2^2 \leq (1-0.5\alpha_k\lambda_k\mu_f)^2\|\bar x_k-x^*_{\lambda_k}\|_2^2.
$
\end{lemma}
\begin{proof}
We have
\begin{align*}
&\|\bar x_{k} -\alpha_k \bar g_k -x^*_{\lambda_k}\|_2^2 
= \|\bar x_k -x^*_{\lambda_k}\|_2^2 + \alpha_k^2\|\bar g_k \|_2^2 -2\alpha_k\bar g_k^{\top}(\bar x_k -x^*_{\lambda_k})\\
&\stackrel{\text{Lemma }\ref{lemma:grad_track_props} \text{(iii)}}{\leq} \|\bar x_k -x^*_{\lambda_k}\|_2^2 + \alpha_k^2L_k^2\|\bar x_k -x^*_{\lambda_k}\|_2^2 -2\alpha_k\lambda_k\mu_f\|\bar x_k -x^*_{\lambda_k}\|_2^2\\ 
&=(1+\alpha_k^2L_k^2-2\alpha_k\lambda_k\mu_f)\|\bar x_k -x^*_{\lambda_k}\|_2^2.
\end{align*}
Let $\alpha_k   \leq\tfrac{\lambda_k\mu_f}{L_0^2}$ and recall that $L_k\leq L_0$. \fy{We} obtain
$\|\bar x_{k} -\alpha_k  \bar g_k -x^*_{\lambda_k}\|_2^2 \leq (1-\alpha_k\lambda_k\mu_f)\|\bar x_k-x^*_{\lambda_k}\|_2^2.$
\fy{Then,} the desired relation is obtained by noting that $1-\alpha_k\lambda_k\mu_f \leq (1-0.5\alpha_k\lambda_k\mu_f)^2$.
\end{proof}}
We state the following result from \cite{pu2020push} introducing two matrix norms induced by matrices $\mathbf{R}$ and $\mathbf{C}$.
\begin{lemma}[cf. Lemma 4 and Lemma 6 in \cite{pu2020push}]\label{lemma:norms_relations}
Let Assumption \ref{assum:RC} hold. Then: (i) There exist matrix norms $\|\cdot\|_{\mathbf{R}}$ and  $\|\cdot\|_{\mathbf{C}}$ such that for $\sigma_{\mathbf{R}}\triangleq \|\mathbf{R}-\frac{\mathbf{1}u^{\top}}{m}\|_{\mathbf{R}}$ and $\sigma_{\mathbf{C}}\triangleq \|\mathbf{C}-\frac{\mathbf{1}v^{\top}}{m}\|_{\mathbf{C}}$ we have that $\sigma_{\mathbf{R}}<1$ and $\sigma_{\mathbf{C}}<1$. (ii) There exist scalars $\delta_{\mathbf{R},2},\delta_{\mathbf{C},2},\delta_{\mathbf{R},\mathbf{C}},\delta_{\mathbf{C},\mathbf{R}}>0$ such that for any $W \in \mathbb{R}^{m \times n}$, we have $\|W\|_{\mathbf{R}}\leq \delta_{\mathbf{R},2}\|W\|_2$, $\|W\|_{\mathbf{C}}\leq \delta_{\mathbf{C},2}\|W\|_2$, $\|W\|_{\mathbf{R}}\leq \delta_{\mathbf{R},\mathbf{C}}\|W\|_\mathbf{C}$, $\|W\|_{\mathbf{C}}\leq \delta_{\mathbf{C},\mathbf{R}}\|W\|_\mathbf{R}$, $\|W\|_{2}\leq \|W\|_\mathbf{R}$, and $\|W\|_{2}\leq \|W\|_\mathbf{C}$.
\end{lemma}
The following result will be employed in the rate analysis. 
\yq{\begin{lemma}\label{lemma:alpha_k:bound}
Let $\alpha_k$ be \fy{given by} Assumption \ref{assum:update_rules}. \fy{Suppose $\Gamma \geq  \hat{\Gamma}_2 \triangleq  \max\bigg\{\sqrt[a]{\tfrac{L_0\hat{\gamma}_0 u^{\top}v}{m}},\sqrt[a-b]{\tfrac{\hat \gamma u^{\top}v L_0^2 }{m \mu_f\lambda }}\bigg\}$. Then, we have $\alpha_k \leq \min\{\tfrac{1}{L_0},\tfrac{\lambda_k \mu_f}{L_0^2}\}$} for all $k \geq 0$.
\end{lemma}
\begin{proof}\fy{First, we show that $\alpha_k \leq \tfrac{1}{L_0}$.} Recall the definition of $\alpha_k$ and $\hat\gamma_k$. \fy{We may write}
$\alpha_k  \triangleq  \tfrac{1}{m}u^{\top}\boldsymbol{\gamma}_kv = \tfrac{1}{m}\sum\nolimits_{i=1}^m \gamma_{i,k}u_iv_i \leq \tfrac{1}{m}\sum\nolimits_{i=1}^m \hat\gamma_k u_iv_i ,$
where $u_i$ and $v_i$ \fy{are} the $i$th element of $u$ and $v$, \fy{respectively. Then,} we have
$\alpha_k \leq \tfrac{1}{m}\sum\nolimits_{i=1}^m \tfrac{\fy{\hat\gamma} u_iv_i}{(k+\Gamma)^a} = \tfrac{\fy{\hat\gamma} u^{\top} v}{m(k+\Gamma)^a},$ \fy{for all $k\geq 0$. From $\Gamma \geq \sqrt[a]{\tfrac{L_0\hat{\gamma}_0 u^{\top}v}{m}}$ and rearranging the terms,} we obtain
$\tfrac{1}{L_0} \geq \tfrac{\fy{\hat\gamma} u^{\top} v}{m\Gamma^a}.$ Therefore, we have
$\alpha_k \leq \tfrac{\fy{\hat\gamma} u^{\top} v}{m(k+\Gamma)^a} \leq \tfrac{\fy{\hat\gamma} u^{\top} v}{m\Gamma^a} \leq \tfrac{1}{L_0}.$ \fy{The second bound on $\alpha_k$ can be established in a similar manner, and its proof is therefore omitted.}
\end{proof}}
\subsection{Convergence and rate analysis of Algorithm \ref{algorithm:IR-push-pull}}
We analyze the convergence of Algorithm \ref{algorithm:IR-push-pull} by introducing the \fy{error} metrics $\|\bar x_{k+1}-x^*_{\lambda_k}\|_2$,  $\|\mathbf{x}_{k+1}-\mathbf{1}\bar x_{k+1}\|_{\mathbf{R}}$, $\|\mathbf{y}_{k+1} - v \bar y_{k+1}\|_{\mathbf{C}}$. Of these, the first term relates the averaged iterate with the Tikhonov trajectory, the second term measures the consensus \yq{error} for the decision matrix, and the third term measures the consensus \yq{error} for the matrix of the regularized gradients.  For $k\geq 1$, let us define $\Delta_k$ as $\Delta_k \triangleq [\|\bar x_k-x^*_{\lambda_{k-1}}\|_2, \|\mathbf{x}_k-\mathbf{1}\bar x_k\|_{\mathbf{R}}, \|\mathbf{y}_k-v \bar y_k\|_{\mathbf{C}}]^{\top}.$
\begin{proposition}\label{prop:main_ineq_1}
Consider Algorithm \ref{algorithm:IR-push-pull} under Assumptions~\ref{assum:problem}, \ref{assum:RC}, and~\ref{assum:update_rules}. Let $\alpha_k $ and $\hat\gamma_k$ be given by Assumption~\ref{assum:update_rules}, $c_0\triangleq \delta_{\mathbf{C},2}\|\mathbf{I}-\tfrac{1}{m}v\mathbf{1}^{\top}\|_\mathbf{C}$, \fy{and $\Gamma \geq  \hat{\Gamma}_2 $}. Then, there exist scalars $M>0$ and $B_{\fy{F}}>0$ \yq{such that for any $k\geq 0$}, we have $\Delta_{k+1} \leq H_k\Delta_k+h_k$ where $H_k =[H_{ij,k}]_{3\times 3}$ and $h_k =[h_{i,k}]_{3\times 1}$ are given as follows:
\begin{align*}
&H_{11,k} :=\yq{1-0.5\alpha_k\lambda_k\mu_f}, \ 
H_{12,k} := \textstyle{\frac{\alpha_kL_k}{\sqrt{m}}}, \ 
H_{13,k} :=\textstyle{\frac{\hat\gamma_k\|u\|_2}{m}},\ H_{21,k} :=  \textstyle{\sigma_{\mathbf{R}}\hat\gamma_k L_k\|v\|_{\mathbf{R}}}\\
&H_{22,k} := \textstyle{\sigma_{\mathbf{R}}(1+\hat\gamma_k\|v\|_{\mathbf{R}}\frac{L_k}{\sqrt{m}})},\ H_{23,k} :=  \textstyle{\sigma_{\mathbf{R}}\hat\gamma_k\delta_{\mathbf{R},\mathbf{C}}},\ H_{31,k} := \textstyle{c_0L_{k}\left(\hat\gamma_k\|\mathbf{R}\|_2 \|v\|_2L_k+2\sqrt{m}\Lambda_{k}\right)},\\ 
&H_{32,k} := \textstyle{c_0L_{k}(\| \mathbf{R}-\mathbf{I}\|_2 +\hat\gamma_k\|\mathbf{R}\| \|v\|_2 \frac{L_k}{\sqrt{m}}   + 2\Lambda_{k})},\ H_{33,k} :=\textstyle{\sigma_{\mathbf{C}} +c_0L_{k}\hat\gamma_k \|\mathbf{R}\|_2},\ h_{1,k}:=\textstyle{\frac{M\Lambda_{k-1}}{\mu_f}},\\
&h_{2,k}:=\textstyle{\frac{M\sigma_{\mathbf{R}}\hat\gamma_k L_k\|v\|_{\mathbf{R}}}{\mu_f}\Lambda_{k-1}}, h_{3,k}:= \textstyle{c_0L_{k}(\hat\gamma_k\|\mathbf{R}\|_2 \|v\|_2 L_k+\sqrt{m}\Lambda_{k}+\frac{\mu_f c_0B_{\fy{F}}}{M}) \frac{M\Lambda_{k-1}}{\mu_f}}.
\end{align*}
\end{proposition}
\begin{proof} 
	\noindent First, we show $\Delta_{1,k+1}\leq \sum_{j=1}^3 H_{1j,k}\Delta_{j,k} +h_{1,k} $. From \eqref{alg:IRPP_compact1} and Def. \ref{eqn:defs}, we obtain 
	$\bar x_{k+1} =  u^{\top}\mathbf{R}\left(\mathbf{x}_k -\boldsymbol{\gamma}_k\mathbf{y}_k\right)/m = \bar x_k - u^{\top}\boldsymbol{\gamma}_k\mathbf{y}_k/m.$
	
Thus, we have
$
\bar x_{k+1} =\bar x_k - u^{\top}\boldsymbol{\gamma}_k\left(\mathbf{y}_k-v \bar y_k+v\bar y_k\right)/m 
 = \bar x_k -\alpha_k  \bar g_k -\alpha_k  \left(\bar y_k -\bar g_k\right) - u^{\top}\boldsymbol{\gamma}_k\left(\mathbf{y}_k-v\bar y_k\right)/m.
$
	From \yq{Lemma \ref{lemma:alpha_k:bound}, $\alpha_k \leq \tfrac{1}{L_0}   <\frac{2}{L_0}\leq \frac{2}{L_k}$ for all $k\geq 0$ for $\Gamma \geq \sqrt[a]{\tfrac{L_0\hat{\gamma}_0 u^{\top}v}{m}}$.} From Lemma \ref{lemma:grad_track_props}(iii), $\mathscr{G}_k(x)$ is $(\mu_f\lambda_k)$-strongly convex and $L_k$-smooth. Invoking Lemma \ref{lem:gradient_recursion}, we obtain
$
	\|\bar x_{k+1}-x^*_{\lambda_k}\|_2 
	 = \|\bar x_k -x^*_{\lambda_k}-\alpha_k  \bar g_k -\alpha_k  \left(\bar y_k -\bar g_k\right) -\tfrac{1}{m}u^{\top}\boldsymbol{\gamma}_k\left(\mathbf{y}_k-v\bar y_k\right)\|_2 \\
	\leq \yq{(1-0.5\alpha_k\lambda_k\mu_f)}\|\bar x_k-x^*_{\lambda_k}\|_2 +\alpha_k  \|\bar y_k -\bar g_k\|_2 +\tfrac{1}{m}\|u^{\top}\boldsymbol{\gamma}_k(\mathbf{y}_k-v\bar y_k)\|_2.
$

Adding and subtracting $x^*_{\lambda_{k-1}}$ and using Lemmas \ref{lemma:IR-props} and \ref{lemma:grad_track_props}(iv), \yq{we obtain}
	\bz{
		$
		\|\bar x_{k+1}-x^*_{\lambda_k}\|_2 
		\leq \yq{(1-0.5\alpha_k\lambda_k\mu_f)}\|\bar x_k-x^*_{\lambda_{k-1}}\|_2+\tfrac{M\Lambda_{k-1}}{\mu_f}+\tfrac{\alpha_k  L_k}{\sqrt{m}}\|\mathbf{x}_k-\mathbf{1}\bar x_k\|_2 +\tfrac{\|u\|_2\|\boldsymbol{\gamma}_k\|_2}{m}\left\|\mathbf{y}_k-v\bar y_k\right\|_2.
$
	}
	
	Then, the desired inequality is obtained by invoking Lemma~\ref{lemma:norms_relations}(ii), Remark \ref{rem:norms}, and definition of $\hat\gamma_k$.
	
	\noindent  Second, we show $\Delta_{2,k+1}\leq \sum_{j=1}^3 H_{2j,k}\Delta_{j,k} +h_{2,k} $. From \eqref{alg:IRPP_compact1} and Def. \ref{eqn:defs} and that $\mathbf{R}\mathbf{1} = \mathbf{1}$, we have
$
	\mathbf{x}_{k+1} - \mathbf{1} \bar x_{k+1} \textstyle{= \mathbf{R}\left(\mathbf{x}_k-\boldsymbol{\gamma}_k \mathbf{y}_k\right) - \mathbf{1}\bar x_k +\tfrac{1}{m}\mathbf{1}u^{\top}\boldsymbol{\gamma}_k \mathbf{y}_k} \textstyle{= \left(\mathbf{R} -\mathbf{1}u^{\top}/m\right)\left(\left(\mathbf{x}_k-\mathbf{1} \bar x_k\right) -\boldsymbol{\gamma}_k  \mathbf{y}_k\right).}
$

	Applying Lemma \ref{lemma:norms_relations}, Remark \ref{rem:norms}, and Lemma \ref{lemma:grad_track_props}, we obtain
	\begin{align*}
	&\|\mathbf{x}_{k+1}-\mathbf{1}\bar x_{k+1}\|_{\mathbf{R}} \leq \sigma_{\mathbf{R}}\left\|\mathbf{x}_k-\mathbf{1}\bar x_k\right\|_{\mathbf{R}}+ \sigma_{\mathbf{R}}\|\boldsymbol{\gamma}_k \|_{\mathbf{R}} \|\mathbf{y}_k\|_{\mathbf{R}}\\
	&\leq \sigma_{\mathbf{R}}\left\|\mathbf{x}_k-\mathbf{1}\bar x_k\right\|_{\mathbf{R}}+ \sigma_{\mathbf{R}}\|\boldsymbol{\gamma}_k \|_{2}\|\mathbf{y}_k-v \bar y_k\|_{\mathbf{R}}
	+ \sigma_{\mathbf{R}}\|\boldsymbol{\gamma}_k \|_{2} \|v\|_{\mathbf{R}}\| \bar y_k\|_{2}\\
	&\leq \sigma_{\mathbf{R}}\left(1+ \hat\gamma_k\|v\|_{\mathbf{R}}L_k/\sqrt{m}\right) \left\|\mathbf{x}_k-\mathbf{1}\bar x_k\right\|_{\mathbf{R}}
	+ \sigma_{\mathbf{R}}\hat\gamma_k\delta_{\mathbf{R},\mathbf{C}} \|\mathbf{y}_k-v \bar y_k\|_{\mathbf{C}}+ \sigma_{\mathbf{R}}\hat\gamma_k L_k\|v\|_{\mathbf{R}}\left\|\bar x_k-x^*_{\lambda_k}\right\|_2.
	\end{align*}
	Adding and subtracting $x^*_{\lambda_{k-1}}$ and using Lemma \ref{lemma:IR-props}, we obtain the desired inequality.
	
	\noindent Third, we show $\Delta_{3,k+1}\leq \sum_{j=1}^3 H_{3j,k}\Delta_{j,k} +h_{3,k} $.  From \eqref{alg:IRPP_compact2} and the definition of $\mathbf{G}_k(\mathbf{x})$ in Def. \ref{eqn:defs}, we obtain $\mathbf{y}_{k+1} = \mathbf{C}\mathbf{y}_k +\mathbf{G}_{k+1}\left(\mathbf{x}_{k+1}\right)-\mathbf{G}_{k}\left(\mathbf{x}_{k}\right).$ Multiplying both sides of the preceding relation by $\tfrac{1}{m}\mathbf{1}^{\top}$ and using the definition of $\bar y_k$ in Def. \ref{eqn:defs}, we obtain that $\bar y_{k+1} = \bar y_{k}+\tfrac{1}{m}\mathbf{1}^{\top}\mathbf{G}_{k+1}\left(\mathbf{x}_{k+1}\right)-\tfrac{1}{m}\mathbf{1}^{\top}\mathbf{G}_{k}\left(\mathbf{x}_{k}\right)$. From the last two relations, we have
	\begin{align*}
	&\mathbf{y}_{k+1} - v \bar y_{k+1} = \left(\mathbf{C}-v\mathbf{1}^{\top}/m\right)\left(\mathbf{y}_k-v \bar y_k\right) + \left(\mathbf{I}-v\mathbf{1}^{\top}/m\right)\left(\mathbf{G}_{k+1}\left(\mathbf{x}_{k+1}\right)-\mathbf{G}_{k}\left(\mathbf{x}_{k}\right)\right).
	\end{align*}
	Invoking Lemma \ref{lemma:norms_relations}, $\mathbf{G}_k(\mathbf{x})$ in Def. \ref{eqn:defs} and $c_0$, and we obtain
	\begin{align}
	&\left\|\mathbf{y}_{k+1} - v \bar y_{k+1}\right\|_{\mathbf{C}} \leq \sigma_{\mathbf{C}}\left\|\mathbf{y}_k-v \bar y_k\right\|_{\mathbf{C}}  + c_0\left\|\mathbf{G}_{k+1}\left(\mathbf{x}_{k+1}\right)-\mathbf{G}_{k}\left(\mathbf{x}_{k}\right)\right\|_2\notag\\
	&\leq \sigma_{\mathbf{C}}\left\|\mathbf{y}_k-v \bar y_k\right\|_{\mathbf{C}} + c_0\left\|\lambda_{k+1}\nabla \mathbf{f}(\mathbf{x}_{k})-\lambda_{k}\nabla \mathbf{f}(\mathbf{x}_{k})\right\|_2
	+c_0\left\|\mathbf{G}_{k+1}\left(\mathbf{x}_{k+1}\right) -\yq{\mathbf{F}(\mathbf{x}_k)}  -\lambda_{k+1}\nabla \mathbf{f}(\mathbf{x}_k)\right\|_2\notag\\
	&\leq \sigma_{\mathbf{C}}\left\|\mathbf{y}_k-v \bar y_k\right\|_{\mathbf{C}}+c_0\left|1- \lambda_{k+1}/\lambda_{k}\right|\left\|\lambda_k\nabla \mathbf{f}(\mathbf{x}_{k})\right\|_2
	+ c_0L_{k}\left\|\mathbf{x}_{k+1}-\mathbf{x}_k\right\|_2.\label{ineq:main_lemma_conv_part_c}
	\end{align}
	From Lemma \ref{lemma:IR-props}, there exists a scalar $B_{\fy{F}}<\infty$ such that $\yqo{L_F}\|\mathbf{1}x^*_{\lambda_k}-\mathbf{1}x^*\|_2\leq B_{\fy{F}}$ \fy{for all $k\geq 0$}. \yq{Given that $ F(x^*)=0$, we have}
	\begin{align*}
	& \|\lambda_k\nabla \mathbf{f}(\mathbf{x}_{k})\|_2 \leq \|\yq{\mathbf{F}(\mathbf{x}_{k})}+\lambda_k\nabla \mathbf{f}(\mathbf{x}_{k})\|_2
	 +\|\yq{\mathbf{F}(\mathbf{x}_{k})-\mathbf{F}(\mathbf{1}x^*)}\|_2\\
	& \leq \|\yq{\mathbf{F}(\mathbf{x}_{k})}+\lambda_k\nabla \mathbf{f}(\mathbf{x}_{k}) - \yq{ \mathbf{F}(\mathbf{1}x^*_{\lambda_k})}-\lambda_k\nabla \mathbf{f}(\mathbf{1}x^*_{\lambda_k}) \|_2 
	+ \yq{L_F}\|\mathbf{x}_{k}-\mathbf{1}x^*\|_2 \\
	&\leq (L_k+L_F)\|\mathbf{x}_{k}-\mathbf{1}x^*_{\lambda_k}\|_2+ L_F\|\mathbf{1}x^*_{\lambda_k}-\mathbf{1}x^*\|_2 
	\leq 2L_k\left(\|\mathbf{x}_{k}-\mathbf{1}\bar x_{k}\|_2+\|\mathbf{1}\bar x_{k}-\mathbf{1}x^*_{\lambda_k}\|_2\right)+ B_{\fy{F}} \\
	&\leq 2L_k\|\mathbf{x}_{k}-\mathbf{1}\bar x_{k}\|_2+2\sqrt{m}L_k\|\bar x_{k}-x^*_{\lambda_k}\|_2+ B_{\fy{F}}.
	\end{align*}
	From row-stochasticity of $\mathbf{R}$, we have $\left(\mathbf{R}-\mathbf{I}\right)\mathbf{1}\bar x_k =0$. Thus, from Lemma \ref{lemma:grad_track_props} we have
	\begin{align*}
	&\left\|\mathbf{x}_{k+1}-\mathbf{x}_k\right\|_2 =  \left\|\mathbf{R}\left(\mathbf{x}_k-\boldsymbol{\gamma}_k \mathbf{y}_k\right)- \mathbf{x}_k\right\|_2 =  \left\| \left(\mathbf{R}-\mathbf{I}\right)\left(\mathbf{x}_k-\mathbf{1}\bar x_k \right)-\mathbf{R}\boldsymbol{\gamma}_k \mathbf{y}_k\right\|_2\\
	& \leq  \left\| \mathbf{R}-\mathbf{I}\right\|_2\left\|\mathbf{x}_k-\mathbf{1}\bar x_k \right\|_2+\|\mathbf{R}\|_2\|\boldsymbol{\gamma}_k \|_2(\|\mathbf{y}_k-v\bar y_k\|_2 +\|v\|_2\|\bar y_k-\bar g_k\|_2+\|v\|_2\|\bar g_k\|_2)\\ 
	&\leq   \left\| \mathbf{R}-\mathbf{I}\right\|_2\left\|\mathbf{x}_k-\mathbf{1}\bar x_k \right\|_2
	+\hat\gamma_k\|\mathbf{R}\|_2\left(\|\mathbf{y}_k-v\bar y_k\|_2 +L_k \|v\|_2\left(\|\mathbf{x}_k-\mathbf{1}\bar x_k\|_2/\sqrt{m}+\|\bar x_k-x^*_{\lambda_k}\|_2\right)\right).
	\end{align*}
	It suffices to find a recursive bound for the term $\left\|\mathbf{x}_{k+1}-\mathbf{x}_k\right\|_2$. From Lemma \ref{lemma:IR-props}, we \fy{may} write
$
	\|\bar x_k-x^*_{\lambda_k}\|_2 \leq  \|\bar x_k-x^*_{\lambda_{k-1}}\|_2+\|x^*_{\lambda_{k-1}}-x^*_{\lambda_k}\|_2  \leq \|\bar x_k-x^*_{\lambda_{k-1}}\|_2+  M/\mu_f\Lambda_{k-1}.
$

	From \eqref{ineq:main_lemma_conv_part_c} \fy{and} the preceding three relations, we can obtain the desired inequality.
\end{proof}
Next, we derive a unifying recursive bound for the three error bounds introduced earlier. 

\begin{proposition}\label{prop:recursive_bound_for_rate} Consider Algorithm \ref{algorithm:IR-push-pull}. Let Assumptions \ref{assum:problem}, \ref{assum:RC}, and \ref{assum:update_rules} hold \fy{where $\tau := 0.5\theta$. Suppose we have \yq{$\Gamma\geq \fy{\hat \Gamma_3 \triangleq }
	\max\Bigl\{\fy{\hat \Gamma_2}, \sqrt[a]{\frac{\mu_f\fy{\lambda}\fy{\hat\gamma}u^{\top}v+2\fy{\hat\gamma}\|v\|_\mathbf{R} L_0}{(1-\sigma_\mathbf{R})\sqrt{m}}},
	\sqrt[a]{\frac{\mu_f\fy{\lambda}\fy{\hat\gamma}u^{\top}v+2\fy{\hat\gamma}c_0L_{0}\|\mathbf{R}\|_2}{1-\sigma_\mathbf{C}}},\\
	\sqrt[a+b]{\tfrac{0.5\mu_f \fy{\lambda} \fy{\hat\gamma}u^{\top}v}{m}}, 
	\sqrt[2a-b]{\tfrac{3c_1\fy{\hat\gamma}^2}{c_3\fy{\lambda}}}, \sqrt[a-b]{\tfrac{3c_2\fy{\hat\gamma}}{c_3\fy{\lambda}}}, \sqrt[1-b]{\tfrac{3c_4}{c_3\fy{\lambda}}}\Bigl\}$.}} Then, \fy{for any $k\geq 1$}, the following holds:

	\noindent \yq{(i)} $\|\Delta_{k+1}\|_2 \leq (1-0.5\mu_f\alpha_k\lambda_k)\|\Delta_k\|_2 +\Theta\Lambda_{k-1}$, where 
$
	\Theta \triangleq \max\left\{1,\sigma_{\mathbf{R}}\fy{\hat\gamma} L_0\|v\|_{\mathbf{R}},c_0L_{0}\left(\fy{\hat\gamma}\|\mathbf{R}\|_2 \|v\|_2 L_0+\sqrt{m}\Lambda_{0}+\mu_f c_0B_{\fy{F}}/M\right)\right\}\sqrt{3}M/\mu_f.
$
	
	\noindent \yq{(ii)}    $\|\Delta_{k}\|_2 \leq \yq{\frac{\mathscr{B}}{(k+\Gamma-1)^{1-a-b}}}$ \fy{where 
	$\textstyle{\mathscr{B}\triangleq \max\{\yq{(\Gamma+1)^{1-a-b}\|\Delta_{1}\|_2},\frac{4\Theta}{\mu_f\fy{\lambda}\fy{\hat\gamma} \theta}\}.}$}
\end{proposition}
\begin{proof}
	\noindent \yq{(i)} In the first step, we consider Proposition \ref{prop:main_ineq_1}. 
	Let us define the sequence $\{\rho_k\}$ as $\rho_k \triangleq 1-0.5\mu_f\alpha_k  \lambda_k$ for $k\geq 0$. Next, we utilize our assumptions to find suitable upper bounds for some of the above terms. We define $\hat H_k =[\hat H_{ij,k}]_{3\times 3}$ and $\hat h_k =[\hat h_{i,k}]_{3\times 1}$ as follows:
	\begin{align*}
	&\hat H_{11,k} := H_{11,k} ,\ 
	\hat H_{12,k} := \bz{\tfrac{\alpha_k L_0}{\sqrt{m}}}, \  
	\hat H_{13,k} :=H_{13,k} ,\
	\hat H_{21,k} :=  \sigma_{\mathbf{R}}\hat\gamma_k L_0\|v\|_{\mathbf{R}},\  
	\hat H_{22,k} := \rho_k-\bz{\tfrac{1-\sigma_{\mathbf{R}}}{2}},\\
	& \bz{\hat H_{23,k} =  H_{23,k}}, \ \bz{\hat H_{31,k} := c_0L_{0}(\hat\gamma_k\|\mathbf{R}\|_2 \|v\|_2L_0+2\sqrt{m}\Lambda_{k}),}\\
	&\hat H_{32,k} := c_0L_{0}(\| \mathbf{R}-\mathbf{I}\|_2 +\hat\gamma_k\|\mathbf{R}\| \|v\|_2 \bz{\tfrac{L_0}{\sqrt{m}}} + 2\Lambda_{0}),\
	\hat H_{33,k} := \rho_k-\bz{\tfrac{1-\sigma_{\mathbf{C}}}{2}},\\
	&\hat h_{1,k}:=\bz{\tfrac{\Theta}{\sqrt{3}}\Lambda_{k-1}}, \quad 
	\hat h_{2,k}:=\bz{\tfrac{\Theta}{\sqrt{3}}\Lambda_{k-1}},\quad 
	\hat h_{3,k}:=\bz{\tfrac{\Theta}{\sqrt{3}}\Lambda_{k-1}}.
	\end{align*}
	Note that we have
$
	\hat H_{22,k}- H_{22,k} = 1-0.5\mu_f\alpha_k\lambda_k - \frac{1-\sigma_{\mathbf{R}}}{2}-\sigma_{\mathbf{R}}\left(1+\hat\gamma_k\|v\|_{\mathbf{R}}\frac{L_k}{\sqrt{m}}\right)
	=\tfrac{1-\sigma_{\mathbf{R}}}{2} -0.5\mu_f\alpha_k\lambda_k-\hat\gamma_k\|v\|_{\mathbf{R}}\tfrac{L_k}{\sqrt{m}}.
$
	
	\yq{Next we show that for $\Gamma \geq \sqrt[a]{\frac{\mu_f\fy{\lambda}\fy{\hat\gamma}u^{\top}v+2\fy{\hat\gamma}\|v\|_\mathbf{R} L_0}{(1-\sigma_\mathbf{R})\sqrt{m}}}$, we have $H_{22,k}\leq \hat H_{22,k}$ for all $k\geq 1$. By using Lemma \ref{lemma:alpha_k:bound}, it suffices to show that 
	\begin{align}\label{prop3:Gamma1}
	\tfrac{\frac{1-\sigma_{\mathbf{R}}}{2} -\tfrac{\fy{\hat\gamma}}{(k+\Gamma)^a}\|v\|_{\mathbf{R}}\frac{L_k}{\sqrt{m}}}{0.5\mu_f\tfrac{\fy{\lambda}}{(k+\Gamma)^b}} \geq \tfrac{\fy{\hat\gamma}u^{\top}v}{m(k+\Gamma)^a}.
	\end{align}
	\fy{Consider} $\Gamma \geq \sqrt[a]{\frac{\mu_f\fy{\lambda}\fy{\hat\gamma}u^{\top}v+2\fy{\hat\gamma}\|v\|_\mathbf{R} L_0}{(1-\sigma_\mathbf{R})\sqrt{m}}}$. \fy{By} rearranging terms and \fy{invoking} $L_0 \geq L_k$, \fy{we obtain}
	\bz{
$
		\tfrac{1-\sigma_\mathbf{R}}{2} \geq \tfrac{0.5\mu_f\fy{\lambda}\fy{\hat\gamma}u^{\top}v+\fy{\hat\gamma}\|v\|_\mathbf{R} L_0}{\Gamma^a\sqrt{m}}
		\geq \tfrac{0.5\mu_f\fy{\lambda}\fy{\hat\gamma}u^{\top}v+\fy{\hat\gamma}\|v\|_\mathbf{R} L_k}{(k+\Gamma)^a\sqrt{m}}.
$
	}
	
	This relation can be further written as
	\bz{$\frac{1-\sigma_\mathbf{R}}{2} \geq \frac{0.5\mu_f\fy{\lambda}\fy{\hat\gamma}u^{\top}v}{(k+\Gamma)^{a+b}m}+\frac{\fy{\hat\gamma}\|v\|_\mathbf{R} L_k}{(k+\Gamma)^a\sqrt{m}}.$}
	By rearranging terms and dividing both sides by $\frac{1}{(k+\Gamma)^{b}}$, we obtain
	\bz{$\frac{\frac{1-\sigma_\mathbf{R}}{2}-\frac{\fy{\hat\gamma}\|v\|_\mathbf{R} L_k}{(k+\Gamma)^a\sqrt{m}}}{\frac{1}{(k+\Gamma)^{b}}} \geq \frac{0.5\mu_f\fy{\lambda}\fy{\hat\gamma}u^{\top}v}{(k+\Gamma)^{a}m}$},
	which implies \eqref{prop3:Gamma1}. Similarly, we can prove that for $\Gamma \geq \sqrt[a]{\frac{\mu_f\fy{\lambda}\fy{\hat\gamma}u^{\top}v+2\fy{\hat\gamma}c_0L_{0}\|\mathbf{R}\|_2}{1-\sigma_\mathbf{C}}}$, we have $H_{33,k}\leq \hat H_{33,k}$ for all $k\geq 1$. Recall that $H_{33,k} =\textstyle{\sigma_{\mathbf{C}} +c_0L_{k}\hat\gamma_k \|\mathbf{R}\|_2}$ and $\hat H_{33,k} = 1-0.5\mu_f\alpha_k\lambda_k-\frac{1-\sigma_{\mathbf{C}}}{2}$. \fy{Then,} by invoking Lemma \ref{lemma:alpha_k:bound}, it suffices to show that
	\begin{align}\label{prop3:Gamma2}
	\tfrac{\frac{1-\sigma_{\mathbf{C}}}{2} -\tfrac{\fy{\hat\gamma}}{(k+\Gamma)^a}c_0L_k\|\mathbf{R}\|_2}{0.5\mu_f\tfrac{\fy{\lambda}}{(k+\Gamma)^b}} \geq \tfrac{\fy{\hat\gamma}u^{\top}v}{m(k+\Gamma)^a}.
	\end{align}
	\fy{Recall that} $\Gamma \geq \sqrt[a]{\frac{\mu_f\fy{\lambda}\fy{\hat\gamma}u^{\top}v+2\fy{\hat\gamma}c_0L_{0}\|\mathbf{R}\|_2}{1-\sigma_\mathbf{C}}}$. \fy{By} rearranging terms and using  $L_0 \geq L_k$ \fy{once again, we} obtain
	\bz{ $\frac{1-\sigma_\mathbf{C}}{2} \geq \frac{0.5\mu_f\fy{\lambda}\fy{\hat\gamma}u^{\top}v+\fy{\hat\gamma}c_0L_{0}\|\mathbf{R}\|_2}{\Gamma^a} \geq \frac{0.5\mu_f\fy{\lambda}\fy{\hat\gamma}u^{\top}v+\fy{\hat\gamma}c_0L_{k}\|\mathbf{R}\|_2}{(k+\Gamma)^a}.$ }
	This relation can be further written as
	\bz{ $\frac{1-\sigma_\mathbf{C}}{2} \geq \frac{0.5\mu_f\fy{\lambda}\fy{\hat\gamma}u^{\top}v}{(k+\Gamma)^{a+b}m}+\frac{\fy{\hat\gamma}c_0L_{k}\|\mathbf{R}\|_2}{(k+\Gamma)^a}.$ }
	By rearranging terms and dividing both sides by $\frac{1}{(k+\Gamma)^{b}}$, we obtain
$\tfrac{\frac{1-\sigma_\mathbf{C}}{2}-\frac{\fy{\hat\gamma}c_0L_{k}\|\mathbf{R}\|_2}{(k+\Gamma)^a}}{\frac{1}{(k+\Gamma)^{b}}} \geq \tfrac{0.5\mu_f\fy{\lambda}\fy{\hat\gamma}u^{\top}v}{(k+\Gamma)^{a}m},$
	which implies \eqref{prop3:Gamma2}. Thus, by taking to account that $\Lambda_k$ is \fy{a} nonincreasing sequence and invoking the definition of $\Theta$, we have $H_k\leq \hat H_k$ and $h_{k}\leq \hat h_{k}$. Then by using the conclusion from Proposition \ref{prop:main_ineq_1}, for $\Gamma\geq \fy{\hat\Gamma_3}$, we have $\Delta_{k+1} \leq \hat H_k\Delta_k+\hat h_k$ for all $k \geq 1$. }
Consequently, we obtain
\begin{align}\label{eqn:rec_Delta_proof}
\left\|\Delta_{k+1}\right\|_2 \leq \rho(\hat H_k)\left\|\Delta_k\right\|_2+\Theta \Lambda_{k-1},
\end{align}
where $\rho(\hat H_k)$ denotes the spectral norm of $\hat H_k$. Next, we show that for a \yq{suitable choice of $\Gamma$}, we have $\rho(\hat H_k) \leq \rho_k $ \yq{for all $k \geq 1$}. To show this relation, employing Lemma 5 in \cite{pu2021distributed}, it suffices to show that $0\leq \hat H_{ii,k} <\rho_k$ for $i \in \{1,2,3\}$ and $\text{det}(\rho_k\mathbf{I}-\hat H_k)>0$. Among these, it can be easily seen that \yq{$\hat H_{ii,k} <\rho_k$} holds for all $i \in \{1,2,3\}$.
\yq{
We then show that for $\Gamma \geq \sqrt[a+b]{\tfrac{0.5\mu_f \fy{\lambda} \fy{\hat\gamma}u^{\top}v}{m}}$, we have $\hat H_{11,k}= 1-0.5\mu_f\alpha_k  \lambda_k\geq 0$ for all $k\geq 1$. \fy{From} $\Gamma \geq \sqrt[a+b]{\tfrac{0.5\mu_f \fy{\lambda} \fy{\hat\gamma}u^{\top}v}{m}}$, we obtain $1 \geq \tfrac{0.5\mu_f \fy{\lambda} \fy{\hat\gamma}u^{\top}v}{m\Gamma^{a+b}}$, \fy{implying that} $1 - \tfrac{0.5\mu_f \fy{\lambda} \fy{\hat\gamma}u^{\top}v}{m(k+\Gamma)^{a+b}}\geq 0$. \fy{In view of the definition of} $\lambda_k$, we have $1 - \tfrac{0.5\mu_f \fy{\hat\gamma}u^{\top}v}{m(k+\Gamma)^{a}}\lambda_k \geq 0$. \fy{Then,} invoking Lemma \ref{lemma:alpha_k:bound}, we obtain $1 - 0.5\mu_f \alpha_k \lambda_k \geq 0$.
	Also, $\hat H_{22,k} = 1 - 0.5\mu_f \alpha_k \lambda_k-\tfrac{1-\sigma_\mathbf{R}}{2} = \tfrac{1}{2} - 0.5\mu_f \alpha_k \lambda_k + \tfrac{\sigma_\mathbf{R}}{2} = \tfrac{\hat H_{11,k}}{2}  + \tfrac{\sigma_\mathbf{R}}{2} \geq 0$. Similarly, $\hat H_{33,k} = 1 - 0.5\mu_f \alpha_k \lambda_k-\tfrac{1-\sigma_\mathbf{C}}{2} = \tfrac{\hat H_{11,k}}{2}  + \tfrac{\sigma_\mathbf{C}}{2} \geq 0$. Note that $\hat H_{22,k}$ and $\hat H_{33,k}$ are nonnegative \fy{due to} $\Gamma \geq \sqrt[a+b]{\tfrac{\mu_f \fy{\lambda} \fy{\hat\gamma}u^{\top}v}{m}}$.
}
\fy{In the following}, we show \fy{that} $\text{det}(\rho_k\mathbf{I}-\hat H_k)>0$. \yq{We have}
\begin{align*}
	&\text{det}(\rho_k\mathbf{I}-\hat H_k) 
	 \textstyle{= (0.5\mu_f\alpha_k \lambda_k)( \frac{1-\sigma_{\mathbf{R}}}{2})( \frac{1-\sigma_{\mathbf{C}}}{2})}
	\textstyle{- (0.5\mu_f\alpha_k \lambda_k)(\sigma_{\mathbf{R}}\hat\gamma_k\delta_{\mathbf{R},\mathbf{C}})c_0L_{0}(\| \mathbf{R}-\mathbf{I}\|_2+ 2\Lambda_{0}} \\ 
	&\textstyle{ +\hat\gamma_k\|\mathbf{R}\| \|v\|_2 \frac{L_0}{\sqrt{m}} )-(\frac{1-\sigma_{\mathbf{C}}}{2})( \frac{\alpha_k   L_0}{\sqrt{m}})( \sigma_{\mathbf{R}}\hat\gamma_k L_0\|v\|_{\mathbf{R}})}
	\textstyle{ -(\frac{\alpha_k   L_0}{\sqrt{m}})(\sigma_{\mathbf{R}}\hat\gamma_k\delta_{\mathbf{R},\mathbf{C}})(c_0L_{0}(\hat\gamma_k\|\mathbf{R}\|_2 \|v\|_2L_0}\\
	& \textstyle{+2\sqrt{m}\Lambda_{k}))-(\frac{\hat\gamma_k\|u\|_2}{m})(\sigma_{\mathbf{R}}\hat\gamma_k L_0\|v\|_{\mathbf{R}})(c_0L_{0}(\| \mathbf{R}-\mathbf{I}\|_2} 
	\textstyle{+\hat\gamma_k\|\mathbf{R}\| \|v\|_2 \frac{L_0}{\sqrt{m}}   + 2\Lambda_{0}))}\\
	& - (\tfrac{1-\sigma_{\mathbf{R}}}{2})(c_0L_{0}(\hat\gamma_k\|\mathbf{R}\|_2 \|v\|_2L_0+2\sqrt{m}\Lambda_{k}))(\tfrac{\hat\gamma_k\|u\|_2}{m}).
\end{align*}
	Next, we find lower and upper bounds on $\alpha_k$ in terms of $\hat \gamma_k$. Assumption \ref{assum:update_rules} provides $\theta\hat \gamma_k$ as a lower bound for $\alpha_k$ \yq{while Lemma \ref{lemma:alpha_k:bound} provides its upper bound.} 
	Let us define $\bar \theta = \tfrac{1}{m}u^{\top}v$. Thus, we have $\theta\hat \gamma_k \leq \alpha_k \leq \bar \theta \hat \gamma_k$ for all $k\geq 0$. Using these bounds and rearranging the terms, \yq{we can obtain}
$
	\text{det}(\rho_k\mathbf{I}-\hat H_k) \geq  -c_1\hat\gamma_k^3-c_2\hat\gamma_k^2+c_3\hat\gamma_k\lambda_k-c_4\hat\gamma_k\Lambda_k,
$
	where the scalars $c_1$, $c_2$, $c_3$ are defined as below:
	\begin{align*}
	c_1 &\textstyle{\triangleq (0.5\mu_f\bar \theta \fy{\lambda})(\sigma_{\mathbf{R}}\delta_{\mathbf{R},\mathbf{C}})c_0L_{0}(\|\mathbf{R}\| \|v\|_2 \frac{L_0}{\sqrt{m}})}
	\textstyle{+(\tfrac{\bar \theta   L_0}{\sqrt{m}})\left(\sigma_{\mathbf{R}}\delta_{\mathbf{R},\mathbf{C}}\right)\left(c_0L_{0}\left(\|\mathbf{R}\|_2 \|v\|_2L_0\right)\right)}\\
	& \textstyle{+(\frac{\|u\|_2}{m})(\sigma_{\mathbf{R}} L_0\|v\|_{\mathbf{R}})(c_0L_{0}\|\mathbf{R}\| \|v\|_2 \frac{L_0}{\sqrt{m}})},\\
	c_2 &\textstyle{\triangleq (0.5\mu_f\bar \theta \fy{\lambda})(\sigma_{\mathbf{R}}\delta_{\mathbf{R},\mathbf{C}})c_0L_{0}(\| \mathbf{R}-\mathbf{I}\|_2 + 2\Lambda_{0})}
	\textstyle{+(\frac{1-\sigma_{\mathbf{C}}}{2})( \frac{\bar \theta  L_0}{\sqrt{m}})( \sigma_{\mathbf{R}} L_0\|v\|_{\mathbf{R}})}\\
	&\textstyle{+(\frac{\bar \theta   L_0}{\sqrt{m}})(\sigma_{\mathbf{R}}\delta_{\mathbf{R},\mathbf{C}})(c_0L_{0}2\sqrt{m}\Lambda_{0})}
	\textstyle{+(\frac{\|u\|_2}{m})(\sigma_{\mathbf{R}} L_0\|v\|_{\mathbf{R}})(c_0L_{0}(\| \mathbf{R}-\mathbf{I}\|_2  + 2\Lambda_{0}))}\\
	&\textstyle{+(\frac{1-\sigma_{\mathbf{R}}}{2})(c_0L_{0}(\mathbf{R}\|_2 \|v\|_2L_0))(\frac{\|u\|_2}{m})},\\
	c_3 &\textstyle{\triangleq (0.5)^3\mu_f\theta(1-\sigma_{\mathbf{R}})( 1-\sigma_{\mathbf{C}})},\ 
	c_4 \textstyle{\triangleq(\frac{1-\sigma_{\mathbf{R}}}{2})(c_0L_{0}\sqrt{m})(\frac{\|u\|_2}{m}).}
	\end{align*}
	It suffices to show that $-c_1\hat\gamma_k^3-c_2\hat\gamma_k^2+c_3\hat\gamma_k\lambda_k-c_4\hat\gamma_k\Lambda_k>0$ for any \yq{$k\geq 1$.} \yq{Or equivalently, we show that $\tfrac{1}{3}c_3\tfrac{\fy{\lambda}}{(k+\Gamma)^{b}} > c_1\tfrac{\fy{\hat\gamma}^2}{(k+\Gamma)^{2a}}$, $\tfrac{1}{3}c_3\tfrac{\fy{\lambda}}{(k+\Gamma)^{b}} > c_2\tfrac{\fy{\hat\gamma}}{(k+\Gamma)^{a}}$ and $\tfrac{1}{3}c_3\tfrac{\fy{\lambda}}{(k+\Gamma)^{b}} >c_4\Lambda_k$ for $k\geq 1$.
	
	Let $\Gamma > \sqrt[2a-b]{\tfrac{3c_1\fy{\hat\gamma}^2}{c_3\fy{\lambda}}}$, we obtain $(k+\Gamma)^{2a-b}>\tfrac{3c_1\fy{\hat\gamma}^2}{c_3\fy{\lambda}}$, then by rearranging terms, we obtain $\tfrac{c_3\fy{\lambda}}{3(k+\Gamma)^{b}} > \tfrac{c_1\fy{\hat\gamma}^2}{(k+\Gamma)^{2a}}$.
	Similarly ,we can obtain $\tfrac{c_3\fy{\lambda}}{3(k+\Gamma)^{b}} > \tfrac{c_2\fy{\hat\gamma}}{(k+\Gamma)^{a}}$ by letting $\Gamma > \sqrt[a-b]{\tfrac{3c_2\fy{\hat\gamma}}{c_3\fy{\lambda}}}$.
	Next, we choose $\Gamma > \sqrt[1-b]{\tfrac{3c_4}{c_3\fy{\lambda}}}$, we can obtain $(k+\Gamma)^{1-b}>\tfrac{3c_4}{c_3\fy{\lambda}}$, then by rearranging terms, we obtain $\tfrac{c_3\fy{\lambda}}{3(k+\Gamma)^{b}} > \tfrac{c_4}{k+\Gamma}$, then utilize Lemma \ref{lem:update_rules_props}(ii), we obtain $\tfrac{c_3\fy{\lambda}}{3(k+\Gamma)^{b}} > \tfrac{c_4}{k+\Gamma} \geq c_4\Lambda_{k-1} \geq c_4\Lambda_k$ for all $k\geq 1$.
	}
\bz{We conclude that for \fy{$\Gamma \geq \hat{\Gamma}_3$}, we have $\text{det}(\rho_k\mathbf{I}-\hat H_k)>0$ for any \yq{$k \geq 1$}.}
	Therefore, we have $\rho(\hat H_k) \leq  1-0.5\mu_f\alpha_k  \lambda_k$ for all \yq{$k \geq 1$}. The desired inequality is obtained from this \fy{relation} and the relation \eqref{eqn:rec_Delta_proof}.
	
	\noindent \yq{(ii)} From Lemma \ref{lem:update_rules_props}, we have that $\Lambda_{k-1} \leq \yq{\tfrac{1}{k+\Gamma}}$. From part (i) \yq{and invoking $\alpha_k \geq \theta\hat\gamma_k$}, we obtain for all \yq{$k\geq 1$},
	\begin{align}\label{eqn:rec_Delta_proof_v2}
	\textstyle{\|\Delta_{k+1}\|_2 \leq (1-0.5\mu_f\yq{\lambda_k}\hat \gamma_k\theta)\|\Delta_k\|_2 +\frac{\Theta}{\yq{k+\Gamma}}.}
	\end{align}
	We use induction to show that the desired relation holds for \fy{$\mathscr{B}$.} \yq{First, it can be easily verified that the inequality holds for $k=1$. Let us assume that $\|\Delta_{k}\|_2 \leq \frac{\mathscr{B}}{(k+\Gamma-1)^{1-a-b}}$ holds for some $k \geq 2$. Next, we show that this relation also holds for $k+1$. Consider Lemma \ref{lem:update_rules_props}. \fy{Let} $\tau = 0.5\theta$. \fy{We} obtain $\frac{(k+\Gamma)\hat\gamma_k\lambda_k}{(k+\Gamma-1)\hat\gamma_{k-1}\lambda_{k-1}}\leq 1+0.25\mu_f \hat\gamma_k\lambda_k\theta$ for all $k\geq 1$. Using \fy{the} definition of $\hat \gamma_k$ and $\lambda_k$ and rearranging terms, we obtain
	\begin{align}\label{eqn:ineq1_for_induction}
	\textstyle{\frac{1}{(k+\Gamma-1)^{1-a-b}}\leq \frac{1}{(k+\Gamma)^{1-a-b}}(1+0.25\mu_f\fy{\lambda}\fy{\hat\gamma}\theta).}
	\end{align}
	\fy{Consider} $\Gamma > \sqrt[a+b]{0.5\mu_f\fy{\lambda}\fy{\hat \gamma}\theta}$. \fy{We} obtain $1>0.5\mu_f\lambda_k\hat \gamma_k\theta$. From \eqref{eqn:rec_Delta_proof_v2} and the induction hypothesis, we obtain
$
	\textstyle{\|\Delta_{k+1}\|_2 \leq (1-0.5\mu_f\lambda_k\hat \gamma_k\theta) \frac{\mathscr{B}}{(k+\Gamma-1)^{1-a-b}} +\frac{\Theta}{k+\Gamma}.}
$

	From the preceding relation and \eqref{eqn:ineq1_for_induction}, we obtain
$
	\|\Delta_{k+1}\|_2 \leq   \tfrac{\mathscr{B}(1-0.5\mu_f\lambda_k\hat \gamma_k\theta)(1+0.25\mu_f\lambda_k\hat \gamma_k\theta)}{(k+\Gamma)^{1-a-b}} +\tfrac{\Theta}{k+\Gamma}.
$

	From the definition of $\mathscr{B}$, we have $\Theta \leq 0.25\mu_f \fy{\lambda} \fy{\hat\gamma}\theta\mathscr{B}$. Therefore, we obtain
$
	\|\Delta_{k+1}\|_2 \leq  \tfrac{\mathscr{B}(1-0.25\mu_f \fy{\lambda} \fy{\hat\gamma}\theta -0.125(\mu_f \fy{\lambda} \fy{\hat\gamma}\theta)^2)}{(k+\Gamma)^{1-a-b}}  + \tfrac{0.25\mu_f \fy{\lambda} \fy{\hat\gamma}\theta\mathscr{B}}{(k+\Gamma)^{1-a-b}}
	= \tfrac{\mathscr{B}(1 -0.125(\mu_f \fy{\lambda} \fy{\hat\gamma}\theta)^2)}{(k+\Gamma)^{1-a-b}}.
$
	We have $1>0.5\mu_f\lambda_k\hat \gamma_k\theta$. \fy{This implies that $\|\Delta_{k+1}\|_2 \leq \frac{\mathscr{B}}{(k+\Gamma)^{1-a-b}}$. Thus, the induction statement holds for $k+1$ and hence, the proof is completed.}}
	\end{proof}
	Our first main result is provided below where we provide \fy{convergence guarantees for addressing problem}~\eqref{eqn:bilevel_problem}. 
	\begin{theorem}[\fy{Convergence statements for \eqref{eqn:bilevel_problem} over digraphs}]\label{thm:rate_ana}
	Consider problem \eqref{eqn:bilevel_problem} and Algorithm \ref{algorithm:IR-push-pull}. Let Assumptions \ref{assum:problem}, \ref{assum:RC}, and \ref{assum:update_rules} hold. \fy{Suppose $\Gamma \geq \hat \Gamma_3$, where $\hat \Gamma_3$ is given by Prop.~\ref{prop:recursive_bound_for_rate}.} Then, \fy{the following results hold:} 
	
	\noindent (i) \fy{[Consensus error bound for $x$] For any $k \geq 1$, we have}
	
	 \yq{$$\|\mathbf{x}_{k+1}-\mathbf{1}\bar x_{k+1}\|_{\mathbf{R}} \leq \fy{\tfrac{\max\{\yq{(\Gamma+1)^{1-a-b}\|\Delta_{1}\|_2},\frac{4\Theta}{\mu_f\fy{\lambda}\fy{\hat\gamma} \theta}\}}{(k+\Gamma-1)^{1-a-b}}}.$$}
	
	\noindent (ii)  \fy{[Consensus error bound for $y$] For any $k \geq 1$, we have}
	
	\yq{$$\|\mathbf{y}_{k+1} - v \bar y_{k+1}\|_{\mathbf{C}} \leq \fy{\tfrac{\max\{\yq{(\Gamma+1)^{1-a-b}\|\Delta_{1}\|_2},\frac{4\Theta}{\mu_f\fy{\lambda}\fy{\hat\gamma} \theta}\}}{(k+\Gamma-1)^{1-a-b}}}.$$ }

	\noindent (iii) \fy{[Asymptotic convergence] The sequence of the averaged iterate, $\{\bar x_{k}\}$,  admits a limit point. It converges to the unique optimal solution of problem~\eqref{eqn:bilevel_problem}}, i.e., $\lim_{k\to\infty}\bar x_{k}=x^*$.

	\end{theorem}
	\begin{proof}
	\noindent \fy{(i, ii) The bound in (i) holds immediately from Prop.~\ref{prop:recursive_bound_for_rate}(ii), in view of $\|\mathbf{x}_k-\mathbf{1}\bar x_k\|_{\mathbf{R}} \leq \|\Delta_k\|_2$. The bound in (ii) can be shown in a similar vein.}
	
	\fy{
	\noindent (iii) From the triangle inequality, for any $k\geq 1$ we may write 
	\begin{align}\label{eqn:thm1_1}
	\|\bar x_{k} - x^*\|_2 \leq \|\bar x_{k} - x^*_{\lambda_{k-1}}\|_2+\|x^*_{\lambda_{k-1}} - x^*\|_2.
\end{align}	 
Recall from Lemma~\ref{lemma:IR-props}(i) that $\lim_{k \to \infty}\ x^*_{\lambda_k} =x^*$. From Proposition~\ref{prop:recursive_bound_for_rate}(ii), we have for any $k \geq 1$, 
	$\|\bar x_{k}-x^*_{\lambda_{k-1}}\|_2 \leq \tfrac{\max\{\yq{(\Gamma+1)^{1-a-b}\|\Delta_{1}\|_2},\frac{4\Theta}{\mu_f\fy{\lambda}\fy{\hat\gamma} \theta}\}}{(k+\Gamma-1)^{1-a-b}}.$
	Taking the limit on both sides of \eqref{eqn:thm1_1} and invoking the preceding results, we obtain $\lim_{k\to\infty}\bar x_{k}=x^*$.
	}
	\end{proof}
\yqp{\begin{remark}
\fy{Importantly, Theorem~\ref{thm:rate_ana} provides convergence guarantees for computing the unique optimal NE.} We observe that the choice of \fy{the parameters} $a$ and $b$ plays an important role in the convergence speed. \fy{Recall that} $a$ and $b$ \fy{need} to satisfy $0<b<a<1$ and $2a+3b<2$. \fy{This implies that the term $1-a-b$ in Theorem~\ref{thm:rate_ana}(i,ii) ranges between $(0,1)$. In particular, with suitable choices of $a \in (0,1)$ and $b \in (0,\yqt{\min}\{a,0.4\})$, the consensus error bounds can achieve a nearly sublinear rate (e.g., $a=0.01$ and $b=0.001$). We should note, however, that a too small $b$ would slow down the convergence speed of $\|x^*_{\lambda_{k-1}} - x^*\|_2$ in \eqref{eqn:thm1_1} to zero.  This indeed shows that in implementations, $b$ should be chosen within the range $(0,\yqt{\min}\{a,0.4\})$ but not too small. This trade-off will be numerically studied in section~\ref{sec:num} as well.}   
\end{remark}
}

\section{Stochastic setting over undirected networks}\label{sec:stoch}
\yqp{The goal in this section is to \fy{devise a gradient tracking method} for addressing the stochastic problem \eqref{eqn:bilevel_problem_stoch} over undirected networks.} \fy{To this end, we consider the DSGT method in~\cite{pu2021distributed} that addresses unconstrained distributed stochastic optimization problems. By leveraging the IR framework, we devise a single-timescale method called IR-DSGT, presented in Algorithm~\ref{algorithm:IR-DSGT}.} In this section, an (undirected) graph is denoted by $\mathcal{G} =\left(\mathcal{N},\mathcal{E}\right)$ where $\mathcal{N}$ is a set of nodes and $\mathcal{E} \subseteq \mathcal{N}\times \mathcal{N}$ is the set of ordered pairs of vertices. We let $\mathcal{N}(i)$ denote the set of neighbors of agent $i$, i.e., $\mathcal{N}(i) \triangleq \{j\mid (i,j) \in \mathcal{E}\}$. Throughout, we consider the definitions given by \eqref{eqn:notation_eqn1}, \eqref{eqn:notation_eqn2}, and \eqref{eqn:notation_eqn3}. Additionally, we define the following terms:
	\begin{align*}
	&\bxi \triangleq [\xi_1,\ \ldots,\ \xi_m]^{\top}\in \mathbb{R}^{m\times d}, \quad f_i(x) \triangleq \EXP{f_i(x,\xi_i)\mid x},\
	\mathbf{f}(\mathbf{x},\bxi)  \triangleq \textstyle\sum\nolimits_{i=1}^m f_i(x_i,\xi_i), \\
	&\nabla \mathbf{f}(\mathbf{x},\bxi) \triangleq [\nabla f(x_1,\xi_1),\ \ldots,\ \nabla f(x_m,\xi_m)]^{\top}  \in \mathbb{R}^{m\times n},\\
	&\yq{F(\mathbf{x},\bxi)  \triangleq \textstyle\sum\nolimits_{i=1}^m F_i(x_i,\xi_i)},\
	\bz{\mathbf{F}(\mathbf{x},\bxi) \triangleq [F_1(x_1,\xi_1),\ \ldots,\ F_m(x_m,\xi_m)]^{\top}  \in \mathbb{R}^{m\times n}}.
	\end{align*}
	Here, we assume \fy{that} each agent $i$ has access to \fy{a stochastic oracle} to obtain \fy{sampled mappings} $\nabla f_i(x,\xi_i)$ and \yq{$F_i(x,\xi_i)$}. Throughout this section, we use $\|\cdot\|$ to denote the Euclidean norm and the Frobenius norm of a vector and a matrix, respectively. We let $\langle\cdot , \cdot\rangle$ denote the Frobenius inner product.
	\begin{algorithm}
	  \caption{Iteratively Regularized distributed stochastic gradient tracking (IR-DSGT)}\label{algorithm:IR-DSGT}
		\begin{algorithmic}[1]
		\STATE\textbf{Input:} Choose $\gamma_{0} > 0$, $\lambda_{0} > 0$, the weight matrix $\mathbf{W}$. For all $i \in [m]$, agents set an arbitrary initial point $x_{i,0} \in \mathbb{R}^n$ and $y_{i,0}:=\yq{F_i(x_{i,0},\xi_{i,0})}+\lambda_0\nabla f_i(x_{i,0},\xi_{i,0})$; 
		\FOR {$k=0,1,\ldots$}
				 \STATE For all $i \in [m]$, agent $i$ receives the vector $x_{j,k}-\gamma_{j,k}y_{j,k}$ from each agent $j \in \mathcal{N}_{\mathbf{W}}(i)$, sends $W_{\ell i}y_{i,k}$ to each agent $\ell \in \mathcal{N}_{\mathbf{W}}(i)$, and does the following updates:
		  \STATE $x_{i,k+1} := \sum\nolimits_{j=1}^m W_{ij}\left(x_{j,k}-\yq{\gamma_{k}}y_{j,k}\right)$
		  \STATE $\begin{aligned}y_{i,k+1} &:= \textstyle\sum\nolimits_{j=1}^m W_{ij}y_{j,k}+\yq{F_i(x_{i,k+1},\xi_{i,k+1})}+\lambda_{k+1}\nabla f_i(x_{i,k+1},\xi_{i,k+1})\\
		  &-\yq{F_i(x_{i,k},\xi_{i,k})} -\lambda_{k}\nabla f_i(x_{i,k},\xi_{i,k})\end{aligned}$
	   \ENDFOR
	   \end{algorithmic}
	\end{algorithm}
	The update rules in Algorithm \ref{algorithm:IR-DSGT} can be compactly represented as the following:
	\begin{align}
	\mathbf{x}_{k+1} :=& \mathbf{W}\left(\mathbf{x}_k-\gamma_k\mathbf{y}_k\right),\label{alg:IRDSGT_compact1}\\
	\mathbf{y}_{k+1} := &\mathbf{W}\mathbf{y}_k+ \yq{\mathbf{F}(\mathbf{x}_{k+1},\bxi_{k+1})}+ \lambda_{k+1}\nabla \mathbf{f}(\mathbf{x}_{k+1},\bxi_{k+1}) 
	-\yq{\mathbf{F}(\mathbf{x}_k,\bxi_k)}  -\lambda_k\nabla \mathbf{f}(\mathbf{x}_k,\bxi_k)\label{alg:IRDSGT_compact2},
	\end{align}
	where ${\gamma}_k> 0$ denotes the \fy{stepsize} at the iteration $k \geq 0$.
	Throughout, we define the history of the method as $\mathscr{F}_k \triangleq \cup_{i=1}^m\{x_{i,0},\xi_{i,1},\ldots,\xi_{i,k-1}\}$ for $k\geq 1$, and $\mathscr{F}_0 \triangleq \cup_{i=1}^m\{x_{i,0}\}$.
	\yq{Next, we present key assumptions regarding Algorithm \ref{algorithm:IR-DSGT}.}
	\begin{assumption}\label{assum:Weight}
	The weight matrix $\mathbf{W}$ is doubly stochastic and we have $w_{i,i}>0$ for all $i \in [m]$. Also, the graph $\mathcal{G}$ corresponding to the communication network \fy{is} undirected and connected.
	\end{assumption}
	\begin{assumption}\label{assum:stoch_oracle}
	For all $i \in [m]$ and $x \in \mathbb{R}^n$, random vectors $\xi_i \in \mathbb{R}^d$ are independent and satisfy the following conditions:
	
	\noindent (a) $\EXP{\nabla f_i(x,\xi_i)\mid x} = \nabla f_i(x),\ \EXP{\yqo{F_i(x,\xi_i)}\mid x} = \yqo{F_i(x)}$. \, (b) $\EXP{\|\nabla f_i(x,\xi_i) - \nabla f_i(x)\|^2\mid x} \leq \nu_f^2$ for some $\sigma_f>0$ and $\EXP{\|\yqo{F_i(x,\xi_i)-F_i(x)}\|^2\mid x} \leq \yqt{\nu_F^2}$ for some $\yqo{\nu_F}>0$. 
	\end{assumption}
	\yq{We use the following definitions for analyzing the convergence of Algorithm \ref{algorithm:IR-DSGT}.}
	\begin{definition}\label{eqn:defs_stoch}
	Let $x^*$, $x^*_{\lambda_k}$, $L_k$, and $\Lambda_k$ be defined as in Definition \ref{eqn:defs}. Also, let us define $\nu_k^2 \triangleq \yqt{\nu_F^2} +\lambda_k^2\nu_f^2$, $\mathbf{G}_k(\mathbf{x},\bxi)\triangleq \yq{\mathbf{F}\left(\mathbf{x},\bxi\right)}+\lambda_k \nabla \mathbf{f}\left(\mathbf{x},\bxi\right)\in \mathbb{R}^{m\times n}$,  $G_k(\mathbf{x},\bxi)\triangleq \tfrac{1}{m}\mathbf{1}^{\top}\mathbf{G}_k(\mathbf{x},\bxi)\in \mathbb{R}^{1\times n}$. We let $\mathbf{G}_k(\mathbf{x})$, $G_k(\mathbf{x})$ denote the expected values of $\mathbf{G}_k(\mathbf{x},\bxi)$ and $G_k(\mathbf{x},\bxi)$, respectively. Also, we define , $\mathscr{G}_k(x) \triangleq G_k\left(\mathbf{1}x^{\top}\right)\ \in \mathbb{R}^{1\times n}$, $\bar g_k \triangleq \mathscr{G}_k(\bar x_k)\ \in \mathbb{R}^{1\times n}$, $\bar x_k \triangleq \tfrac{1}{m}\mathbf{1}^{\top}\mathbf{x}_k\ \in \mathbb{R}^{1\times n}$, and $\bar y_k \triangleq \tfrac{1}{m}\mathbf{1}^{\top}\mathbf{y}_k\ \in \mathbb{R}^{1\times n}$.
	\end{definition}
	\yq{The following \fy{result} will be employed in the analysis.}
	\begin{lemma}[\fy{\cite[Lemma 1]{pu2021distributed}}]\label{lem:rho_W}
	Let Assumption \ref{assum:Weight} hold. Let $\rho_{W} $ denote the spectral norm of the matrix $\mathbf{W}-\tfrac{1}{m}\mathbf{1}\mathbf{1}^{\top}$. Then, $\rho_{W}<1 $ and $\|\mathbf{W}\mathbf{u}-\mathbf{1}\bar u\| \leq \rho_W \|\mathbf{u}-\mathbf{1}\bar u\|$ for all $\mathbf{u} \in \mathbb{R}^{m \times n}$, where $\bar u \triangleq \tfrac{1}{m}\mathbf{1}^{\top}\mathbf{u}$.
	\end{lemma}
	\begin{lemma}\label{lemma:stoch_grad_track_props}
	Consider Algorithm \ref{algorithm:IR-DSGT} and Definition \ref{eqn:defs_stoch}. Let Assumptions \ref{assum:problem}, \ref{assum:Weight}, and \ref{assum:stoch_oracle} hold. Then: (i) We have that $\bar y_k = G_k(\mathbf{x}_k,\bxi_k)$. (ii) We have $\mathscr{G}_k\left(x^*_{\lambda_k}\right) = 0$. (iii)  The mapping $\mathscr{G}_k(x)$ is  $(\lambda_k\mu_f)$-strongly monotone and Lipschitz continuous with parameter $L_k$. (iv) $\EXP{\|\bar y_k - G_k(\mathbf{x}_k) \|^2\mid \mathscr{F}_k}\leq \frac{2\nu_k^2}{m}$. (v) $ \|G_k(\mathbf{x}_k) - \bar g_k \|\leq \frac{L_k}{\sqrt{m}}\left\|\mathbf{x}_k-\mathbf{1}\bar x_k\right\|$. (vi) $\|\bar g_k \|\leq L_k\|\bar x_k-x^*_{\lambda_k}\|$.
	\end{lemma}
	\begin{proof}
	\noindent \textbf{(i)} This can be shown in a similar way to the proof of Lemma \ref{lemma:grad_track_props}(i) using the update rule \eqref{alg:IRDSGT_compact2} and the definition of $G_k(\mathbf{x}_k,\bxi_k)$ in Definition \ref{eqn:defs_stoch}.
	
	\noindent \textbf{(ii–iii)} See the proof of Lemma \ref{lemma:grad_track_props}\fy{(ii–iii)}.
	
	\noindent \textbf{(iv)} From part (i), we \fy{may write}
	\begin{align*}
	&\EXP{\|\bar y_k - G_k(\mathbf{x}_k) \|^2\mid \mathscr{F}_k} = \EXP{\|G_k(\mathbf{x}_k,\bxi_k) - G_k(\mathbf{x}_k) \|^2\mid \mathscr{F}_k} \\
	& \leq \tfrac{2}{m^2}\EXP{\|\txsum_{i=1}^m(\yqo{F_i}(x_{i,k},\xi_{i,k}) -\yqo{F_i}(x_{i,k})) \|^2\mid \mathscr{F}_k}
	+\lambda_k^2 \tfrac{2}{m^2}\EXP{\|\txsum_{i=1}^m(\nabla f_i(x_{i,k},\xi_{i,k}) -\nabla f_i(x_{i,k}))\|^2\mid \mathscr{F}_k}\\
	& = \tfrac{2}{m^2}\EXP{\txsum_{i=1}^m\|\yqo{F_i}(x_{i,k},\xi_{i,k}) -\yqo{F_i}(x_{i,k}) \|^2\mid \mathscr{F}_k}
	+ \tfrac{2\lambda_k^2}{m^2}\EXP{\txsum_{i=1}^m\|\nabla f_i(x_{i,k},\xi_{i,k}) -\nabla f_i(x_{i,k})\|^2\mid \mathscr{F}_k} \leq\tfrac{2\nu_k^2}{m},
	\end{align*}
	where the equation is implied by Assumption~\ref{assum:stoch_oracle}(a) and that vectors $\xi_{i,k}$ are independent across the agents for any $k\geq 0$.
	\noindent \textbf{(iv–v)} See the proof of \fy{Lemma~\ref{lemma:grad_track_props}(iv–v)}.
	\end{proof}
	\yq{\bz{As in} Proposition \ref{prop:main_ineq_1}, we introduce three errors metrics \bz{$\EXP{\|\bar x_{k+1}-x^*_{\lambda_k}\|^2\mid \mathscr{F}_k }$}, $\EXP{\|\mathbf{x}_{k+1}-\mathbf{1}\bar x_{k+1}\|^2\mid \mathscr{F}_k }$ and $\EXP{\|\by_{k+1}-\mathbf{1}\bar y_{k+1}\|^2\mid \mathscr{F}_k }$ to analyze the convergence of Algorithm \ref{algorithm:IR-DSGT}.} \fy{In the following, we provide three recursive inequalities for these error metrics.}
	\begin{proposition}\label{prop:stoch_alg_recursions} 
	Consider Algorithm \ref{algorithm:IR-DSGT}. Let Assumptions \ref{assum:problem}, \ref{assum:Weight}, and \ref{assum:stoch_oracle} hold. Let $\gamma_k\leq\frac{1}{L_k}$ for $k\geq 0$. Then, there exist $M>0$ and $B_{\yqt{F}}>0$ such that for all $k\geq 1$, the following hold.
	\begingroup
	\allowdisplaybreaks
	\begin{align*}
	(a)\ &\EXP{\|\bar x_{k+1}-x^*_{\lambda_k}\|^2\mid \mathscr{F}_k}  \leq  (1-\tfrac{\mu_f\lambda_k\gamma_k}{2})\|\bar x_k -x^*_{\lambda_{k-1}}\|^2  \\
	&+ \tfrac{L_k^2\gamma_k(1+\mu_f\lambda_k\gamma_k)}{m\lambda_k\mu_f} \|\mathbf{x}_k-\mathbf{1}\bar x_k\|^2+\tfrac{2\nu_k^2\gamma_k^2}{m} +\tfrac{2M^2\Lambda_{k-1}^2}{\mu_f^3\lambda_k\gamma_k}.\\
	(b)\  &\|\mathbf{x}_{k+1}-\mathbf{1} \bar x_{k+1}\|^2  \leq   \tfrac{1+\rho_{W}^2}{2} \|\mathbf{x}_k-\mathbf{1}{\bar x_k\|^2}
		+\tfrac{\gamma_k^2(1+\rho_{W}^2)\rho_{W}^2}{1-\rho_{W}^2} \| \mathbf{y}_k-\mathbf{1}\bar y_k\|^2.  \\
	(c)\ &\EXP{\| \by_{k+1}-\mathbf{1}\bar y_{k+1}\|^2\mid \sF_k}
	 \leq  \left((1+\zeta^{-1})\left(4L_k^2\gamma_k^2\right)+\zeta+1\right)\rho_W^2\EXP{\|\by_k-\bunit\bar y_{k}\|^2\mid \sF_k} \notag\\
	&+(1+\zeta^{-1})4L_k^2\left(6\Lambda_{k}^2 L_k^2+\|\bW-\mathbf{I}\|^2 +1.5\gamma_k^2\right)\|\bx_k-\bunit\bar x_k\|^2\notag\\
	&  +(1+\zeta^{-1})12m L_k^2(4\Lambda_{k}^2+L_k^2\gamma_k^2)\|\bar x_{k}-x^*_{\lambda_{k-1}}\|^2 
	 +(1+\zeta^{-1})12L_k^2\gamma_k^2\nu_k^2 + (1+\zeta^{-1})6\Lambda_{k}^2 B_{\fy{F}}^2\notag \\
	&+2m(\nu_{k+1}^2+\nu_k^2)+4m\gamma_k L_{k+1}\nu_{k}^2 +4\nu_k^2 
	+ (1+\zeta^{-1})12m L_k^2(4\Lambda_{k}^2+L_k^2\gamma_k^2) \tfrac{M^2}{\mu_f^2}\Lambda_{k-1}^2.
	\end{align*}
	\endgroup
\end{proposition}

\begin{proof} \fy{\bf (a)} Applying $\tfrac{1}{m}\mathbf{1}^{\top}$ to \eqref{alg:IRDSGT_compact1} and using $\mathbf{1}^{\top}\mathbf{W} = \mathbf{1}^{\top}$ implies that $\bar x_{k+1} = \bar x_k -{\gamma}_k \bar y_k$. Thus, we can write
$
	\|\bar x_{k+1}-x^*_{\lambda_k}\|^2 = \|\bar x_k-x^*_{\lambda_k}\|^2 - 2{\gamma}_k\bar y_k^{\top}( \bar x_k-x^*_{\lambda_k})
	 +\gamma_k^2\|\bar y_k\|^2.
$

	Taking conditional \fy{expectation on} the preceding relation and adding and subtracting $G_k(\mathbf{x}_k)$ in the last term, we obtain
$
	\EXP{\|\bar x_{k+1}-x^*_{\lambda_k}\|^2\mid \mathscr{F}_k} = \|\bar x_k-x^*_{\lambda_k}\|^2 + \gamma_k^2\|G_k(\mathbf{x}_k)\|^2
	+\gamma_k^2\EXP{\|\bar y_k-G_k(\mathbf{x}_k)\|^2\mid \mathscr{F}_k}- 2{\gamma}_kG_k(\mathbf{x}_k)^{\top}( \bar x_k-x^*_{\lambda_k}) ,
$
	where we used $\EXP{\bar y_k\mid \mathscr{F}_k}=G_k(\mathbf{x}_k)$. 
	
	Adding and subtracting $\bar g_k$ and using Lemma \ref{lemma:stoch_grad_track_props}(iv), we have
	\begin{align*}
	&\EXP{\|\bar x_{k+1}-x^*_{\lambda_k}\|^2\mid \mathscr{F}_k} = \|\bar x_k-x^*_{\lambda_k}\|^2 + \gamma_k^2\|G_k(\mathbf{x}_k)-\bar g_k\|^2+\gamma_k^2\|\bar g_k\|^2+2\gamma_k^2\bar{g}_k^{\top}(G_k(\mathbf{x}_k)-\bar g_k)\\
	&+\tfrac{2\nu_k^2}{m}\gamma_k^2
	- 2{\gamma}_k(G_k(\mathbf{x}_k)-\bar{g}_k)^{\top}( \bar x_k-x^*_{\lambda_k})- 2{\gamma}_k\bar{g}_k^{\top}( \bar x_k-x^*_{\lambda_k}) \\
	& \leq \|\bar x_k -\gamma_k\bar g_k -x^*_{\lambda_k}\|^2 + \gamma_k^2\|G_k(\mathbf{x}_k)-\bar g_k\|^2+\tfrac{2\nu_k^2}{m}\gamma_k^2
	 - 2{\gamma}_k(G_k(\mathbf{x}_k)-\bar{g}_k)^{\top}( \bar x_k-x^*_{\lambda_k}-\gamma_k\bar{g}_k). 
	\end{align*}
	Let $\rho_k \triangleq 1-\mu_f\lambda_k\gamma_k$. From Lemmas \ref{lem:gradient_recursion} and \ref{lemma:stoch_grad_track_props}(v) we obtain
$
	\EXP{\|\bar x_{k+1}-x^*_{\lambda_k}\|^2\mid \mathscr{F}_k} 
	 \leq \rho_k^2 \|\bar x_k -x^*_{\lambda_k}\|^2 + \tfrac{L_k^2\gamma_k^2}{m}\left\|\mathbf{x}_k-\mathbf{1}\bar x_k\right\|^2+\tfrac{2\nu_k^2}{m}\gamma_k^2
	 + 2\tfrac{L_k{\gamma}_k\rho_k}{\sqrt{m}}\|\bx_k -\mathbf{1}\bar{x}_k\|\| \bar x_k-x^*_{\lambda_k}\|.  
$

	\fy{We write}
$
	 2\tfrac{L_k{\gamma}_k\rho_k}{\sqrt{m}}\|\bx_k -\mathbf{1}\bar{x}_k\|\| \bar x_k-x^*_{\lambda_k}\|
	 \leq \gamma_k(\tfrac{L_k^2}{\lambda_k\mu_f m}\|\mathbf{x}_k-\mathbf{1} \bar x_k\|^2+\lambda_k\mu_f\rho_k^2\|\bar x_k - x^*_{\lambda_k}\|^2).
$

	From the preceding two relations, we obtain
$
	\EXP{\|\bar x_{k+1}-x^*_{\lambda_k}\|^2\mid \mathscr{F}_k}  \leq \rho_k^2 (1+\mu_f\lambda_k\gamma_k)\|\bar x_k -x^*_{\lambda_k}\|^2  
	+ \tfrac{L_k^2\gamma_k}{m\mu_f\lambda_k} (1+\mu_f\lambda_k\gamma_k)\left\|\mathbf{x}_k-\mathbf{1}\bar x_k\right\|^2+\tfrac{2\nu_k^2}{m}\gamma_k^2.
$

	Note that we have $ \rho_k^2 (1+\lambda_k\mu_f\gamma_k) \leq \rho_k$. For any $u,v \in \mathbb{R}^n$ and $\theta>0$ we have $\|u+v\|^2 \leq (1+\theta)\|u\|^2+(1+\tfrac{1}{\theta})\|v\|^2$. Using this and invoking Lemma \ref{lemma:IR-props}, we have
$
	\|\bar x_k-x^*_{\lambda_k}\|^2 \leq  (1+\tfrac{\mu_f\lambda_k\gamma_k}{2})\|\bar x_k-x^*_{\lambda_{k-1}}\|^2 
	+(1+\tfrac{2}{\mu_f\lambda_k\gamma_k})\|x^*_{\lambda_{k-1}}-x^*_{\lambda_k}\|^2 \leq (1+\tfrac{\mu_f\lambda_k\gamma_k}{2})\|\bar x_k-x^*_{\lambda_{k-1}}\|^2+(1+\tfrac{2}{\mu_f\lambda_k\gamma_k})\tfrac{M^2}{\mu_f^2}\Lambda_{k-1}^2.
$

	\fy{We obtain the result by noting} that $\rho_k(1+\tfrac{\mu_f\lambda_k\gamma_k}{2}) \leq 1-\tfrac{\mu_f\lambda_k\gamma_k}{2}$ and that $\rho_k(1+\tfrac{2}{\mu_f\lambda_k\gamma_k}) \leq \tfrac{2}{\mu_f\lambda_k\gamma_k}$.
	
	\noindent  \fy{\bf (b)} Next, we show the second recursive relation. 
	
	\fy{We write}
$
	\|\mathbf{x}_{k+1}-\mathbf{1} \bar x_{k+1}\|^2=\|\mathbf{W}\mathbf{x}_k-\gamma_k\mathbf{W}\mathbf{y}_k-\mathbf{1}(\bar x_k - \gamma_k\bar y_k)\|^2
	 =\|\mathbf{W}\mathbf{x}_k-\mathbf{1}\bar x_k\|^2 -2\gamma_k\langle \mathbf{W}\mathbf{x}_k-\mathbf{1}\bar x_k,\mathbf{W}\mathbf{y}_k-\mathbf{1}\bar y_k\rangle 
	+\gamma_k^2\|\mathbf{W}\mathbf{y}_k-\mathbf{1}\bar y_k\|^2.
$

	Invoking Lemma \ref{lem:rho_W}, we obtain
	\begin{align*}
	&\|\mathbf{x}_{k+1}-\mathbf{1} \bar x_{k+1}\|^2
	 =\rho_{W}^2\|\mathbf{x}_k-\mathbf{1}\bar x_k\|^2 +\rho_{W}^2\gamma_k^2\|\mathbf{y}_k-\mathbf{1}\bar y_k\|^2
	  +2\gamma_k\| \mathbf{W}\mathbf{x}_k-\mathbf{1}\bar x_k\|\|\mathbf{W}\mathbf{y}_k-\mathbf{1}\bar y_k\| \\
	& \leq\rho_{W}^2\|\mathbf{x}_k-\mathbf{1}\bar x_k\|^2 +\rho_{W}^2\gamma_k^2 \|\mathbf{y}_k-\mathbf{1}\bar y_k\|^2 
	+\rho_{W}^2\gamma_k(\tfrac{1-\rho_{W}^2}{2\gamma_k\rho_{W}^2}  \|\mathbf{x}_k-\mathbf{1}\bar x_k\|^2+\tfrac{2\gamma_k\rho_{W}^2}{1-\rho_{W}^2}\| \mathbf{y}_k-\mathbf{1}\bar y_k\|^2) \\
	& =   \tfrac{1+\rho_{W}^2}{2} \|\mathbf{x}_k-\mathbf{1}{\bar x_k\|^2}+\tfrac{\gamma_k^2(1+\rho_{W}^2)\rho_{W}^2}{1-\rho_{W}^2} \| \mathbf{y}_k-\mathbf{1}\bar y_k\|^2.  
	\end{align*}
	 
	\noindent  \fy{\bf (c)} Next we obtain the third recursive relation. To utilize the space, we use the abstract notation $\bG_k \triangleq \bG_k(\bx_{k})$, $\tilde \bG_k \triangleq \bG_k(\bx_{k},\bxi_{k})$, $\nabla_{i,k}^{\fy{F}} \triangleq \fy{F}_i(x_{i,k})$, $\tilde \nabla^{\fy{F}}_{i,k} \triangleq \fy{F}_i(x_{i,k},\xi_{i,k})$, $\nabla_{i,k}^f \triangleq \nabla f_i(x_{i,k})$, $\tilde \nabla^f_{i,k} \triangleq \nabla f_i(x_{i,k},\xi_{i,k})$, $\nabla_{i,k}^\lambda \triangleq \nabla_{i,k}^{\fy{F}}+\lambda\nabla_{i,k}^f$, $\tilde\nabla_{i,k}^\lambda \triangleq \tilde \nabla_{i,k}^{\fy{F}}+\lambda\tilde \nabla_{i,k}^f$ for some $\lambda>0$. From \eqref{alg:IRDSGT_compact2} we have
	\begin{align}\label{eqn:third_recursion_proof_eqn0}
	&\| \by_{k+1}-\mathbf{1}\bar y_{k+1}\|^2
	 \leq \|\bW\by_k+\tilde \bG_{k+1}-\tilde \bG_k-\bunit\bar y_{k}+\bunit\bar y_{k}-\bunit\bar y_{k+1}\|^2\notag\\
	& = \|\bW\by_k-\bunit\bar y_{k}\|^2+ \|\tilde \bG_{k+1}-\tilde \bG_k\|^2
	+2\langle \bW \by_k -\bunit \bar y_k,\tilde \bG_{k+1}-\tilde \bG_k \rangle\notag\\
	& +2\langle  \by_{k+1} -\bunit \bar y_k,\bunit(\bar y_{k}-\bar y_{k+1}) \rangle+ m\|\bar y_{k}-\bar y_{k+1}\|^2 \notag\\
	&= \rho_W^2\|\by_k-\bunit\bar y_{k}\|^2 + \|\tilde \bG_{k+1}-\tilde \bG_k\|^2
	+ 2\langle \bW \by_k -\bunit \bar y_k,\tilde \bG_{k+1}-\tilde \bG_k \rangle -m\|\bar y_{k}-\bar y_{k+1}\|^2\notag\\
	&\leq \rho_W^2\|\by_k-\bunit\bar y_{k}\|^2 + \|\tilde \bG_{k+1}-\tilde \bG_k\|^2
	+ 2\langle \bW \by_k -\bunit \bar y_k,\tilde \bG_{k+1}-\tilde \bG_k \rangle.
	\end{align}
	In the following, we present some intermediary results that will be used to derive the third recursive inequality.
	
	\noindent \textbf{Claim 1.} The following holds
	\begin{align}\label{eqn:IRDSGT-claim1}
	& \EXP{\|\tilde \bG_{k+1}-\tilde \bG_k\|^2\mid \sF_k}  \leq  \EXP{\|\bG_{k+1}-\bG_k\|^2\mid \sF_k}\notag \\
	& + 2\EXP{\langle \bG_{k+1}, -\tilde \bG_k +\bG_k\rangle \mid \sF_k} +2m(\nu_{k+1}^2+\nu_k^2).
	\end{align}
	\noindent \textbf{Proof of Claim 1.} We \fy{may} write
	\begin{align}\label{eqn:third_recursion_proof_eqn1}
	& \EXP{\|\tilde \bG_{k+1}-\tilde \bG_k\|^2\mid \sF_k} = \EXP{\|\bG_{k+1}-\bG_k\|^2\mid \sF_k} 
	 + 2\EXP{\langle \bG_{k+1}, \tilde \bG_{k+1}-\tilde \bG_k-\bG_{k+1}+\bG_k\rangle \mid \sF_k}\notag \\
	& -2\EXP{\langle \bG_k, \tilde \bG_{k+1}-\tilde \bG_k-\bG_{k+1}+\bG_k\rangle \mid \sF_k} 
	 +\EXP{\| \tilde \bG_{k+1}-\tilde \bG_k-\bG_{k+1}+\bG_k\|^2 \mid \sF_k}.
	\end{align}
	Note that since $\bx_{k+1}$ is characterized in terms of $\bxi_k$, we have 
$
	\EXP{\tilde \bG_{k+1} -\bG_{k+1}\mid \sF_k}
	= \mathsf{E}_{\bxi_k}\left[\mathsf{E}_{\bxi_{k+1}}\left[\tilde \bG_{k+1}-\bG_{k+1} \mid \sF_{k+1}\right]\right]=\mathsf{E}_{\bxi_k}\left[0 \right] =0. 
$
Also, we have $\EXP{\tilde \bG_{k} -\bG_{k}\mid \sF_k}  =0$. 

Thus, we obtain
$
	\EXP{\langle \bG_k, \tilde \bG_{k+1}-\tilde \bG_k-\bG_{k+1}+\bG_k\rangle \mid \sF_k}=0.
$

	We \fy{may} also write 
	\begin{align*}
	& \mathsf{E}\left[\langle \mathbf{G}_{k+1}, \tilde{\mathbf{G}}_{k+1}-\tilde{\mathbf{G}}_k-\mathbf{G}_{k+1}+\mathbf{G}_k\rangle \mid \mathscr{F}_k\right] \\
	& =  \mathsf{E}_{\boldsymbol{\xi}_k}\left[ \mathsf{E}_{\boldsymbol{\xi}_{k+1}}\left[\langle \mathbf{G}_{k+1}, \tilde{\mathbf{G}}_{k+1}-\tilde{\mathbf{G}}_k-\mathbf{G}_{k+1}+\mathbf{G}_k\rangle \mid  \mathscr{F}_{k+1} \right] \right]\\
	& =  \mathsf{E}_{\boldsymbol{\xi}_k}\left[\langle \mathbf{G}_{k+1},  \mathbf{G}_{k+1}- \tilde{\mathbf{G}}_k-\mathbf{G}_{k+1}+\mathbf{G}_k\rangle\mid \mathscr{F}_k\right]
	 =  \mathsf{E}\left[\langle \mathbf{G}_{k+1},   - \tilde{\mathbf{G}}_k+\mathbf{G}_k\rangle\mid \mathscr{F}_k\right].
	\end{align*}
	From the preceding relations and \eqref{eqn:third_recursion_proof_eqn1} we have
	\begin{align*}
	& \EXP{\|\tilde \bG_{k+1}-\tilde \bG_k\|^2\mid \sF_k} \leq  \EXP{\|\bG_{k+1}-\bG_k\|^2\mid \sF_k}
	 + 2\EXP{\langle \bG_{k+1}, -\tilde \bG_k +\bG_k\rangle \mid \sF_k} \\
	& +\EXP{\| \tilde \bG_{k+1}-\tilde \bG_k-\bG_{k+1}+\bG_k\|^2 \mid \sF_k} 
	 \leq  \EXP{\|\bG_{k+1}-\bG_k\|^2\mid \sF_k}  \\
	& + 2\EXP{\langle \bG_{k+1}, -\tilde \bG_k +\bG_k\rangle \mid \sF_k}+\EXP{\| \tilde \bG_{k+1} -\bG_{k+1} \|^2 \mid \sF_k}+\EXP{\| \tilde \bG_k-\bG_k\|^2 \mid \sF_k}.
	\end{align*}
	From Lemma \ref{lemma:stoch_grad_track_props}(iv), we obtain Claim 1. 
	
	\noindent \textbf{Claim 2.} The following holds
	\begin{align}\label{eqn:IRDSGT-claim2}
	\EXP{\langle \bG_{k+1}, -\tilde \bG_k +\bG_k\rangle \mid \sF_k}\leq 2m\gamma_k L_{k+1}\nu_{k}^2.
	\end{align}
	\noindent \textbf{Proof of Claim 2.} Let us define for all $i \in [m]$
$
	\hat x_{i,k+1} \triangleq  x_{i,k+1} +\gamma_kW_{ii}(\tilde \nabla^{\lambda_k}_{i,k}-\nabla^{\lambda_k}_{i,k}).
$

	We \fy{may} write 
	\begin{align*}
	&\EXP{\langle \nabla^{\lambda_{k+1}}_{i,k+1} , - \tilde \nabla^{\lambda_{k}}_{i,k}- \nabla^{\lambda_{k}}_{i,k} \rangle \mid \sF_k}
	=\EXP{\langle \nabla^{\lambda_{k+1}}_{i,k+1} -\fy{F}_i(\hat x_{i,k+1}) -\lambda_{k+1}\nabla f_i(\hat x_{i,k+1}) - \tilde \nabla^{\lambda_{k}}_{i,k}+ \nabla^{\lambda_{k}}_{i,k} \rangle \mid \sF_k}\\
	&+\EXP{\langle  \fy{F}_i(\hat x_{i,k+1})+\lambda_{k+1}\nabla f_i(\hat x_{i,k+1}), - \tilde \nabla^{\lambda_{k}}_{i,k}+\nabla^{\lambda_{k}}_{i,k} \rangle \mid \sF_k}.
	\end{align*}
	It can be seen from the update rules of Algorithm \ref{algorithm:IR-DSGT} that $\hat x_{i,k+1}$ is independent of $\xi_{i,k}$. Thus, we \fy{may} write 
$
	\EXP{\langle  \fy{F}_i(\hat x_{i,k+1})+\lambda_{k+1}\nabla f_i(\hat x_{i,k+1}), - \tilde \nabla^{\lambda_{k}}_{i,k}+\nabla^{\lambda_{k}}_{i,k} \rangle \mid \sF_k} = 0.
$

	Also, using Assumption \ref{assum:problem}, we have
$
	\|\nabla^{\lambda_{k+1}}_{i,k+1} -\fy{F}_{i}(\hat x_{i,k+1})-\lambda_{k+1}\nabla f_{i}(\hat x_{i,k+1}) \|
	= L_{k+1}\|x_{i,k+1}-\hat x_{i,k+1}\|\leq \gamma_k L_{k+1}\|\tilde \nabla^{\lambda_k}_{i,k}- \nabla^{\lambda_k}_{i,k}\|.
$

From the preceding relations, we have
	\begin{align*}
	&\EXP{\langle \nabla^{\lambda_{k+1}}_{i,k+1} , - \tilde \nabla^{\lambda_{k}}_{i,k}+ \nabla^{\lambda_{k}}_{i,k} \rangle \mid \sF_k}
	\leq \EXP{\| \nabla^{\lambda_{k+1}}_{i,k+1} -\fy{F}_i(\hat x_{i,k+1}) -\lambda_{k+1}\nabla f_i(\hat x_{i,k+1})\| \times \|\tilde \nabla^{\lambda_{k}}_{i,k}- \nabla^{\lambda_{k}}_{i,k}\| \mid \sF_k}\\
	& \leq \gamma_k L_{k+1}\EXP{\|\tilde \nabla^{\lambda_k}_{i,k}- \nabla^{\lambda_k}_{i,k}\|^2\mid \sF_k}
	 \leq 2\gamma_k L_{k+1}\EXP{\|\tilde \nabla^{\fy{F}}_{i,k}- \nabla^{\fy{F}}_{i,k}\|^2+\lambda_k^2\|\tilde \nabla^f_{i,k}- \nabla^f_{i,k}\|^2\mid \sF_k}\\
	& \leq 2\gamma_k L_{k+1}\nu_{k}^2.
	\end{align*}
Summing \fy{the} preceding relation over $i$, we obtain \eqref{eqn:IRDSGT-claim2}.
	
	\noindent \textbf{Claim 3.} The following holds
	\begin{align}\label{eqn:IRDSGT-claim3}
	&\|\bG_{k+1}-\bG_k\|^2 
	\leq 4L_k^2\left(6\Lambda_{k}^2 L_k^2+\|\bW-\mathbf{I}\|^2+1.5\gamma_k^2\right)\|\bx_k-\bunit\bar x_k\|^2
	+4L_k^2\rho_W^2\gamma_k^2\|\by_k-\bunit \bar y_k	\|^2 \notag\\
	&  +6m L_k^2(4\Lambda_{k}^2+L_k^2\gamma_k^2)\|\bar x_{k}-x^*_{\lambda_k}\|^2 
	 +12L_k^2\gamma_k^2\nu_k^2 + 6\Lambda_{k}^2 B_{\fy{F}}^2.
	\end{align}
	\noindent \textbf{Proof of Claim 3.} From Definition \ref{eqn:defs_stoch} we have
$
	 \|\bG_{k+1}-\bG_k\|^2 \leq 2\|\bG_{k+1}-\yqt{{\mathbf F}}(\bx_k)-\lambda_{k+1}\nabla \bfun(\bx_k)\|^2
	+2\|\lambda_{k+1}\nabla \bfun(\bx_k)-\lambda_{k}\nabla \bfun(\bx_k)\|^2
	 \leq 2\Lambda_{k}^2\| \lambda_k\nabla \bfun(\bx_k)\|^2+2L_k^2\|\bx_{k+1}-\bx_k\|^2. 
$
	From Lemma \ref{lemma:IR-props}, there exists a scalar $B_{\fy{F}}<\infty$ such that $\fy{L_F}\|\mathbf{1}x^*_{\lambda_k}-\mathbf{1}x^*\|_2\leq B_{\fy{F}}$. Invoking $\fy{F}(x^*)=0$, we have
	\begin{align*}
	&\| \lambda_k\nabla \bfun(\bx_k)\| \leq  \|\fy{\mathbf{F}}(\mathbf{x}_{k})+\lambda_k\nabla \mathbf{f}(\mathbf{x}_{k})\| 
	 +\|\fy{\mathbf{F}}(\mathbf{x}_{k})-\fy{\mathbf{F}}(\mathbf{1}x^*)\| \\
	& \leq \|\fy{\mathbf{F}}(\mathbf{x}_{k})+\lambda_k\nabla \mathbf{f}(\mathbf{x}_{k}) - \fy{\mathbf{F}}(\mathbf{1}x^*_{\lambda_k})-\lambda_k\nabla \mathbf{f}(\mathbf{1}x^*_{\lambda_k}) \|  
	+ \fy{L_F}\|\mathbf{x}_{k}-\mathbf{1}x^*\|  \\
	&\leq (L_k+\fy{L_F})\|\mathbf{x}_{k}-\mathbf{1}x^*_{\lambda_k}\| + \fy{L_F}\|\mathbf{1}x^*_{\lambda_k}-\mathbf{1}x^*\| 
	\leq 2L_k(\|\mathbf{x}_{k}-\mathbf{1}\bar x_{k}\| +\|\mathbf{1}\bar x_{k}-\mathbf{1}x^*_{\lambda_k}\|_2)+ B_{\fy{F}} \\
	&\leq 2L_k\|\mathbf{x}_{k}-\mathbf{1}\bar x_{k}\| +2\sqrt{m}L_k\|\bar x_{k}-x^*_{\lambda_k}\| + B_{\fy{F}}.
	\end{align*}
	This implies that 
$
	\| \lambda_k\nabla \bfun(\bx_k)\|^2 \leq  12L_k^2\|\mathbf{x}_{k}-\mathbf{1}\bar x_{k}\|^2 +12mL_k^2\|\bar x_{k}-x^*_{\lambda_k}\|^2 + 3B_{\fy{F}}^2.
$

	We also have
	\begin{align*}
	&\|\bx_{k+1}-\bx_{k}\|^2 =\|\bW\bx_k-\gamma_k\bW\by_k-\bx_k\|^2
	 = \|(\bW-\mathbf{I})(\bx_k-\bunit\bar x_k)-\gamma_k\bW\by_k-\bx_k\|^2\\
	& \leq  \|\bW-\mathbf{I}\|^2\|\bx_k-\bunit\bar x_k\|^2+\gamma_k^2\|\bW\by_k\|^2 
	-2\gamma_k\langle (\bW-\mathbf{I})(\bx_k-\bunit\bar x_k), \bW\by_k\rangle\\
	& =  \|\bW-\mathbf{I}\|^2\|\bx_k-\bunit\bar x_k\|^2+\gamma_k^2\|\bW\by_k-\bunit \bar y_k	\|^2\\
	&+m\gamma_k^2\|\bar y_k\|^2-2\gamma_k\langle (\bW-\mathbf{I})(\bx_k-\bunit\bar x_k), \bW\by_k-\bunit\bar y_k\rangle \\
	& \leq \|\bW-\mathbf{I}\|^2\|\bx_k-\bunit\bar x_k\|^2+\rho_W^2\gamma_k^2\|\by_k-\bunit \bar y_k	\|^2\\
	&+m\gamma_k^2\|\bar y_k\|^2+2\rho_{W}\gamma_k\|\bW-\mathbf{I}\|\|\bx_k-\bunit\bar x_k\|  \|\by_k-\bunit\bar y_k \|\\
	& \leq 2\|\bW-\mathbf{I}\|^2\|\bx_k-\bunit\bar x_k\|^2+2\rho_W^2\gamma_k^2\|\by_k-\bunit \bar y_k	\|^2
	+m\gamma_k^2\|\bar y_k\|^2.
	\end{align*}
	From Lemma \ref{lemma:grad_track_props}, we have
$
	\EXP{\|\bar y_k\|^2\mid \mathscr{F}_k} \leq 3\EXP{\|\bar y_k - G_k(\mathbf{x}_k) \|^2\mid \mathscr{F}_k} 
	 +3\|G_k(\mathbf{x}_k) - \bar g_k \|^2+ 3\|\bar g_k \|^2  
	 \leq  \tfrac{6\nu_k^2}{m} + \tfrac{3L_k^2}{m}\left\|\mathbf{x}_k-\mathbf{1}\bar x_k\right\|^2 +3L_k^2\|\bar x_k-x^*_{\lambda_k}\|^2.
$

	From the preceding relations, we obtain \eqref{eqn:IRDSGT-claim3}.

	\noindent \textbf{Claim 4.} The following holds
	\begin{align}\label{eqn:IRDSGT-claim4}
	&\EXP{\langle \bW \by_k -\bunit \bar y_k,\tilde \bG_{k+1}-\tilde \bG_k \rangle\mid \sF_k} \notag\\
	&= \EXP{\langle \bW \by_k -\bunit \bar y_k,  \bG_{k+1}-  \bG_k \rangle\mid \sF_k}
	+\EXP{\langle \bW \by_k -\bunit \bar y_k,-\tilde \bG_k+ \bG_k \rangle\mid \sF_k} .
	\end{align}
	\noindent \textbf{Proof of Claim 4.} We have
$$
	 \EXP{\langle \bW \by_k -\bunit \bar y_k,\tilde \bG_{k+1}- \bG_{k+1} \rangle\mid \sF_k} 
	 =  \mathsf{E}_{\bxi_k}\left[ \mathsf{E}_{\bxi_{k+1}}\left[\langle \bW \by_k -\bunit \bar y_k,\tilde \bG_{k+1}- \bG_{k+1} \rangle\mid \sF_{k+1}\right]\right] = 0.
$$
	The result follows by adding the above expectation to the left-hand side of \eqref{eqn:IRDSGT-claim4}.
	
	\noindent \textbf{Claim 5.} The following holds
	\begin{align}\label{eqn:IRDSGT-claim5}
	\EXP{\langle \bW \by_k -\bunit \bar y_k,-\tilde \bG_k +\bG_{k}\rangle\mid \sF_k}\leq 2\nu_k^2.
	\end{align}
	\noindent \textbf{Proof of Claim 5.} First, note that from Algorithm \ref{algorithm:IR-DSGT}, we have for any $i,j \in [m]$
	\begin{align*}
	&\EXP{\langle y_{j,k}, -\tilde \nabla^{\lambda_{k}}_{i,k}+ \nabla^{\lambda_{k}}_{i,k} \rangle\mid \sF_k}
	 = \EXP{\langle \txsum_{\ell=1}^mW_{j\ell}y_{\ell,k-1}+\tilde\nabla_{j,k}^{\lambda_k}-\tilde\nabla_{j,k-1}^{\lambda_{k-1}} , -\tilde \nabla^{\lambda_{k}}_{i,k}+ \nabla^{\lambda_{k}}_{i,k} \rangle\mid \sF_k}\\
	&=\EXP{\langle  \tilde\nabla_{j,k}^{\lambda_k}  , -\tilde \nabla^{\lambda_{k}}_{i,k}+ \nabla^{\lambda_{k}}_{i,k} \rangle\mid \sF_k}.
	\end{align*}
	Multiplying by $W_{ij}$ and summing over $j \in [m]$, we have
	\begin{align}\label{eqn:Wy_k_one_bary_k_eqn1}
	&\EXP{\langle \txsum_{j=1}^m W_{ij}y_{j,k}, -\tilde \nabla^{\lambda_{k}}_{i,k}+ \nabla^{\lambda_{k}}_{i,k} \rangle\mid \sF_k}
	=\EXP{\langle W_{ii} \tilde\nabla_{i,k}^{\lambda_k}  , -\tilde \nabla^{\lambda_{k}}_{i,k}+ \nabla^{\lambda_{k}}_{i,k} \rangle\mid \sF_k}\notag\\
	&=W_{ii} \EXP{\langle \tilde\nabla_{i,k}^{\lambda_k} - \nabla^{\lambda_{k}}_{i,k} , -\tilde \nabla^{\lambda_{k}}_{i,k}+ \nabla^{\lambda_{k}}_{i,k} \rangle\mid \sF_k} \leq 0.
	\end{align}
	Second, from Lemma \ref{lemma:stoch_grad_track_props}(i) we have 
	\begin{align}\label{eqn:Wy_k_one_bary_k_eqn2}
	&-\EXP{\langle \bar y_k, -\tilde \nabla^{\lambda_{k}}_{i,k}+ \nabla^{\lambda_{k}}_{i,k} \rangle\mid \sF_k}
	 = -\EXP{\langle \tfrac{1}{m}\txsum_{j=1}^m \tilde \nabla^{\lambda_k}_{j,k}, -\tilde \nabla^{\lambda_{k}}_{i,k}+ \nabla^{\lambda_{k}}_{i,k} \rangle\mid \sF_k}\notag\\
	&= \tfrac{1}{m}\EXP{\langle - \tilde \nabla^{\lambda_k}_{i,k}  , -\tilde \nabla^{\lambda_{k}}_{i,k}+ \nabla^{\lambda_{k}}_{i,k} \rangle\mid \sF_k}
	= \tfrac{1}{m}\EXP{\langle - \tilde \nabla^{\lambda_k}_{i,k}  +\nabla^{\lambda_{k}}_{i,k} , -\tilde \nabla^{\lambda_{k}}_{i,k}+ \nabla^{\lambda_{k}}_{i,k} \rangle\mid \sF_k}\notag\\
	&= \tfrac{1}{m}\EXP{\| \nabla^{\lambda_{k}}_{i,k}- \tilde \nabla^{\lambda_k}_{i,k} \|^2 \mid \sF_k}
	\leq \tfrac{2}{m}\EXP{\|\tilde \nabla^{\fy{F}}_{i,k}- \nabla^{\fy{F}}_{i,k}\|^2+\lambda_k^2\|\tilde \nabla^f_{i,k}- \nabla^f_{i,k}\|^2\mid \sF_k}
	\leq \tfrac{2\nu_k^2}{m}.
	\end{align}
	Employing \eqref{eqn:Wy_k_one_bary_k_eqn1} and \eqref{eqn:Wy_k_one_bary_k_eqn2}, we have
$
	\EXP{\langle \bW \by_k -\bunit \bar y_k,-\tilde \bG_k +\bG_{k}\rangle\mid \sF_k}  =\txsum_{i=1}^m\EXP{\langle \txsum_{j=1}^m W_{ij}y_{j,k}-\bar y_k, -\tilde \nabla^{\lambda_{k}}_{i,k}+ \nabla^{\lambda_{k}}_{i,k} \rangle\mid \sF_k} 
	 \leq \txsum_{i=1}^m\tfrac{2\nu_k^2}{m} = 2\nu_k^2.
$
	This implies that \eqref{eqn:IRDSGT-claim5} holds.
	
	\noindent \textbf{Claim 6.} The following holds for any $\zeta>0$
	\begin{align}\label{eqn:IRDSGT-claim6}
	&2\langle \bW \by_k -\bunit \bar y_k, \bG_{k+1} -\bG_{k}\rangle 
	\leq  \zeta\rho_W^2\|\by_k -\bunit \bar y_k\|^2+{\zeta^{-1}}\|\bG_{k+1}-\bG_{k}\|^2.
	\end{align}
	\noindent \textbf{Proof of Claim 6.} This relation is directly implied by the properties of the Frobenius inner product and using Lemma \ref{lem:rho_W}. 
\medskip 
	
	From Claim 1 to 6 and relation \eqref{eqn:third_recursion_proof_eqn0}, we have
	\begin{align*} 
	&\EXP{\| \by_{k+1}-\mathbf{1}\bar y_{k+1}\|^2\mid \sF_k}
	 \leq  \rho_W^2\EXP{\|\by_k-\bunit\bar y_{k}\|^2\mid \sF_k} \notag\\
	&+(1+\zeta^{-1})4L_k^2\left(6\Lambda_{k}^2 L_k^2+\|\bW-\mathbf{I}\|^2+1.5\gamma_k^2\right)\|\bx_k-\bunit\bar x_k\|^2\notag\\
	&+\left((1+\zeta^{-1})\left(4L_k^2\gamma_k^2\right)+\zeta\right)\rho_W^2\EXP{\|\by_k-\bunit \bar y_k	\|^2\mid \sF_k} 
	  +(1+\zeta^{-1})6m L_k^2(4\Lambda_{k}^2+L_k^2\gamma_k^2)\|\bar x_{k}-x^*_{\lambda_k}\|^2 \notag\\
	& +(1+\zeta^{-1})12L_k^2\gamma_k^2\nu_k^2 + (1+\zeta^{-1})6\Lambda_{k}^2 B_{\fy{F}}^2
	+2m(\nu_{k+1}^2+\nu_k^2)+4m\gamma_k L_{k+1}\nu_{k}^2 +4\nu_k^2.
	\end{align*}
	Also, in view of Lemma \ref{lemma:IR-props}, we have
$$
	\|\bar x_k-x^*_{\lambda_k}\|^2 \leq  2\|\bar x_k-x^*_{\lambda_{k-1}}\|^2+2\|x^*_{\lambda_{k-1}}-x^*_{\lambda_k}\|^2  
	\leq 2\|\bar x_k-x^*_{\lambda_{k-1}}\|^2+ \tfrac{2M^2}{\mu_f^2}\Lambda_{k-1}^2.
$$
The preceding two relations give the third recursive inequality. 
\end{proof}
\yq{Next, we present convergence guarantees \fy{for} Algorithm \ref{algorithm:IR-DSGT} \fy{in resolving the} stochastic \fy{NE seeking} problem \eqref{eqn:bilevel_problem_stoch}.} 
\begin{theorem}[\fy{Convergence statements for \eqref{eqn:bilevel_problem_stoch}}]\label{thm:rate} 
Consider Algorithm \ref{algorithm:IR-DSGT}.  Let Assumptions \fy{\ref{assum:problem}}, \yq{\ref{assum:Weight}, \fy{and} \ref{assum:stoch_oracle}} hold. Let us define
$\mathsf{e}_{1,k} \triangleq \EXP{\|\bar x_k - \yq{x^*_{\lambda_{k-1}}}\|^2}$, $\mathsf{e}_{2,k}\triangleq\EXP{\|\mathbf{x}_k-\mathbf{1} \bar x_k\|^2}$,
and $\mathsf{e}_{3,k}\triangleq\EXP{ \| \mathbf{y}_{k}-\mathbf{1}\bar y_{k}\|^2}$ for $k\geq 0$. Suppose $\gamma_k:=\tfrac{\gamma}{(k+\Gamma)^a}$ and 
$\lambda_k:=\tfrac{\lambda}{(k+\Gamma)^b}$ with $\gamma>0$, $\lambda>0$, $\Gamma\geq 1$, \yq{$a>b>0$, \fy{$3a+b<2$}} and
\bz{
$
 \Gamma \geq \max\bigg\{\sqrt[a+b]{\tfrac{\mu_f\gamma\lambda}{2}},\sqrt[b]{\tfrac{\lambda L_f}{\fy{L_F}}},\sqrt[a]{2\gamma \fy{L_F}},\sqrt[a]{4\gamma \fy{L_F}\zeta^{-1}\sqrt{1+\zeta}}\bigg\},
$}
where $0 < \zeta \leq \tfrac{1-\rho_W^2}{4\rho_W^2}$. \fy{The following results hold.}

\noindent (a) \fy{[Recursive error bounds] There exist scalars} $\theta_t>0$ for $t=1,\ldots,9$ with $\theta_4<1$ and $\theta_6<1$ such that for all $k \geq 0$, we have
$
 \mathsf{e}_{1,k+1} \leq (1-\theta_1\gamma_k\lambda_k)\mathsf{e}_{1,k} + \theta_2\mathsf{e}_{2,k}+\theta_3\gamma_k^2,\
\mathsf{e}_{2,k+1} \leq (1-\theta_4)\mathsf{e}_{2,k} + \theta_5\gamma_k^2\mathsf{e}_{3,k},\ 
   \mathsf{e}_{3,k+1} \leq (1-\theta_6)\mathsf{e}_{3,k} + \theta_7\mathsf{e}_{1,k}+\theta_8\mathsf{e}_{2,k}+\theta_9,
$
where
\bz{
\begin{align*}
&\theta_1:=0.5\mu_f,\ \theta_2:=\tfrac{L_0^2\fy{\gamma}(1+\mu_f\fy{\lambda}\fy{\gamma})}{m\lambda_k\mu_f},\
\yq{\theta_3:= \tfrac{2(\yqt{\nu_F^2}\Gamma^{2b}+\fy{\lambda}^2\nu_f^2)}{m\Gamma^{2b}} +\tfrac{2M^2}{\mu_f^3\fy{\lambda}\fy{\gamma}^3\Gamma^{2-3a-b}},}\\
&\yq{\theta_4:= \tfrac{1-\rho_{W}^2}{2},  \quad \theta_5:= \tfrac{(1+\rho_{W}^2)\rho_{W}^2}{1-\rho_{W}^2}},\
\yq{\theta_6:=1- \left((1+\zeta^{-1})\left(4L_k^2\gamma_k^2\right)+\zeta+1\right)\rho_W^2}\\
&\yq{\theta_7:= (1+\zeta^{-1})12m L_0^2(\tfrac{4}{\Gamma^2}+1)},\
\yq{\theta_8:= (1+\zeta^{-1})4L_0^2\left(6\tfrac{1}{\Gamma^2} L_0^2+\|\bW-\mathbf{I}\|^2+1.5\tfrac{\fy{\gamma}^2}{\Gamma^{2a}}\right),}\\
&\yq{\theta_9:= (1+\zeta^{-1})12(\yqt{\nu_F^2}+\tfrac{\fy{\lambda}^2}{\Gamma^2}\nu_f^2) + (1+\zeta^{-1})6\tfrac{1}{\Gamma^2} B_{\fy{F}}^2}
\yq{+4m(\yqt{\nu_F^2}+\tfrac{\fy{\lambda}^2}{\Gamma^2}\nu_f^2)+4m(\yqt{\nu_F^2}+\tfrac{\fy{\lambda}^2}{\Gamma^2}\nu_f^2) }\\
&\yq{\quad\quad +4(\yqt{\nu_F^2}+\tfrac{\fy{\lambda}^2}{\Gamma^2}\nu_f^2)+(1+\zeta^{-1})12m L_0^2\tfrac{M^2}{\mu_f^2}(\tfrac{4}{\Gamma^2}+1) \tfrac{1}{\Gamma^2}}.
\end{align*}
}

\noindent (b) \fy{[Non-asymptotic error bounds]} Let $\gamma\lambda  >\tfrac{1}{\theta_1}$. Let us define 
\begin{align}\label{eqn:hat_e}
	\hat{\mathsf{e}}_{1} := \Gamma^{a-b}\mathsf{e}_{1,0}\mathsf{n}_1, \quad  \hat{\mathsf{e}}_{2} :=  \Gamma^{2a}\mathsf{e}_{2,0}\mathsf{n}_2,  \quad  \hat{\mathsf{e}}_{3} := \Gamma^{2a}\tfrac{3\theta_9}{\theta_6}\mathsf{n}_3.
\end{align} where $\mathsf{n}_1,\mathsf{n}_2,\mathsf{n}_3>0$ are given as 
$
	\mathsf{n}_1 : = \tfrac{2C_2}{C_3\Gamma^{a-b}-2C_1C_4C_5}, \quad \mathsf{n}_2: =\tfrac{2C_4C_5}{ \Gamma^{2a}}\mathsf{n}_1, \quad \mathsf{n}_3: =\tfrac{C_5}{ \Gamma^{2a}}\mathsf{n}_1.
$

where 
$
	 C_1\triangleq  \theta_2\mathsf{e}_{2,0}, \ C_2\triangleq \theta_3\gamma^2, \ C_3 \triangleq  \left(\gamma\lambda\theta_1-1\right)\mathsf{e}_{1,0},\ 
	C_4 \triangleq \tfrac{6\theta_9\theta_5\gamma^2}{\mathsf{e}_{2,0}\theta_4\theta_6},\ C_5\triangleq  \tfrac{\theta_7\mathsf{e}_{1,0}}{\theta_9}, \ C_6 \triangleq \tfrac{\theta_8\mathsf{e}_{2,0}}{\theta_9}.
$

Then, if 
$
	\Gamma \geq \max \{\sqrt[2a]{2C_4C_6},\sqrt[a-b]{\tfrac{2C_1C_4C_5}{C_3}},\tfrac{1}{1-\sqrt[2a]{1-0.5\theta_4}}-1 \},
$ the following \fy{results} hold for all $k \geq 0$.
\begin{equation}\label{eqn:eik}\begin{aligned}
	\mathsf{e}_{1,k} \leq  \tfrac{\hat{\mathsf{e}}_{1}}{(k+\Gamma)^{a-b}}, \qquad \mathsf{e}_{2,k} \leq \tfrac{\hat{\mathsf{e}}_{2}}{(k+\Gamma)^{2a}}, \qquad   \mathsf{e}_{3,k} \leq \hat{\mathsf{e}}_{3}.
\end{aligned}
\end{equation}

\yq{\noindent (c) \fy{[Asymptotic convergence] The sequence of the averaged iterate, $\{\bar x_{k}\}$, converges to the unique optimal solution of problem~\eqref{eqn:bilevel_problem_stoch} in a mean-squared sense, i.e., we have} $\lim_{k\to \infty}\mathbb{E}[\|\bar x_{k+1} - x^*\|^2] = 0.$}
\end{theorem} 

\begin{proof} (a) \fy{The three inequalities in part (a) follow from Proposition~\ref{prop:stoch_alg_recursions} after invoking Lemma~\ref{lem:update_rules_props}. To complete the proof, we elaborate on the condition required for $\Gamma$.} Note that $\Gamma \geq \max\bigg\{ \sqrt[b]{\tfrac{\lambda L_f}{\fy{L_F}}},\sqrt[a]{2\gamma \fy{L_F}}\bigg\}$ implies that $\fy{\gamma}L_0 \leq 1$ and thus, $\gamma_k \leq \tfrac{1}{L_k}$ for all $k \geq 0$ implying that the condition in Proposition \ref{prop:stoch_alg_recursions} holds. Also, the condition $\Gamma \geq \max\bigg\{ \sqrt[b]{\tfrac{\lambda L_f}{\fy{L_F}}}, \sqrt[a]{4\gamma \fy{L_F}\zeta^{-1}\sqrt{1+\zeta}}\bigg\}$ together with $0 < \zeta \leq \tfrac{1-\rho_W^2}{4\rho_W^2}$ imply that $ \left((1+\zeta^{-1})\left(4L_k^2\gamma_k^2\right)+\zeta+1\right)\rho_W^2<1$ for all $k\geq 0$. \yq{\fy{T}herefore\fy{,} $\theta_6<1$.} \fy{We note that the condition $3a+b <2$ is used in obtaining the first inequality, where we bound $\tfrac{2M^2\Lambda_{k-1}^2}{\mu_f^3\lambda_k\gamma_k}$ by $\tfrac{2M^2 \gamma_k^2}{\mu_f^3\fy{\lambda}\fy{\gamma}^3\Gamma^{2-3a-b}}$, where we invoke $\Lambda_{k-1} \leq \frac{1}{k+\yq{\Gamma}}$ given in Lemma~\ref{lem:update_rules_props}.}
 
\noindent  (b) First,  we observe that each one of the scalars $\mathsf{n}_1$, $\mathsf{n}_2$, and $\mathsf{n}_3$ is proportional to the value of $\theta_3$.  We also observe that $C_i$ for $i\neq 2$ are independent of he value of $\theta_3$.  In view of these remarks, without loss of generality, let us assume that $\theta_3$ is sufficiently large such that $\mathsf{n}_i \geq 1$ for $i=1,2,3$ and that $\mathsf{e}_{3,0}\leq \Gamma^{2a}\tfrac{3\theta_9}{\theta_6}\mathsf{n}_3$. To show the three inequalities \fy{in part (b)}, we use induction on $k\geq 0$. For $k:=0$, from the definitions of the terms $\hat{\mathsf{e}}_i$, we \fy{may} write 
\begin{align*}
	&\mathsf{e}_{1,0}=\frac{\hat{\mathsf{e}}_1}{\Gamma^{a-b}\mathsf{n}_1} \leq \frac{\hat{\mathsf{e}}_1}{\Gamma^{a-b}}=\frac{\hat{\mathsf{e}}_1}{(0+\Gamma)^{a-b}},\quad
	 \mathsf{e}_{2,0}=\frac{\hat{\mathsf{e}}_2}{\Gamma^{2a}\mathsf{n}_2} \leq \frac{\hat{\mathsf{e}}_2}{\Gamma^{2a}}=\frac{\hat{\mathsf{e}}_2}{(0+\Gamma)^{2a}},\quad
	\mathsf{e}_{3,0} \leq \Gamma^{2a}\tfrac{3\theta_9}{\theta_6}\mathsf{n}_3 = \hat{\mathsf{e}}_3.
\end{align*}
These imply that the hypothesis statement holds for $k:=0$. Suppose the three inequalities in \eqref{eqn:eik} hold for some $k\geq 0$. We show that they hold for $k+1$ in three steps as follows. 

\noindent {\it Step 1:} From the definition of $\mathsf{n}_1$ we have $C_2 \leq (C_3\Gamma^{a-b}-2C_1C_4C_5)\mathsf{n}_1$. Rearranging the terms and invoking the definition of $\mathsf{n}_2$ we have
$C_1\Gamma^{2a}\mathsf{n}_2+C_2 \leq C_3\Gamma^{a-b}\mathsf{n}_1$.  Substituting the values of $C_1$, $C_2$, and $C_3$ we obtain
$
	\theta_2\mathsf{e}_{2,0}\Gamma^{2a}\mathsf{n}_2+\theta_3\gamma^2 \leq  \left(\gamma\lambda\theta_1-1\right)\mathsf{e}_{1,0}\Gamma^{a-b}\mathsf{n}_1.
$

Using the definition of $\hat{\mathsf{e}}_1$ and $\hat{\mathsf{e}}_2$ we have
$
	\theta_2\hat{\mathsf{e}}_{2} +\theta_3\gamma^2 \leq  \left(\gamma\lambda\theta_1-1\right)\hat{\mathsf{e}}_1
$

This implies that $\hat{\mathsf{e}}_1 \leq \gamma\lambda\theta_1\hat{\mathsf{e}}_1 -\theta_2\hat{\mathsf{e}}_{2} -\theta_3\gamma^2$. We have
$
\tfrac{\hat{\mathsf{e}}_1}{(k+\Gamma)^{2a}} \leq \tfrac{\gamma\lambda\theta_1\hat{\mathsf{e}}_1}{(k+\Gamma)^{2a}} -\tfrac{\theta_2\hat{\mathsf{e}}_{2}}{(k+\Gamma)^{2a}} -\tfrac{\theta_3\gamma^2}{(k+\Gamma)^{2a}}.
$

Note that from $\Gamma \geq 1$ we have that for all $k\geq 0$
$
\tfrac{1}{(k+\Gamma)^{2a}} = \tfrac{1}{(k+\Gamma)^{a-b}(k+\Gamma)^{a+b}}\geq \tfrac{1}{(k+\Gamma+1)^{a-b}(k+\Gamma)} 
\geq \tfrac{1}{(k+\Gamma+1)^{a-b}}((1+\tfrac{1}{k+\Gamma})^{a-b}-1) = \tfrac{1}{(k+\Gamma)^{a-b}}-\tfrac{1}{(k+\Gamma+1)^{a-b}}.
$

From the preceding relations we have 
\begin{align*}
\tfrac{\hat{\mathsf{e}}_1}{(k+\Gamma)^{a-b}}-\tfrac{\hat{\mathsf{e}}_1}{(k+\Gamma+1)^{a-b}}&\leq  \tfrac{\gamma\lambda\theta_1\hat{\mathsf{e}}_1}{(k+\Gamma)^{a+b}(k+\Gamma)^{a-b}} -\tfrac{\theta_2\hat{\mathsf{e}}_{2}}{(k+\Gamma)^{2a}} -\tfrac{\theta_3\gamma^2}{(k+\Gamma)^{2a}}
 \leq \tfrac{\gamma_k\lambda_k\theta_1\hat{\mathsf{e}}_1}{(k+\Gamma)^{a-b}} -\theta_2 \hat{\mathsf{e}}_{2,k}  - \theta_3\gamma_k^2,
\end{align*}
where we used the hypothesis statement and the update rules of $\gamma_k$ and $\lambda_k$. 

We obtain 
$
	(1-\theta_1\gamma_k\lambda_k)\tfrac{ \hat{\mathsf{e}}_1}{(k+\Gamma)^{a-b}} + \theta_2\mathsf{e}_{2,k}+\theta_3\gamma_k^2 \leq \tfrac{\hat{\mathsf{e}}_1}{(k+\Gamma+1)^{a-b}}. 
$

From $\Gamma\geq \sqrt[a+b]{\theta_1\gamma\lambda}$, we have $1-\theta_1\gamma_k\lambda_k\geq 0$.

Thus we have
$
	\mathsf{e}_{1,k+1} \leq (1-\theta_1\gamma_k\lambda_k)\mathsf{e}_{1,k} + \theta_2\mathsf{e}_{2,k}+\theta_3\gamma_k^2 \leq \tfrac{\hat{\mathsf{e}}_1}{(k+\Gamma+1)^{a-b}}. 
$
This shows that the first inequality in \eqref{eqn:eik} holds for $k+1$. 

\noindent {\it Step 2:} From the definition of $\mathsf{n}_1$ and $\mathsf{n}_2$ we have $C_4\mathsf{n}_3\leq \mathsf{n}_2$. We obtain $ (\tfrac{6\theta_9\theta_5\gamma^2}{\mathsf{e}_{2,0}\theta_4\theta_6})\mathsf{n}_3 \leq \mathsf{n}_2$. From \eqref{eqn:hat_e} we have 
$
	(\tfrac{6\theta_9\theta_5\gamma^2}{\mathsf{e}_{2,0}\theta_4\theta_6}) \tfrac{\theta_6\hat{\mathsf{e}}_{3}}{\Gamma^{2a}3\theta_9}  \leq  \tfrac{\hat{\mathsf{e}}_{2}}{ \Gamma^{2a}\mathsf{e}_{2,0}} \quad \Rightarrow  \quad  \hat{\mathsf{e}}_{3}  \leq \tfrac{\theta_4}{2\theta_5\gamma^2}\hat{\mathsf{e}}_{2} \quad 
\Rightarrow \quad \tfrac{\theta_4}{2}\hat{\mathsf{e}}_{2} \leq \theta_4\hat{\mathsf{e}}_{2} - \theta_5\gamma^2\hat{\mathsf{e}}_{3} .
$

From $\Gamma \geq \tfrac{1}{1-\sqrt[2a]{1-0.5\theta_4}}-1$ we have for all $k\geq 0$
\begin{align*}& \tfrac{\theta_4}{2} \geq 1-(1-\tfrac{1}{\Gamma+1})^{2a}\geq 1-(1-\tfrac{1}{k+\Gamma+1})^{2a}
\Rightarrow \ \tfrac{\theta_4}{2} \geq (k+\Gamma)^{2a}(\tfrac{1}{(k+\Gamma)^{2a}}-\tfrac{1}{(k+\Gamma+1)^{2a}}).
\end{align*}
From the preceding two relations we obtain
$
	(\tfrac{1}{(k+\Gamma)^{2a}}-\tfrac{1}{(k+\Gamma+1)^{2a}})\hat{\mathsf{e}}_{2} \leq \tfrac{\theta_4\hat{\mathsf{e}}_{2}}{(k+\Gamma)^{2a}} -\tfrac{ \theta_5\gamma^2\hat{\mathsf{e}}_{3}}{(k+\Gamma)^{2a}} .
$

This implies that $
(1-\theta_4)\mathsf{e}_{2,k} + \theta_5\gamma_k^2\mathsf{e}_{3,k} \leq \tfrac{\hat{\mathsf{e}}_{2}}{(k+\Gamma+1)^{2a}}$. Thus, the second inequality in \eqref{eqn:eik} holds for $k+1$.

	\noindent {\it Step \yq{3}:} From the definition of $\hat{\mathsf{e}}_{1}$, $\hat{\mathsf{e}}_{3}$, $\mathsf{n}_1$, $\mathsf{n}_3$, and $C_5$ we have 
	\begin{align}\label{eqn:third_step_1}
	\tfrac{\hat{\mathsf{e}}_{3}}{\hat{\mathsf{e}}_{1}}=\tfrac{\Gamma^{2a}\tfrac{3\theta_9}{\theta_6}\mathsf{n}_3}{ \Gamma^{a-b}\mathsf{e}_{1,0}\mathsf{n}_1} = \tfrac{\Gamma^{2a}\tfrac{3\theta_9}{\theta_6}C_5}{ \Gamma^{a-b}\mathsf{e}_{1,0}\Gamma^{2a}} = \tfrac{3\theta_7}{\theta_6\Gamma^{a-b}}.
	\end{align}
	From the definition of $\hat{\mathsf{e}}_{2}$, $\hat{\mathsf{e}}_{3}$, $\mathsf{n}_2$, $\mathsf{n}_3$, and $C_6$ we have 
	\begin{align}\label{eqn:third_step_2}
	\tfrac{\hat{\mathsf{e}}_{3}}{\hat{\mathsf{e}}_{2}}=\tfrac{\Gamma^{2a}\tfrac{3\theta_9}{\theta_6}\mathsf{n}_3}{ \Gamma^{2a}\mathsf{e}_{2,0}\mathsf{n}_2}  =\tfrac{\Gamma^{2a}\tfrac{3\theta_9}{\theta_6}}{ 2\Gamma^{2a}\mathsf{e}_{2,0}C_4}  =\tfrac{3\theta_8}{\theta_6(2C_4C_6)} \geq \tfrac{3\theta_8}{\theta_6\Gamma^{2a}},
	\end{align}
	where the last inequality is due to $\Gamma \geq \sqrt[2a]{2C_4C_6}$.  Also recalling $\mathsf{n}_3\geq 1$ and $\Gamma\geq 1$, from the definition of $\hat{\mathsf{e}}_{3}$ we have $ \hat{\mathsf{e}}_{3}  \geq \tfrac{3\theta_9}{\theta_6}$. From this relation, \eqref{eqn:third_step_1}, and \eqref{eqn:third_step_1} we have
$
	\tfrac{\theta_7}{\Gamma^{a-b}}\hat{\mathsf{e}}_{1} + \tfrac{\theta_8}{\Gamma^{2a}}\hat{\mathsf{e}}_{2}+\theta_9 \leq \theta_6\hat{\mathsf{e}}_{3}.
$

This implies that 
$
(1-\theta_6)\hat{\mathsf{e}}_{3}  +\theta_7\tfrac{\hat{\mathsf{e}}_{1}}{(k+\Gamma)^{a-b}} + \theta_8\tfrac{\hat{\mathsf{e}}_{2}}{(k+\Gamma)^{2a}}+\theta_9 \leq  \hat{\mathsf{e}}_{3}.
$

	From the hypothesis statement we conclude that $ \mathsf{e}_{3,k+1} \leq \hat{\mathsf{e}}_{3}$. Therefore, from the preceding three steps, the three inequalities in \eqref{eqn:eik} hold for $k+1$. Hence the proof is completed. 
	
\yq{\noindent (c) \fy{Consider the term $\EXP{\|\bar x_{k+1} - x^*\|^2}$. Adding and subtracting $x_{\lambda_k}^*$ inside the norm, we have}
\begin{align}\label{eqn:tm2_1}
\EXP{\|\bar x_{k+1} - x^*\|^2} &= \EXP{\|\bar x_{k+1} - x_{\lambda_k}^* +x_{\lambda_k}^* - x^*\|^2}\notag\\
&\leq 2\EXP{\|\bar x_{k+1} - x_{\lambda_k}^*\|^2} + 2\|x_{\lambda_k}^* - x^*\|^2
 \fy{= 2\mathsf{e}_{1,k+1}+ 2\|x_{\lambda_k}^* - x^*\|^2,}
\end{align}
\fy{where we note that $x_{\lambda_k}^* $ is a deterministic term. Invoking \eqref{eqn:eik}, we obtain 
\begin{align*}
\EXP{\|\bar x_{k+1} - x^*\|^2} & \leq \tfrac{2\hat{\mathsf{e}}_{1}}{(k+1+\Gamma)^{a-b}}+ 2\|x_{\lambda_k}^* - x^*\|^2.
\end{align*}
Taking limits on both sides when $k\to \infty$, and invoking $a>b$ and Lemma~\ref{lemma:IR-props}, we obtain $\lim_{k\to \infty}\mathbb{E}[\|\bar x_{k+1} - x^*\|^2] = 0$.
} }
\end{proof}
\begin{remark} 
\yq{Note that Theorem~\ref{thm:rate}\fy{(c) provides an asymptotic convergence guarantee for the averaged iterate generated by} Algorithm~\ref{algorithm:IR-DSGT} \fy{for addressing} problem \eqref{eqn:bilevel_problem_stoch}. \fy{Further, Theorem~\ref{thm:rate}(b) provides a non-asymptotic bound for the consensus error $\EXP{\|\mathbf{x}_k-\mathbf{1} \bar x_k\|^2}$, \yqt{which} is of the order $1/k^{2a}$. Notably, the given conditions on parameters $a$ and $b$ in Theorem~\ref{thm:rate} imply that $a \in (0,\tfrac{2}{3})$ and $b \in (0,\min\{0.5,a\})$. Consequently, with a sufficiently small $b$, $\EXP{\|\mathbf{x}_k-\mathbf{1} \bar x_k\|^2}$ may converge nearly as fast as $1/k^{4/3}$. However, note that a value of $b$ that is too small will cause $\lambda_k$ to diminish too slowly. Therefore, the convergence of the Tikhonov trajectory $x_{\lambda_k}$ to the unique solution $x^*$ will be too slow. This, in turn, will result in a slow global convergence of the term $\mathbb{E}[\|\bar x_{k+1} - x^*\|^2]$ to zero. These observations indeed suggest that in implementations, $b$ should be chosen within the range $(0,\yqt{\min}\{a,0.5\})$ but not too small. This trade-off will be numerically explored in section~\ref{sec:num}.}}
\end{remark}
	
\section{Numerical results}\label{sec:num}
\yq{We \fy{present preliminary} numerical experiments \fy{for computing} the optimal NE of \fy{a Cournot game considered in~\cite{tatarenko2020geometric}}. \fy{To this end, we} demonstrate the performance of \fy{Algorithms} \ref{algorithm:IR-push-pull} and \ref{algorithm:IR-DSGT} \fy{under different algorithm parameters, network sizes, and network topologies.} \fy{Consider a Cournot game among $m$ players such that each player $i \in [m]$ seeks to minimize the objective function $J_i(x_i,x_{-i})\triangleq  h_i(x_i)+l_i(x_{-i})x_i$}, where $h_i(x_i):=0.5a_ix_i^2+b_ix_i$ and $l_i(x_{-i}) := \sum_{j\neq i}c_{ij}x_j$, where $a_i \geq 0$. \fy{Each player $i$ is associated with a box constraint set $X_i:=[0,c_i^{\text{up}}]$, where $c_i^{\text{up}}>0$. Succinctly, the game is captured by a collection of $m$ optimization problems as $
\min_{x_i \in \mathbb{R}} \ h_i(x_i)+l_i(x_{-i})x_i +\mathbb{I}_{X_i}(x_i), $
for $i \in [m]$, where $\mathbb{I}_{X_i}(x)$ denotes the indicator function of the set $X_i$. To contend with the nonsmoothess, we consider a Moreau smoothed~\cite{qiu2023zeroth} variant of each player's optimization problem. Consider Moreau smoothing of the indicator function of a set nonempty, closed, and convex set $Y$. Recall that $\mathbb{I}_{Y}^{\eta}(z)=\tfrac{1}{2\eta}\mbox{dist}^2(z,Y)$~(cf. \cite{beck2017first}). Notably, $\mathbb{I}_{Y}^{\eta}(z)$ admits a gradient, defined as $\nabla_z \mathbb{I}_{Y}^{\eta}(z) = \tfrac{1}{\eta}\left(z-\Pi_{Y}(z) \right)$. Consequently, player $i$'s smoothed problem is
\begin{align}\label{eqn:cournot_smooth}
\min_{x_i \in \mathbb{R}} \ f_i(x) \triangleq h_i(x_i)+l_i(x_{-i})x_i +\tfrac{\text{dist}^2(x_i,X_i)}{2\eta}, \tag{P$_i(x_{-i})$}
\end{align} 
where $\eta>0$ is the smoothing parameter and $\Pi_{X_i}(\bullet)$ denotes the Euclidean projection onto the set $X_i$. Next, we reformulate the smoothed game \eqref{eqn:cournot_smooth} as a distributed VI problem. We let $x=[x_i]_{i=1}^m$ denote the tuple of players decisions. 
\begin{definition}\label{def:barC}
Let $\underbar{C} \in \mathbb{R}^{m\times m}$ be defined as $\underbar{C}_{ii}= 0.5a_i$ and $\underbar{C}_{ij}= c_{ij}$ for all $i\neq j$. Let us also define  $ \bar{C}\triangleq  \underbar{C}+ 0.5\,\text{diag}(\bar{a})$, $\bar{a}=[a_i]_{i=1}^m$, and $\bar{b}=[b_i]_{i=1}^m$.
\end{definition}
\begin{lemma}\label{lem:F_numerics}
Let us define $F(x) \triangleq \sum_{i=1}^m F_i(x)$, where $F_i(x)\triangleq   \left(a_i x_i +b_i + \sum_{j\neq i}c_{i j}x_j +\frac{1}{\eta}(x_i-\Pi_{X_i}(x_i))\right)\mathbf{e}_i$, where $\mathbf{e}_i \in \mathbb{R}^m$ denotes a unit vector where the $i$th element is $1$, and all other elements are $0$. Then, the following hold. 

\noindent (i) The set of all Nash equilibria to \eqref{eqn:cournot_smooth}, for $i\in [m]$, is equal to $  \mbox{SOL}(\mathbb{R}^m, F)$. 

\noindent (ii) If $\tfrac{1}{2}(\bar{C}+\bar{C}^\top)$ is positive semidefinite, then $F$ is monotone.

\end{lemma}
\begin{proof}
(i) To show this result we invoke~\cite[Prop. 1.4.2]{facchinei02finite}. First, note that $f_i$ is continuously differentiable and convex in $x_i$, for all $i \in [m]$. This is because the smoothness of $f_i$ follows from its definition, and by invoking the smoothness of the Moreau smoothed component $\tfrac{1}{2\eta}\mbox{dist}^2(x_i,X_i)$, as mentioned earlier. Also, the convexity of $f_i(\bullet,x_{-i})$, given any $x_{-i}$, follows from the convexity of the Moreau smoothed component~(cf. \cite{beck2017first}). To complete the proof, we need to show that $F(x) = [\nabla_{x_1}f_1(x);\ldots;\nabla_{x_m}f_m(x)]$. Note that we have $F_i(x) = \nabla_{x_i} f_i(x) \mathbf{e}_i$. Thus, we obtain $F(x) = \sum_{i=1}^m F_i(x) =\sum_{i=1}^m \nabla_{x_i} f_i(x) \mathbf{e}_i = [\nabla_{x_i}f_i(x)]_{i=1}^m$. Hence, the statement in part (i) follows from~\cite[Prop. 1.4.2]{facchinei02finite}.

\noindent (ii) From the proof in part (ii) and the definition of $\bar{C}$, we have $F(x) = [\nabla_{x_i}f_i(x)]_{i=1}^m = \bar{C}x +\bar{b}+\tfrac{1}{\eta}[x_i-\Pi_{X_i}(x_i)]_{i=1}^m$. Thus, for any $x,y \in \mathbb{R}^{m}$, we may write 
\begin{align*}
&(F(x)-F(y))^\top(x-y) = (x-y)^\top\bar{C}(x-y) \\
&+\tfrac{1}{\eta}\textstyle\sum_{i=1}^m (x_i-y_i)^2 - \tfrac{1}{\eta}\textstyle\sum_{i=1}^m(\Pi_{X_i}(x_i)-\Pi_{X_i}(y_i))(x_i-y_i)\\
& \geq  \tfrac{1}{2}(x-y)^\top\left(\bar{C}+\bar{C}^\top\right)(x-y)\\
&+\tfrac{1}{\eta}\textstyle\sum_{i=1}^m (x_i-y_i)^2 - \tfrac{1}{\eta}\textstyle\sum_{i=1}^m|\Pi_{X_i}(x_i)-\Pi_{X_i}(y_i)||x_i-y_i|,
\end{align*}
where we used $(x-y)^\top\bar{C}(x-y) = (x-y)^\top\bar{C}^\top(x-y)$ and the Cauchy-Schwarz inequality. From the nonexpansivity of the Euclidean projection, we have $|\Pi_{X_i}(x_i)-\Pi_{X_i}(y_i)|\leq |x_i-y_i|. $
From the two preceding equations and $\tfrac{1}{2}(\bar{C}+\bar{C}^\top) \succeq \mathbf{0}_{m\times m}$, for any $x,y \in \mathbb{R}^{m}$ we have $(F(x)-F(y))^\top(x-y)  \geq 0.$ This implies that $F$ is a monotone mapping.

\end{proof}
\begin{remark}
Notably, Lemma~\ref{lem:F_numerics} provides us with a monotone VI problem, given as $\mbox{VI}(\mathbb{R}^m,F)$ whose solution set captures all the NEs that the game \eqref{eqn:cournot_smooth}, for $i\in[m]$, may admit. Importantly, this is a distributed VI problem, of the form $\mbox{VI}(\mathbb{R}^m,\textstyle\sum_{i=1}^m F_i)$ in that $F_i$ is known locally by player $i$. 
\end{remark}
To select among the NEs of the game~\eqref{eqn:cournot_smooth}, we consider a {\it utilitarian} approach, where players seek to find an NE that minimizes $f(x)\triangleq \sum_{i=1}^m f_i(x)$, where $f_i$ denotes the local loss function in \eqref{eqn:cournot_smooth}. A question is whether $f$ is a convex function. This question is addressed next. 
\begin{lemma}\label{lem:f_numerics}
Consider the global welfare loss function $f$. If $\tfrac{1}{2}(\underbar{C}+\underbar{C}^\top)$ is positive semidefinite, then $f$ is convex in $x$. 
 \end{lemma}
 \begin{proof}
 From the definition of $f$, we obtain 
\begin{align*}
f(x) &=  \textstyle\sum_{i=1}^m (\tfrac{a_ix_i^2}{2}+b_ix_i + \textstyle \sum_{j\neq i}c_{ij}x_jx_i +\tfrac{1}{2\eta}\mbox{dist}^2(x_i,X_i) )\\
& = x^\top\underbar{C}x +\bar{b}x+\tfrac{1}{2\eta}\textstyle\sum_{i=1}^m\mbox{dist}^2(x_i,X_i)\\
& = \tfrac{1}{2}x^\top(\underbar{C}+\underbar{C}^\top)x +\bar{b}x+\tfrac{1}{2\eta}\textstyle\sum_{i=1}^m\mbox{dist}^2(x_i,X_i).
\end{align*}
Notably, $\mbox{dist}^2(x_i,X_i)$ is convex in $x_i$, for all $i\in[m]$. Thus, $\tfrac{1}{2\eta}\textstyle\sum_{i=1}^m\mbox{dist}^2(x_i,X_i)$ is convex in $x$. The convexity of $f$ follows from the assumption that  $\tfrac{1}{2}(\underbar{C}+\underbar{C}^\top) \succeq \mathbf{0}_{m\times m}$.
\end{proof}
\noindent {\bf Parameter settings.} To ensure that $F$ is monotone, for each setting of the implementations, we randomly generate a rank-deficient positive semidefinite matrix $C$, where we set $a_i := C_{ii}$ for all $i \in [m]$, and $\bar{C}_{ij}:= C_{ij}$ for $j\neq i$. In terms of the convexity of $f$, we note from Lemmas~\ref{lem:F_numerics} and \ref{lem:f_numerics} that the sufficient condition to guarantee convexity of the global function $f$ is different than that to guarantee monotonicity of the mapping $F$. Indeed, even if $\tfrac{1}{2}(\bar{C}+\bar{C}^\top)\succeq \mathbf{0}$, it is not necessarily guaranteed to have $\tfrac{1}{2}(\underbar{C}+\underbar{C}^\top)\succeq \mathbf{0}$. To validate the theoretical results, we employ a regularization for the welfare loss function $f$ as $\sum_{i=1}^m f_i(x) +\tfrac{\theta}{2}\|x\|_2^2$, where $\theta > 0$, where $\theta:=10^{-5}+\max\{0,-\lambda_{\min}\}$ such that $\lambda_{\min}$ is the minimum eigenvalue of $\tfrac{1}{2}(\underbar{C}+\underbar{C}^\top)$, where following Definition~\ref{def:barC}, we set $\underbar{C}:= \bar{C} - 0.5\text{diag}(\bar{a})$. Thus, in view of Lemma~\ref{lem:F_numerics}, the regularized global function is strongly convex and the optimal NE is unique. Throughout, we use $\eta:=0.1$ and generate each of the parameters $c_i^{\text{up}}$ for the box constraints, for $i \in [m]$, uniformly at random from the interval $[50, 100]$.
}
    
\subsection{\fy{Optimal NE seeking over} directed network\fy{s}}
\noindent \fy{We generate parameters $b_i$ for $i \in [m]$ normally at random with a mean of zero and a variance of $10$.} We test the performance of Algorithm~\ref{algorithm:IR-push-pull} on two \yq{\fy{network} settings: (i)} a directed star graph with $\fy{m} = 10$ \fy{nodes}; (ii) a random \fy{digraph} with $\fy{m} = 100$ \fy{nodes. In (ii), this is done by generating a random tree and then adding} edges at random until we reach the desired number of edges, \fy{that is} $\lfloor 100\ln 100 \rfloor = 460$ edges \fy{in our setting}. We then use a modification of the \textit{max-degree weights} heuristic in \cite{boydGraphs} by setting $\alpha := 1 / (2 d_{max})$. \fy{We also ensure that $\mathcal{R}_{\mathbf{R}}\cap \mathcal{R}_{\mathbf{C}^{\top}}\neq \emptyset$.} For each setting, we compare the effect of three choices of $(a,b) \in \left\{(0.5, 0.3), (0.6, 0.25), (0.675, 0.2)\right\}$ satisfying Assumption~\ref{assum:update_rules}. \yq{We use $\| F(\bar{\mathbf{x}}_k)\|_2$, $\| \bar{\mathbf{x}}_{k+1} - \bar{\mathbf{x}}_k \|_2$ and $\| \mathbf{x}_k - \mathbf{1}\bar{\mathbf{x}}_k \|$ as metric of the lower-level error, the upper-level error, and the consensus error, respectively. The implementation results are shown in \fy{Figure}~\ref{fig:comparison1}.}

\noindent  {\bf Insights.} We observe \fy{in both the star and random graph settings} that \fy{all the three metrics appear to be converging. The performance} of the lower-level error metric for Algorithm~\ref{algorithm:IR-push-pull} is best for $(a,b) = (0.5, 0.3)$ and worst for $(a,b) = (0.675, 0.2)$. \fy{This is reversed for the upper-level error metric. These observations appear to be reasonable and are consistent with our theoretical findings. This is mainly} because a larger $b$ \fy{implies} that the regularization parameter $\lambda_k = \frac{\lambda}{(k+\Gamma)^b}$ decreases faster. \fy{As a result, the information of $\nabla f_i$ is multiplied by smaller values, emphasizing less on the upper-level local objectives and more on the information of lower-level mappings $F_i$.} 

}
\begin{table*}
\setlength{\tabcolsep}{0pt}
\centering
\begin{tabular}{c || c  c  c  c}
  {\footnotesize {}\ \ }& {\footnotesize lower-level error metric} & {\footnotesize upper-level error metric} & {\footnotesize ln(consensus error)} & {\footnotesize push and pull networks}\\
    \hline\\
    \rotatebox[origin=c]{90}{{\footnotesize {Star graph}}}
    &
    \begin{minipage}{.22\textwidth}
        \includegraphics[scale=.25, angle=0]{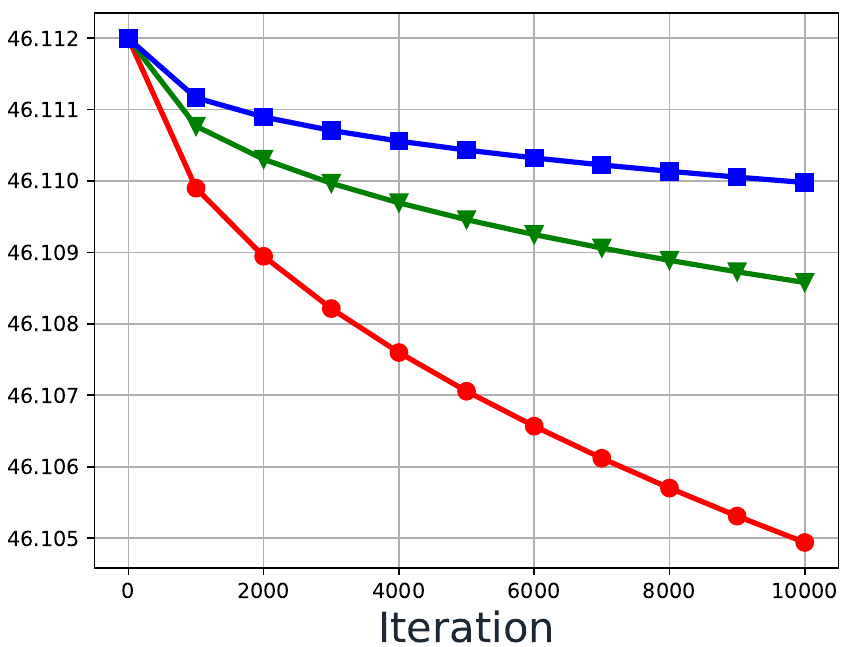}
    \end{minipage}
    &
    \begin{minipage}{.22\textwidth}
        \includegraphics[scale=.25, angle=0]{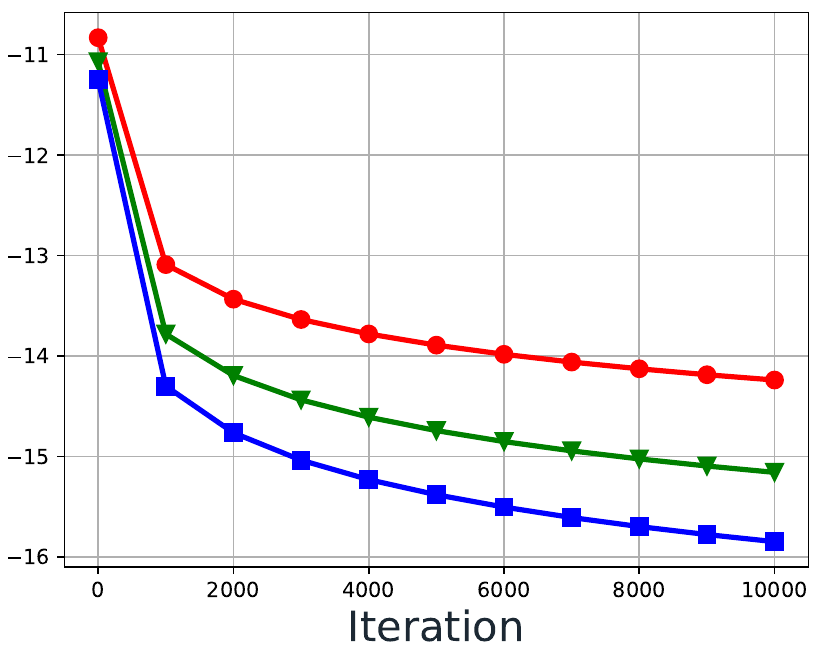}
    \end{minipage}
    &
    \begin{minipage}{.22\textwidth}
        \includegraphics[scale=.25, angle=0]{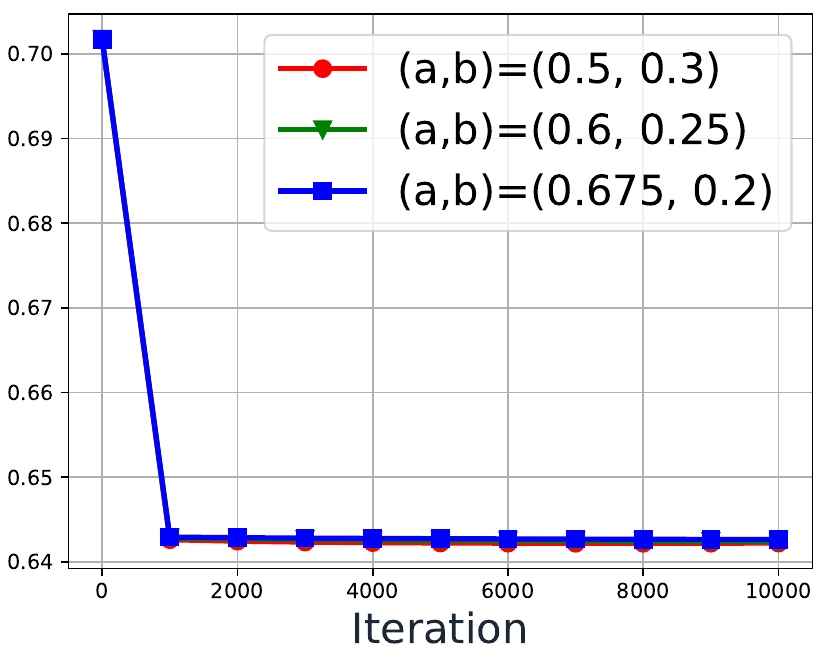}
    \end{minipage}
    &
    \begin{minipage}{.22\textwidth}
        \includegraphics[scale=.18, angle=0]{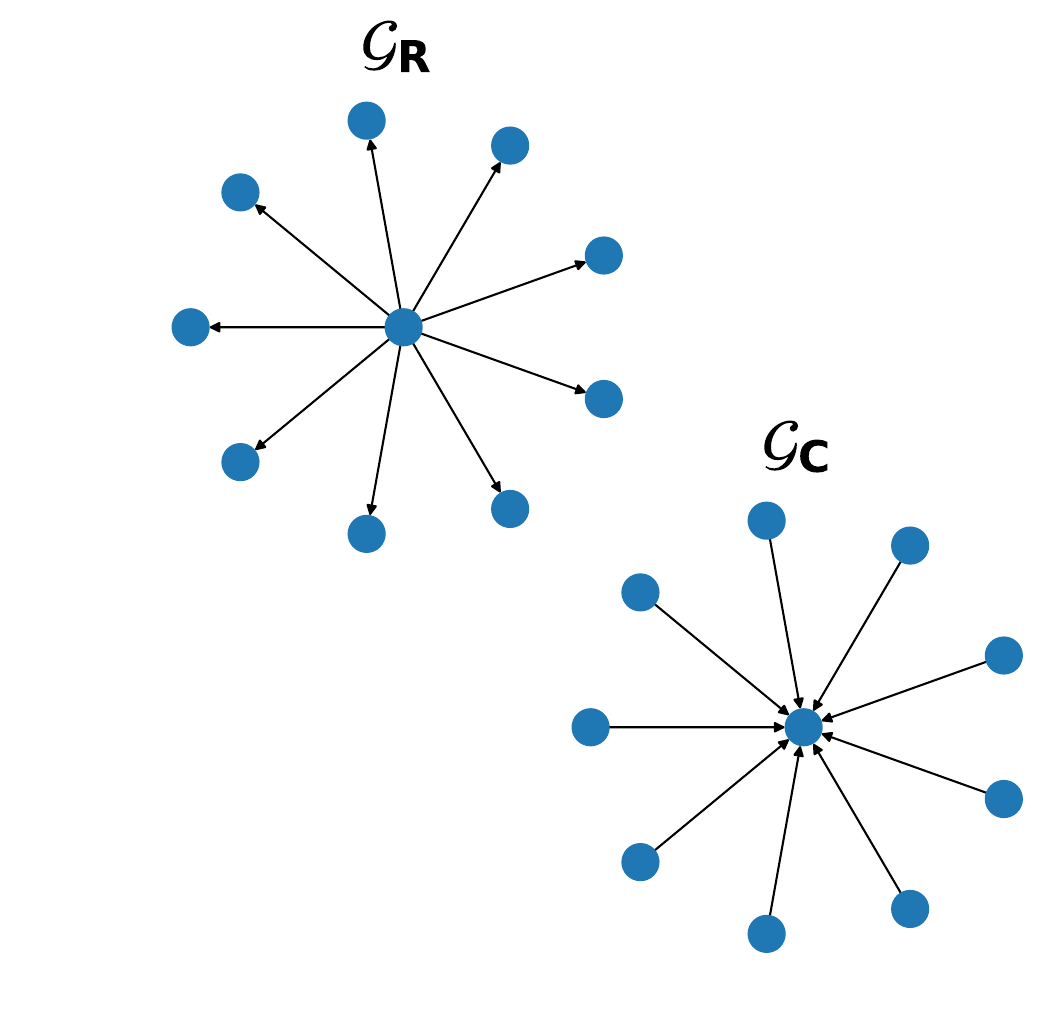}
    \end{minipage}
    \\
    
    \hbox{}& & & \\
    \hline\\
    \rotatebox[origin=c]{90}{{\footnotesize Random graph}}
    &
    \begin{minipage}{.22\textwidth}
        \includegraphics[scale=.25, angle=0]{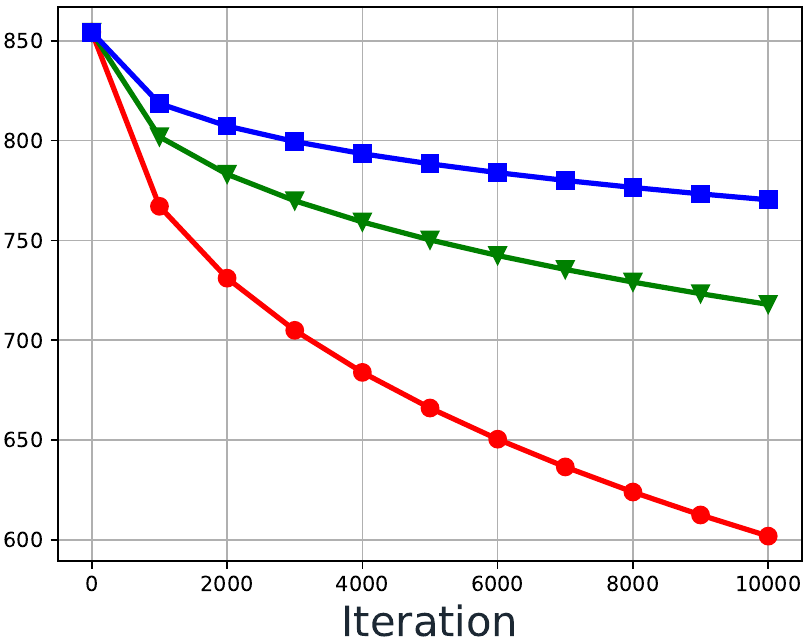}
    \end{minipage}
    &
    \begin{minipage}{.22\textwidth}
        \includegraphics[scale=.25, angle=0]{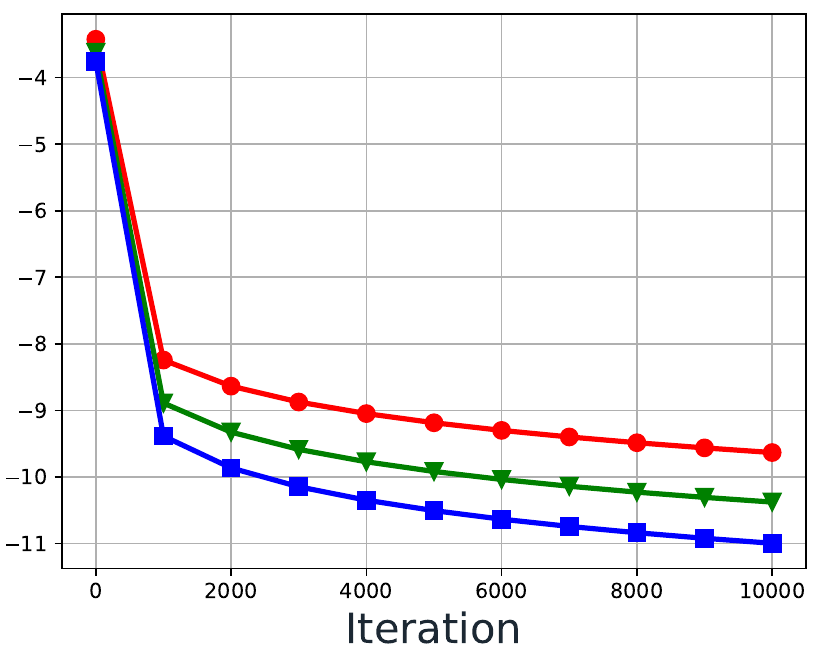}
    \end{minipage}
    &
    \begin{minipage}{.22\textwidth}
        \includegraphics[scale=.25, angle=0]{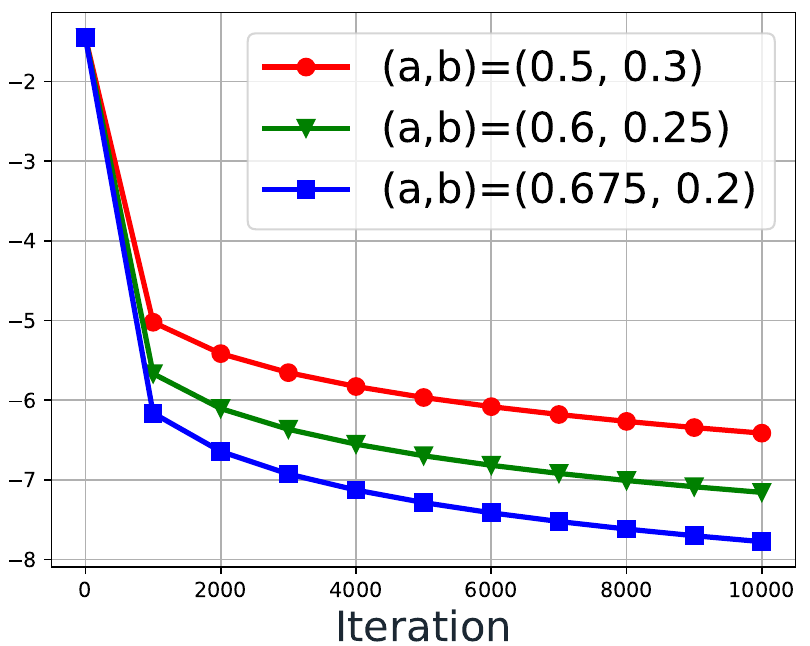}
    \end{minipage}
    &
    \begin{minipage}{.22\textwidth}
        \includegraphics[scale=.2, angle=0]{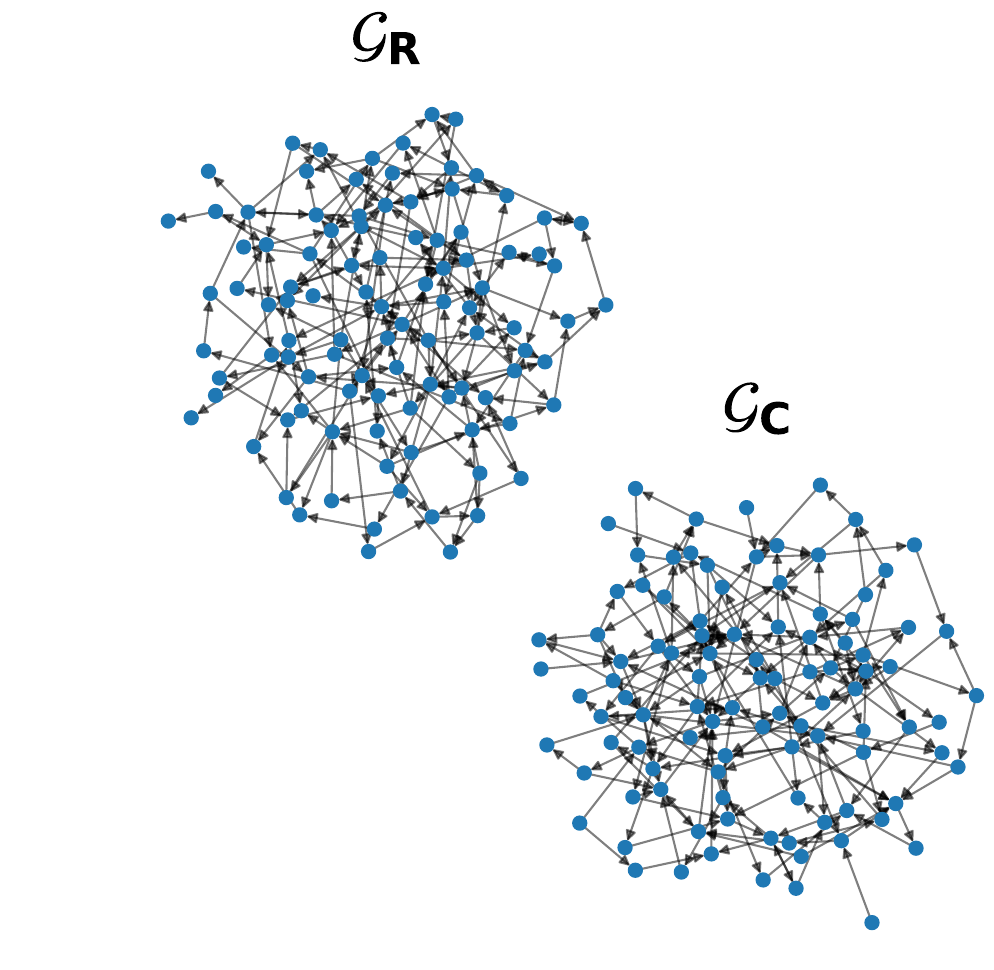}
    \end{minipage}
\end{tabular}
\captionof{figure}{Performance of Algorithm \ref{algorithm:IR-push-pull} on star graph and random graph.}
\label{fig:comparison1}
\vspace{-.1in}
\end{table*}
\yq{
\subsection{\fy{Stochastic setting over} undirected network\fy{s}}
\fy{We implement Algorithm~\ref{algorithm:IR-DSGT} for solving problem~\eqref{eqn:bilevel_problem_stoch}, where we consider a stochastic setting of \eqref{eqn:cournot_smooth} as follows. Let us define $f_i(x,\xi_i)=\tfrac{1}{2}a_ix_i^2 +b_i(\xi_i)x_i + (\sum\nolimits_{j\neq i}c_{ij}x_j)x_i+\tfrac{1}{2\eta}\|x_i-\Pi_{X_i}(x_i)\|^2$ and $F_i(x,\xi_i) = \left(a_i x_i +b_i(\xi_i) + \sum_{j\neq i}c_{i j}x_j +\tfrac{1}{\eta}(x_i-\Pi_{X_i}(x_i))\right)\mathbf{e}_i$. Here, $b_i(\xi_i)$ is sampled iteratively in the method by each player, uniformly at random from $[1,10]$, for all players. For the communication network among the players, we consider the Petersen graph with $10$ nodes and an undirected random graph with $100$ nodes. We compare the effect of $a$ and $b$ with three choices: $(a,b) \in \left\{ (0.5, 0.4), (0.55, 0.3), (0.6, 0.175) \right\}$, satisfying the conditions in Theorem~\ref{thm:rate}. We run each experiment for $10$ sample paths and report the mean of the errors. The results are shown in Figure~\ref{fig:comparison2}.

 \noindent  {\bf Insights.} In both the network settings, we again observe that all the three metrics appear to be converging in a mean sense. Further, in a similar fashion to what we observed in the previous experiment, we again see that the setting with the largest value of $b$ performs the best in terms of the lower-level metric, while the setting with the smallest value of $b$ performs the best in terms of the upper-level metric.}
}
\begin{table*}
\setlength{\tabcolsep}{0pt}
\centering{
\begin{tabular}{c || c  c  c  c}
  {\footnotesize {}\ \ }& {\footnotesize lower-level error metric} & {\footnotesize upper-level error metric} & {\footnotesize ln(consensus error)} & {\footnotesize network}\\
    \hline\\
    \rotatebox[origin=c]{90}{{\footnotesize {Petersen graph}}}
    &
    \begin{minipage}{.22\textwidth}
        \includegraphics[scale=.25, angle=0]{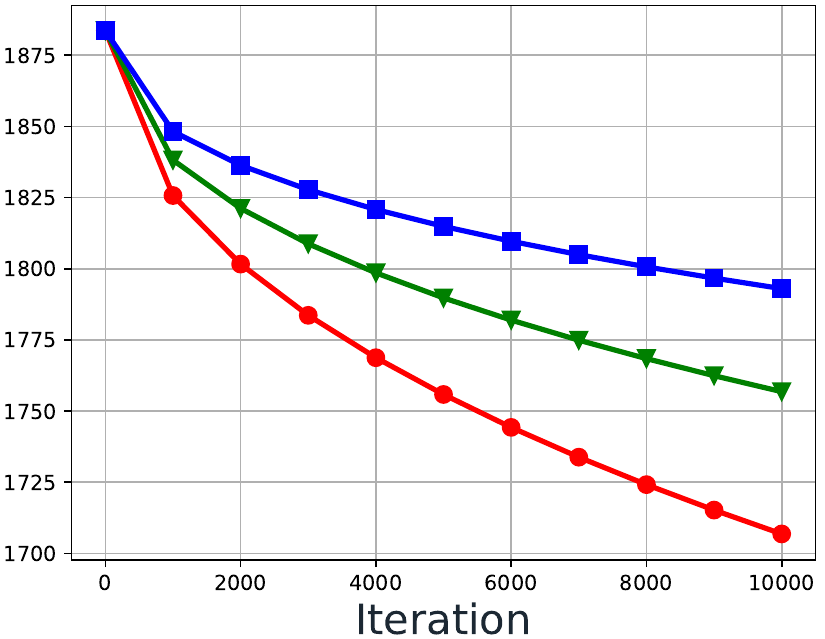}
    \end{minipage}
    &
    \begin{minipage}{.22\textwidth}
        \includegraphics[scale=.25, angle=0]{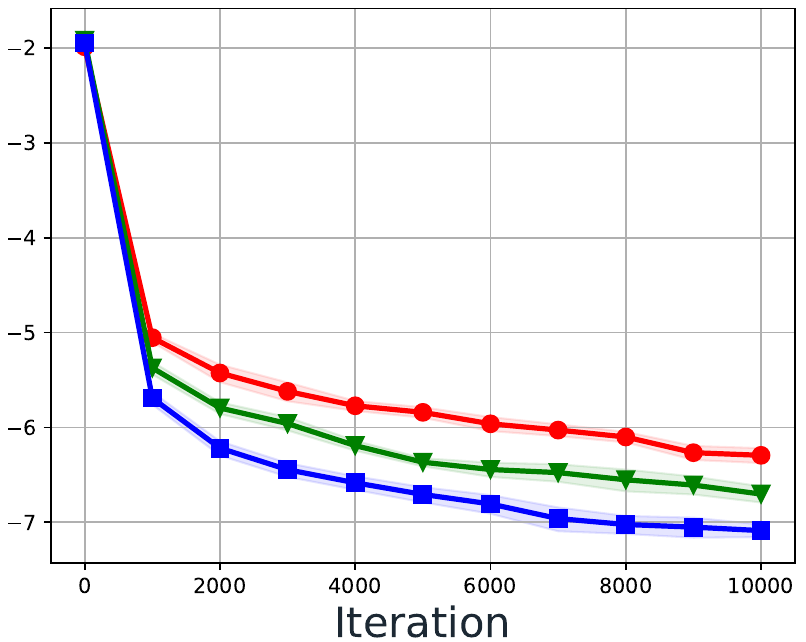}
    \end{minipage}
    &
    \begin{minipage}{.22\textwidth}
        \includegraphics[scale=.25, angle=0]{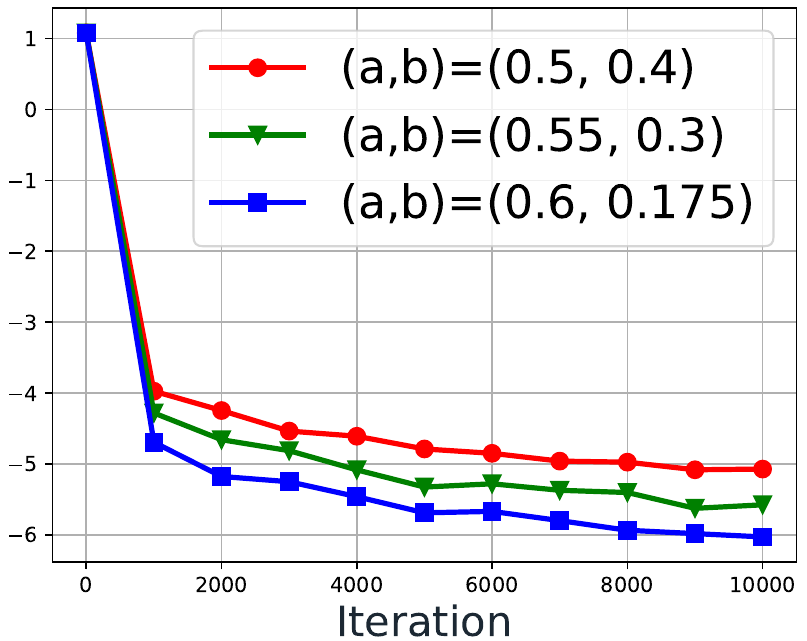}
    \end{minipage}
    &
    \begin{minipage}{.22\textwidth}
        \includegraphics[scale=.22, angle=0]{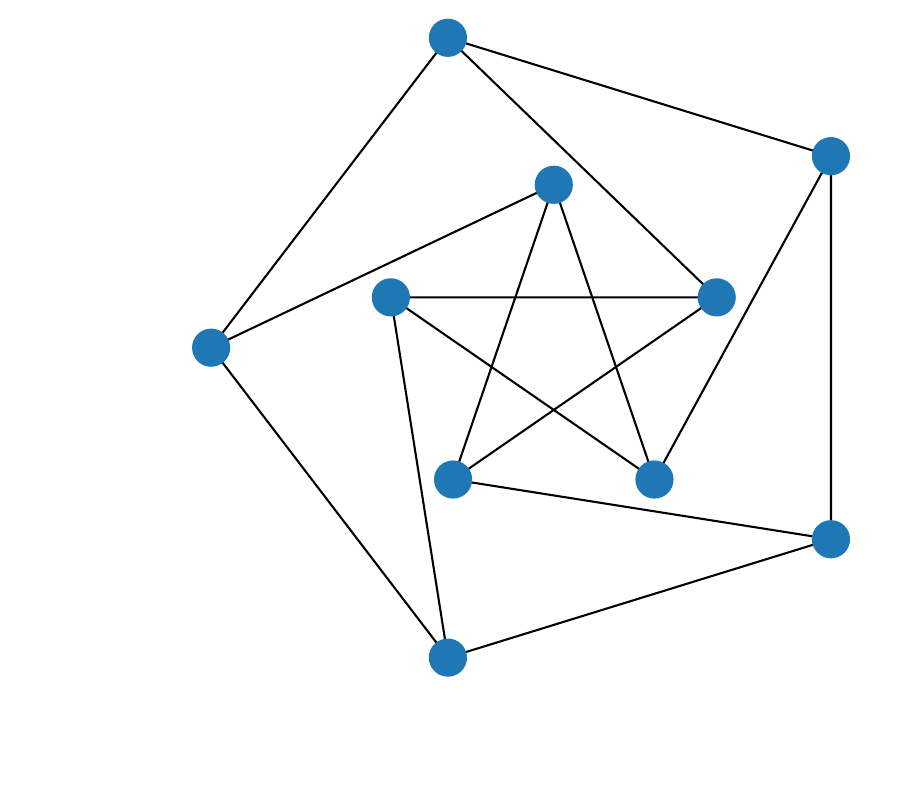}
    \end{minipage}
    \\
    
    \hbox{}& & & \\
    \hline\\
    \rotatebox[origin=c]{90}{{\footnotesize Random graph}}
    &
    \begin{minipage}{.22\textwidth}
        \includegraphics[scale=.25, angle=0]{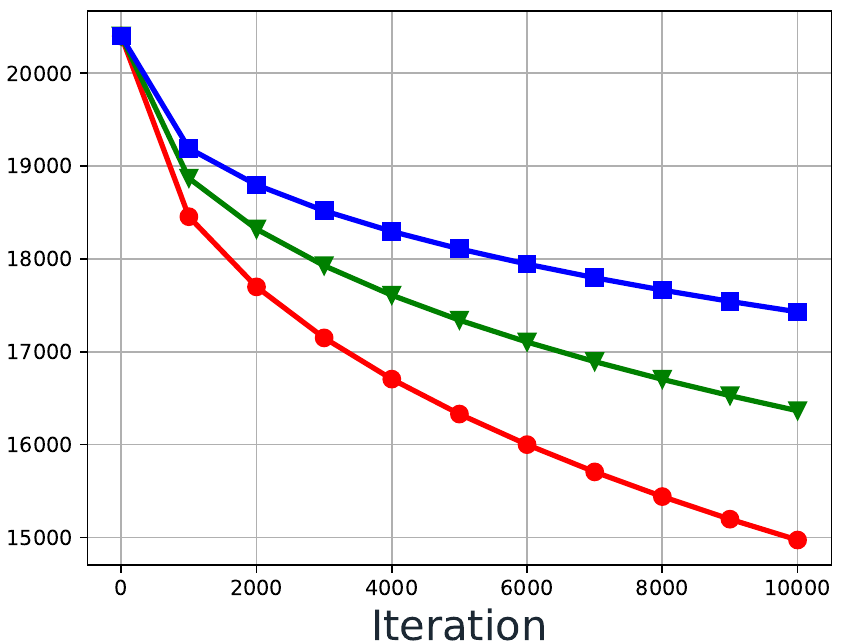}
    \end{minipage}
    &
    \begin{minipage}{.22\textwidth}
        \includegraphics[scale=.25, angle=0]{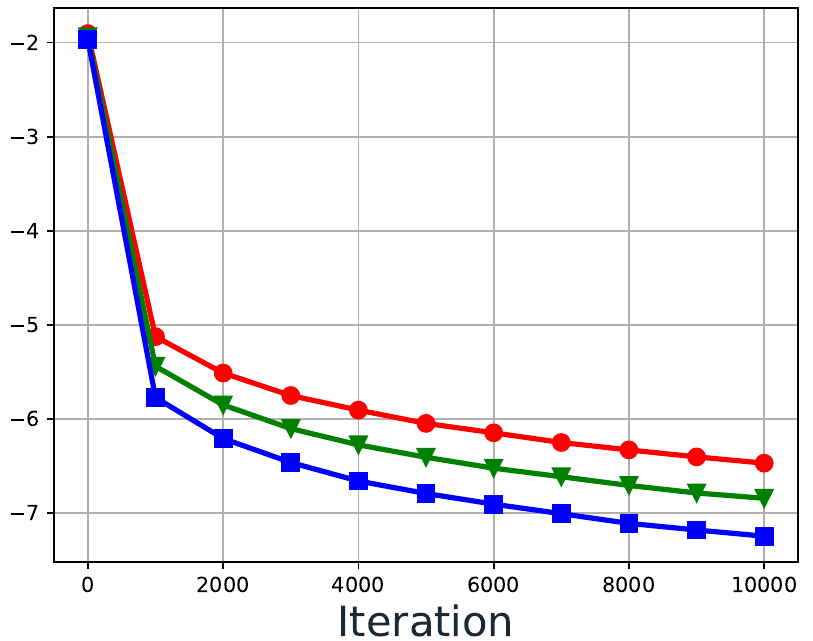}
    \end{minipage}
    &
    \begin{minipage}{.22\textwidth}
        \includegraphics[scale=.25, angle=0]{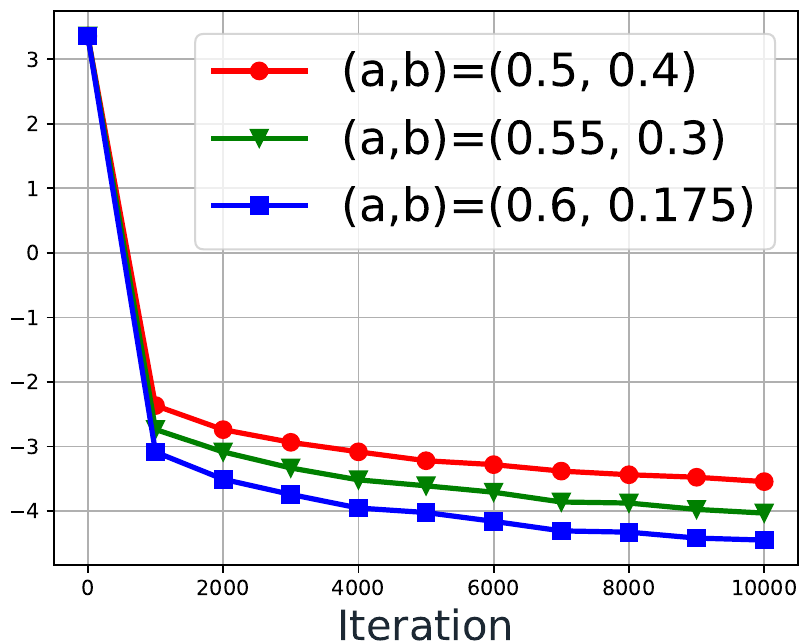}
    \end{minipage}
    &
    \begin{minipage}{.22\textwidth}
        \includegraphics[scale=.25, angle=0]{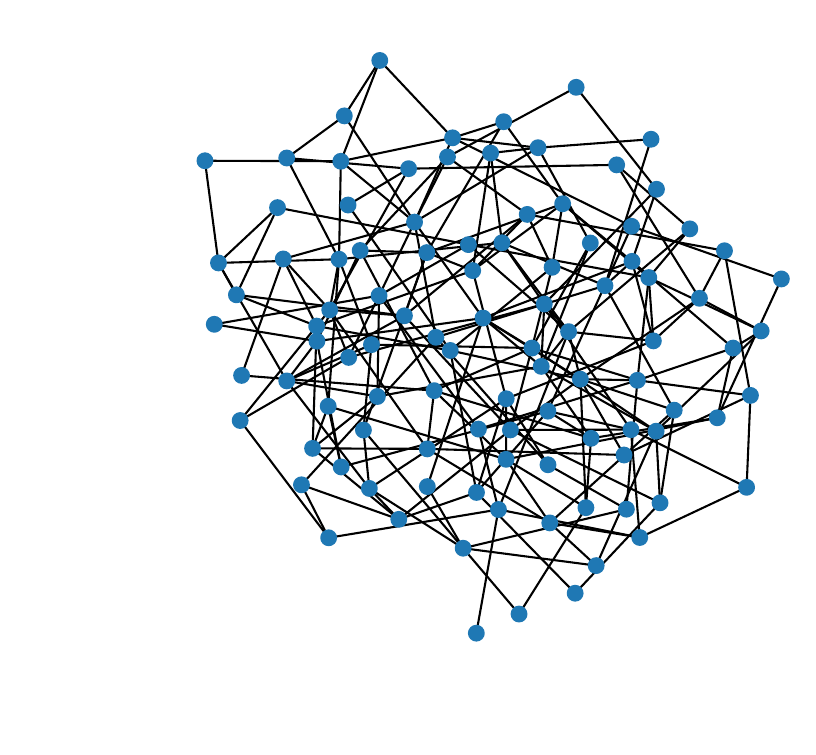}
    \end{minipage}
\end{tabular}}
\captionof{figure}{The performance of Algorithm \ref{algorithm:IR-DSGT} on Petersen graph and random graph.}
\label{fig:comparison2}
\vspace{-.1in}
\end{table*}

\section{Concluding remarks}\label{sec:conc}
\fy{Traditional approaches to addressing optimal equilibrium-seeking problems are computationally inefficient as they are often executed through two-loop schemes and lack provable guarantees. In this \fy{paper}, by leveraging gradient tracking and iterative regularization, we develop two single-timescale methods tailored for optimal equilibrium selection problems. These methods include the Iteratively Regularized Push-Pull (IR-Push-Pull) for directed networks and the Iteratively Regularized Distributed Stochastic Gradient Tracking (IR-DSGT) for undirected networks. Under some standard assumptions, we establish the global convergence to the unique optimal equilibrium and derive provable consensus guarantees. Preliminary numerical experiments on a Cournot game validate our theoretical findings and demonstrate robustness across two network settings with different node sizes.}

\appendices

\bibliographystyle{siam}
\bibliography{references}






\end{document}